\documentclass{amsart}

\usepackage{amsmath}
\usepackage{amssymb}
\usepackage{amsthm}
\usepackage{graphicx}
\usepackage{amscd}
\usepackage{tikz}
\usepackage{tikz-cd}
\usepackage{hyperref} 

\hypersetup{
    colorlinks=true, 
    linkcolor=blue,  
    citecolor=red,  
    urlcolor=red    
}

\newtheorem{thm}{Theorem}[section]
\newtheorem{prop}[thm]{Proposition}
\newtheorem{lem}[thm]{Lemma}
\newtheorem{cor}[thm]{Corollary}

\newtheorem{convention}[thm]{Convention}

\theoremstyle{definition}
\newtheorem{dfn}[thm]{Definition}

\theoremstyle{remark}
\newtheorem{rem}[thm]{Remark}
\newtheorem*{acknowledgment}{Acknowledgements}

\newcommand{\thmref}[1]{%
    \hyperref[#1]{\ref*{#1}}
}

\newcommand{\C}{\mathbb{C}}
\newcommand{\N}{\mathbb{N}}
\newcommand{\R}{\mathbb{R}}

\title[Chern Character and Fermi point]
{Chern Character and Fermi point}

\author[K. Horie]{Kyouhei Horie}

\address{
 DEPARTMENT OF MATHEMATICS, INSTITUTE OF SCIENCE TOKYO, 
2-12-1 Ookayama, Meguro-ku, Tokyo, 152-8551, Japan.}

\email{horie.k.1e2d@m.isct.ac.jp}

\date{}

\begin{document}

\begin{abstract}
This paper expresses the Chern character for topological $K$-theory based on the formulation of the family of Fredholm operators, by using the points at which the Fredholm operator becomes singular (Fermi points). In particular, we explain that the odd Chern character can be thought of as a generalization of the spectral flow. As applications, we give elementary proofs of the evenness of the edge index and the bulk-edge correspondence for four-dimensional topological insulators with time-reversal symmetry of class AI.
\end{abstract}

\maketitle

\tableofcontents


\section{Introduction}
\label{sec:introduction}

\subsection{Topological $K$-theory and spectral flow}

Topological $K$-theory, which was introduced by Atiyah and Hirzebruch, is constructed by isomorphism classes of complex vector bundles \cite{A1}. 
The construction by using complex vectors is a formulation within the scope of finite-dimensional linear algebra, but there is also a construction via infinite dimensions. That is, the formulation of topological $K$-theory by using families of Fredholm operators acting on infinite-dimensional Hilbert spaces with a symmetry of Clifford algebras \cite{A-Si}.
The advantage of the Fredholm operator formulation is that it is suitable for twisted $K$-theory, which is a variant of topological $K$-theory.
Moreover, $K$-theory by using Fredholm operators appears naturally in a description of  the edge states of topological insulators.

\medskip

Since the Fredholm operator formulation involves infinite dimensions, it is generally difficult to directly calculate the $K$-group. However, there exists a feasible invariant called the spectral flow. The spectral flow becomes a complete invariant of the odd $K$-group of the circle $S^1$ \cite{Atiyah76}.
 That is
\begin{align*}
\text{sf}\colon K^1(S^1)\cong \mathbb{Z}.
\end{align*}
Roughly speaking, the spectral flow is a signed count of the number of times that the eigenvalues of a family of Fredholm operators on the circle $S^1$ change from positive to negative as the parameter $k$ moves around $S^1$. $(A$ rigours
 definition can be found in \cite{Phillips}.) Equivalently, one may interpret the spectral flow as the signed count of points on $S^1$ where the eigenvalues of this family transversely cross zero. In this paper, we generalize the condition of transverse crossing by introducing a concept of a \emph{sign coordinate}. This notion allows us to assign a consistent sign to a point where the spectrum intersects zero, even in cases where the crossing is not necessarily transverse. By using this concept, we develop an integer invariant which generalizes the spectral flow. Based on this generalized spectral flow, we define the set $F$ as the collection of points in a topological space $X$ where a family of Fredholm operators $A=\{A_x\}_{x \in X}$ parameterized by $X$ becomes singular, that is,
\begin{align*}
F = \{ x \in X \mid 0 \in \text{Spec}(A_x) \},
\end{align*}
where $\text{Spec}(A_x)$ denotes the spectrum of $A_x$. We refer to a point $x \in F$ at which a sign coordinate can be defined as a \emph{Fermi point}. With each Fermi point $x \in F$, we associate a sign, denoted as $\operatorname{sign}(x) \in \{\pm 1\}$, which corresponds to the contribution of $x$ to the generalized spectral flow.

We note that the notion of a Fermi point used in this paper differs from its conventional usage in condensed matter physics. Typically, the Fermi set refers to the set of momenta with a prescribed energy level, and isolated points in this set are referred to as Fermi points. In contrast, in this paper, we use the term Fermi point to refer specifically to those isolated points in the set $F$ at which a sign coordinate can be defined. Thus, the term is used in a way that deviates from its standard meaning in physics.

\subsection{Main result}

We present the main theorems of this paper, which express the Chern character in terms of Fermi points. Here, \(\mathcal{F}_0(\hat{\mathcal{H}})\) denotes the space of self-adjoint Fredholm operators of degree 1 acting on a \(\mathbb{Z}_2\)-graded Hilbert space \(\hat{\mathcal{H}}\), while \(Fred_{sa}\) denotes the space of self-adjoint Fredholm operators on an ungraded Hilbert space \(\mathcal{H}\), and \(Fred_{sa}^1\) denotes the non-contractible path-connected component of \(Fred_{sa}\). The Chern characters \(\operatorname{ch}_k\) and \(\operatorname{ch}_{k+\frac{1}{2}}\) are invariants associated with the $K$-theory classes in even and odd degrees, respectively. Precise definitions of these terms will be provided in later sections.

\begin{thm}[The Main Theorem for Even Degrees]\label{thm:even main theorem}
  Let $k$ be a non-negative integer. Let $X$ be a compact, oriented, $2k$-dimensional differentiable manifold, and \,$\hat{A}\colon X\to \mathcal{F}_{0}(\mathcal{\hat{H}})$ a continuous map. Then, the following holds.
\begin{enumerate}
\item When each point $x \in F$ is a Fermi point, $\sum_{x \in F}\operatorname{sign}(x)$ depends only on $[\hat{A}] \in K^0(X)$.
\item Under the assumption of $(1)$, the following equality holds.
\begin{align*}
\int_X \operatorname{ch}_{k}([\hat{A}]) = \sum_{x \in F} \operatorname{sign}(x) \in \mathbb{Z}.
\end{align*}
\end{enumerate} 
\end{thm}

\begin{thm}[The Main Theorem for Odd Degrees]\label{thm:odd main theorem}
 Let $k$ be a non-negative integer. Let $X$ be a compact, oriented, $2k+1$-dimensional differentiable manifold, and \,$A\colon X\to Fred_{sa}^1(\mathcal{H})$ a continuous map. Then, the following holds.
\begin{enumerate}
\item When each point $x \in F$ is a Fermi point, $\sum_{x \in F}\operatorname{sign}(x)$ depends only on $[A] \in K^1(X)$.
\item Under the assumption of $(1)$, the following equality holds.
\begin{align*}
\int_X \operatorname{ch}_{k+\frac{1}{2}}([A]) = \sum_{x \in F} \operatorname{sign}(x) \in \mathbb{Z}.
\end{align*}
\end{enumerate}  
\end{thm}

In particular, when considering the case of odd degrees with \(X = S^1\), the odd Chern character \(\operatorname{ch}_{\frac{1}{2}}([A])\) integrated over \(S^1\) coincides with the spectral flow.

\medskip

\subsection{Applications to Topological Insulators}

Topological insulators are quantum systems characterized by insulating bulk states and conducting edge states. A key feature of these systems is the \emph{bulk-edge correspondence}, which relates topological invariants of the bulk to those of the edge, a phenomenon that has been formulated in various settings (e.g., \cite{Avila}, \cite{Avron}, \cite{Bellissard}, \cite{Braverman}, \cite{Guillemin}, \cite{Hatsugai}, \cite{Kellendonk00}).

In this paper, we focus on four-dimensional topological insulators with time-reversal symmetry of class AI, corresponding to spinless or integer-spin quantum systems invariant under time reversal (class AI). These systems are described by a bulk Hamiltonian $\hat{H} \colon T^4 \to \operatorname{Herm}(\mathbb{C}^r)^*$, where $T^4$ is the four-dimensional torus representing the wave number space, and $\operatorname{Herm}(\mathbb{C}^r)^*$ denotes the set of invertible $r \times r$ Hermitian matrices. The edge Hamiltonian $\hat{H}^{\#} = {\{\hat{H}^{\#}(k)\}}_{k \in T^3}$, parameterized by the three-dimensional torus $T^3$, is constructed from $\hat{H}$ as a family of Toeplitz operators \cite{Douglas}.

For four-dimensional class AI topological insulators, we define the bulk index as the second Chern number $c_2(\hat{H})$ and the edge index as the odd Chern character $\int_{T^3} \operatorname{ch}_{\frac{3}{2}}([\hat{H}^{\#}])$. Under class AI symmetry, the edge index is known to be an even integer, and the bulk-edge correspondence is formulated as $c_2(\hat{H}) = -\int_{T^3} \operatorname{ch}_{\frac{3}{2}}([\hat{H}^{\#}])$.

By applying our main theorems, which express the Chern character in terms of Fermi points, we obtained the following results for four-dimensional class AI topological insulators:
\begin{enumerate}
    \item Assuming class AI time-reversal symmetry and a stronger condition, the edge index satisfies $\int_{T^3} \operatorname{ch}_{\frac{3}{2}}([\hat{H}^{\#}]) \in 2\mathbb{Z}$. (Proposition \ref{prop:evenness})
    \item The bulk-edge correspondence holds, i.e., $c_2(\hat{H}) = -\int_{T^3} \operatorname{ch}_{\frac{3}{2}}([\hat{H}^{\#}])$, for insulators defined on the torus $T^4$ with class AI symmetry. (Theorem \ref{BG}.)
\end{enumerate}

The method developed in this paper is intended for surface Hamiltonians whose zero-energy set is sufficiently nondegenerate so that the zero modes are isolated and admit sign coordinates. In particular, the method applies to situations where the low-energy behavior near each zero-energy point is modeled by a finite-dimensional Dirac-type local model, as in Proposition 4.28. It is therefore not designed for surface spectra with extended zero-energy loci such as Fermi surfaces of positive dimension.

\subsection{Outline}

This paper is organized as follows: In \S \ref{sec:Clifford algebra}, we summarize the Clifford algebras and their representations necessary for the formulation of $K$-groups via Fredholm operators. In \S \ref{sec:$K$-theory}, we discuss the $K$-theory associated with families of Fredholm operators. In \S\S\ref{subsec:Definition of $K$-theory}, we first define the separable infinite-dimensional Hilbert space $\mathcal{H}_n$ endowed with a symmetry of the Clifford algebra $Cl_n$. Subsequently, we provide the definition of the $K$-group based on families of Fredholm operators acting on $\mathcal{H}_n$. Next, in \S\S \ref{subsec:Description of the generator}, we describe the suspension isomorphism of $K$-groups in the context of the Fredholm operator formulations, as developed by Atiyah and Singer \cite{A-Si}. As an application of the suspension isomorphism, we explicitly characterize the generators of $K^n(D^n, \partial D^n)$. Since detailed descriptions of the generators of $K^n(D^n, \partial D^n)$ in the Fredholm operator formulation are rarely found in the literature, we elaborate on this point. In \S\S \ref{subsec:Chern character and Push-forward}, we define the Chern character by using Fredholm operators and introduce the push-forward map along a point insertion. In \S\S\ref{sub:Approximation of family of Fredholm operators}, we explain the finite-dimensional approximation of Fredholm operators. By using this approximation, we define a sign coordinate and a Fermi point in \S\S \ref{sub:Sign coorinate}. In \S\S \ref{sub:Proof of the main theorem}, we provide the proof of the main theorem. In \S\S \ref{subsec:Specific calculation example}, we present specific computational examples derived from the main theorem. In \S\S \ref{subsec:Appurication}, we state the assumptions of Application $(1)$ $($Proposition \ref{prop:evenness}$)$ and provide its proof. In \S\S \ref{subsec:Real vector bundle}, we introduce the tools necessary for proving the bulk-edge correspondence. Specifically, we define real vector bundles in the sense of Atiyah and $\mathbb{Z}_2$-equivariant homotopy sets. After this point, we refer to such objects as “Real” vector bundles. Building on this, in \S\S \ref{subsec:Proof of the bulk-edge correspondence}, we prove Application $(2)$, namely the bulk-edge correspondence for four-dimensional topological insulators in class AI. Finally, in \S \ref{sec:local model}, we detail the spectral analysis of local models used in the specific computational examples of \S\S \ref{subsec:Specific calculation example} and the proof of the bulk-edge correspondence in \S\S \ref{subsec:Proof of the bulk-edge correspondence}.

Throughout this paper, we denote the identity map on the set $X$ by $1_X$. Additionally, we denote the interior of the topological space $X$ by $\text{int}(X)$.

\medskip

\begin{acknowledgment}
The author would like to thank my supervisor, Kiyonori Gomi. The author is grateful to Yuki Matsuoka for helpful discussions and comments. This work was supported by JST SPRING, Japan Grant Number JPMJSP2180.
\end{acknowledgment}


\section{Clifford algebra and its representation}
\label{sec:Clifford algebra}

We summarize the basics about Clifford algebra. The references include \cite{A-B-S}, \cite{Lawson}, and \cite{Furuta}.

\begin{dfn}
Let $n \in \mathbb{N}$ be a natural number and \(\delta_{ij}\) denote the Kronecker delta. The Clifford algebra $Cl_{n}$ is an associative algebra over $\mathbb{R}$ with generators $e_{1},...,e_{n}$ which satisfy the following relations:
\begin{align*}
e_ie_j+e_je_i=2\delta_{ij}.
\end{align*}
Let $Cl_0=\mathbb{R}.$
\end{dfn}

\begin{dfn}
\begin{enumerate}
\item Let \(\hat{V}\) be a vector space, and let \(V^0, V^1 \subset \hat{V}\) be subspaces. 
In this case, if \(\hat{V}\) admits a direct sum decomposition \(\hat{V} = V^0 \oplus V^1\) 
by \(V^0\) and \(V^1\), then \(\hat{V}\) is said to be \(\mathbb{Z}_2\)-graded by \(V^0\) 
and \(V^1\). A vector space \(\hat{V}\) that is \(\mathbb{Z}_2\)-graded is called a 
\(\mathbb{Z}_2\)-graded vector space.

\item A degree-0 linear map \(\hat{f}: \hat{A} \to \hat{B}\) between two \(\mathbb{Z}_2\)-graded 
vector spaces \(\hat{A}\) and \(\hat{B}\) is defined as a linear map that satisfies 
\(\hat{f}(A^0) \subset B^0\) and \(\hat{f}(A^1) \subset B^1\).
\item A degree-1 linear map \(\hat{f}: \hat{A} \to \hat{B}\) between two \(\mathbb{Z}_2\)-graded 
vector spaces \(\hat{A}\) and \(\hat{B}\) is defined as a linear map that satisfies 
\(\hat{f}(A^0) \subset B^1\) and \(\hat{f}(A^1) \subset B^0\).
\end{enumerate}
\end{dfn}

In general, a $\mathbb{Z}_2$-graded vector space $\hat{V}$ is the direct sum of the parts of degree $0$ and $1$: $\hat{V}=V^0 \oplus V^1$.
This is equivalent to considering the eigendecomposition of the involution $\hat{\epsilon} \colon \hat{V} \to \hat{V}$ by the eigenspaces 
$V^i=\ker(\hat{\epsilon}-(-1)^i)$. In this case, the linear mapping $\hat{f} \colon \hat{V} \to \hat{V}$ is of degree $0$ (resp.\ $1$), which is equivalent to the commutativity (or anticommutativity) of $\hat{f}$ and $\hat{\epsilon}$.

\begin{dfn}
Let $n \in \mathbb{N}$ be a natural number.
\begin{enumerate}
\item A pair $(V,\gamma)$ is an ungraded (complex) representation of $Cl_{n}$ if $V$ is a complex vector space and $\gamma=(\gamma_1,...,\gamma_n)$ are complex linear maps on $V$ such that 
\begin{align*}
\gamma_i \circ \gamma_j+\gamma_j \circ \gamma_i=2\delta_{ij}1_V,
\end{align*}
where $1_V$ is the identity map on $V$.
\item A pair $(\hat{V}, \hat{\gamma})$ is a $\mathbb{Z}_{2}$-graded (complex) representation of $Cl_n$ if $\hat{V}$ is $\mathbb{Z}_2$-graded complex vector space and 
$\hat{\gamma}=(\hat{\gamma_1},...,\hat{\gamma_n})$ are complex linear maps of degree $1$ on $\hat{V}$ such that 
\begin{align*}
\hat{\gamma_i} \circ \hat{\gamma_j}+\hat{\gamma_j} \circ \hat{\gamma_i}=2\delta_{ij}1_{\hat{V}},
\end{align*}
where $1_{\hat{V}}$ is the identity map on $\hat{V}$.
\end{enumerate}
Here, $V$ and $\hat{V}$ are called representation spaces. In either case, a representation that has no nontrivial subrepresentation is called an irreducible representation.
\end{dfn}

\begin{rem}
Without loss of generality, we can assume that $V$ and $\hat{V}$ have Hermitian inner products and that $\gamma_i$ and $\hat{\gamma_i}$ are unitary operators. 
In that case, $\gamma_i$ and $\hat{\gamma_i}$ are self-adjoint, since $\gamma_i^2=1$ and $\hat{\gamma_i}^2=1$.
\end{rem}

\begin{dfn}
\begin{enumerate}
\item
An ungraded representation $(V, \gamma)$ and $(V^{\prime}, \gamma^{\prime})$ of $Cl_n$ are equivalent if there is a linear isomorphism 
$\phi \colon V\to V^{\prime}$ that preserves the actions of $Cl_{n}$.
\item
A $\mathbb{Z}_2$-graded representation $(\hat{V}, \hat{\gamma})$ and $(\hat{V}^{\prime}, \hat{\gamma}^{\prime})$ of $Cl_n$ are equivalent if there is a linear isomorphism $\hat{\phi} \colon \hat{V}\to \hat{V}^{\prime}$  of degree 0 that preserves the actions of $Cl_{n}$.
\end{enumerate}
\end{dfn}

Let \( (\hat{V} = V^0 \oplus V^1,\, \hat{\gamma} = (\hat{\gamma}_1, \ldots, \hat{\gamma}_n)) \) be a \( \mathbb{Z}_2 \)-graded representation of \( Cl_n \) for a natural number \( n \geq 2 \). We can then consider $V^{0}=V^{1}=V$, and assume that the involution $\hat{\epsilon}$ and $\hat{\gamma}_n$  are of the following form:
\begin{align*}
\hat{\epsilon}&=
\begin{pmatrix} 
  1_V & 0 \\
  0 & -1_V  
\end{pmatrix} ,&
\hat{\gamma}_n=
\begin{pmatrix} 
  0 & -\sqrt{-1} \\
  \sqrt{-1} & 0
\end{pmatrix}.
\end{align*}
Then, each $\hat{\gamma}_i$ has the following form:
\begin{align*}
\hat{\gamma}_i=
\begin{pmatrix}
0 & -\gamma_i\\
\gamma_i & 0
\end{pmatrix}
\end{align*}
for $i=1,\ldots ,n-1$.

The linear maps $\gamma_1,..., \gamma_n$ obtained here give a representation $V$ of $Cl_{n-1}$. However, it is obtained by specifying the concrete form of  $\gamma_n$ and depends on its choice.

\begin{prop}
Let $\nu_n$ be the number of equivalence classes of the irreducible representations of $Cl_n$, and $\hat{\nu}_n$ that of the graded irreducible representations. Then, 
\begin{align*}
\nu_n&=
 \begin{cases}
    1, & (n=0 \mod 2) \\
    2,      & (n=1 \mod 2)
  \end{cases}&
\hat{\nu}_n=
\begin{cases}
    2, & (n=0 \mod 2) \\
    1.      & (n=1 \mod 2)
  \end{cases}
 \end{align*}
\end{prop}
When there are two equivalence classes of irreducible representations, their dimensions are equal.
Furthermore, if the dimension of the irreducible representation of $Cl_n$ is $d_n$ and the dimension of the graded irreducible representation is $\hat{d}_n$, then the following holds:
\begin{align*}
d_1&=1, &d_2&=2, &d_{n+2}&=2d_n,\\
\hat{d}_1&=2, &\hat{d}_2&=2, &\hat{d}_{n+2}&=2\hat{d}_n.
\end{align*}
\begin{table}[h]
    \centering
    \label{tab:hogehoge}
   \begin{tabular}{|c|c|c|c|c|c|c|c|c|c|c|c|}
        \hline
        $n$ & $1$ & $2$ & $3$ & $4$ & $5$ & $6$ & $\cdots$ & $2k$ & $2k+1$ \\ \hline
     $d_n$ & 1 & 2 & 2 & 4 & 4 & 8 & $\cdots$ & $2^k$ & $2^k$ \\ \hline
     $\hat{d}_n$ & 2 & 2 & 4 & 4 & 8 & 8 & $\cdots$ &  $ 2^k $  & $ 2^{k+1} $  \\ \hline
    \end{tabular}
\end{table}
\begin{proof}
See \cite{A-B-S}.
\end{proof}

\begin{dfn}\label{tensor product}
Let $(\hat{V}, \hat{\gamma})$ and $(\hat{V^{\prime}}, \hat{\gamma}^{\prime})$ be $\mathbb{Z}_2$-graded representations of $Cl_n$ and $Cl_m$, respectively. Furthermore, let $\epsilon$ and $\epsilon^{\prime}$ be the involutions that define the $\mathbb{Z}_2$-gradings of $\hat{V}$ and $\hat{V}^{\prime}$, respectively. We then define the $\mathbb{Z}_2$-graded tensor product representation of the two $\mathbb{Z}_{2}$-graded representations $(\hat{V}, \hat{\gamma})$ and $(\hat{V^{\prime}}, \hat{\gamma}^{\prime})$ by the following pair.

\begin{align*}
(\hat{V}\otimes \hat{V^{\prime}}, ((\hat{\gamma_{1}}\otimes 1_{\hat{V^{\prime}}},...,\hat{\gamma_{n}}\otimes 1_{\hat{V^{\prime}}}), (1_{\hat{V}}\otimes \hat{\gamma^{\prime}}_{1},..., 1_{\hat{V}}\otimes \hat{\gamma^{\prime}}_{m})).
\end{align*}

Here, the $\mathbb{Z}_2$-grading of $\hat{V} \otimes \hat{V}^{\prime}$ is given by $\epsilon \otimes \epsilon^{\prime}$. Additionally, the product is defined as follows for $i, j = 1,...,n$ and $k, l=1,..., m$.

\begin{align*}
&\hat{\gamma}_i \otimes 1_{\hat{V^{\prime}}} \circ \hat{\gamma}_j \otimes 1_{\hat{V^{\prime}}}=\hat{\gamma_i}\circ \hat{\gamma_j}\otimes 1_{\hat{V^{\prime}}},\\
&\hat{\gamma_i} \otimes 1_{\hat{V^{\prime}}} \circ 1_{\hat{V}}\otimes \hat{\gamma^{\prime}_k}=\hat{\gamma_i} \otimes \hat{\gamma^{\prime}}_{k},\\
& 1_{\hat{V}} \otimes \hat{\gamma^{\prime}}_k \circ \hat{\gamma}_i \otimes 1_{\hat{V^{\prime}}}=-\hat{\gamma}_i \otimes \hat{\gamma^{\prime}}_k,\\
& 1_{\hat{V}} \otimes \hat{\gamma^{\prime}}_k \circ 1_{\hat{V}} \otimes \hat{\gamma^{\prime}}_l=1_{\hat{V}} \otimes \hat{\gamma^{\prime}}_k \circ \hat{\gamma^{\prime}}_l.
\end{align*}
\end{dfn}

\begin{prop}\label{prop:representation}
The following holds for the $\mathbb{Z}_2$-graded tensor product representation.
\begin{enumerate}
\item The isomorphism $Cl_n \otimes Cl_m \cong Cl_{n+m}$ holds. Furthermore, if $\hat{V}$ and $\hat{V}^{\prime}$ are $\mathbb{Z}_2$-graded representations of $Cl_n$ and $Cl_m$, respectively, then the $\mathbb{Z}_2$-graded tensor product representation $\hat{V} \otimes \hat{V}^{\prime}$ becomes a $\mathbb{Z}_2$-graded representation of $Cl_{n+m}$.
\item Suppose that $\hat{V}$ and $\hat{V}^{\prime}$ are $\mathbb{Z}_2$-graded irreducible representations of $Cl_{2n}$ and $Cl_{2m}$, respectively. Then, the $\mathbb{Z}_2$-graded tensor product representation $\hat{V} \otimes \hat{V}^{\prime}$ becomes a $\mathbb{Z}_2$-graded irreducible representation of $Cl_{2n+2m}$
\end{enumerate}
\end{prop}

\begin{proof}
For $(1)$, see \cite{A-B-S} and \cite{Lawson}. $(2)$ is described in \cite{Furuta}.

\end{proof}

\begin{convention}\label{conv:std-cliff}
We fix once and for all the following $\mathbb{Z}_2$-graded irreducible representation
of $Cl_2$:
\[
\hat S^{(2)}=\mathbb{C}^2,\qquad
\hat\epsilon=
\begin{pmatrix}1&0\\0&-1\end{pmatrix},\qquad
\hat\gamma_1=
\begin{pmatrix}0&1\\1&0\end{pmatrix},\quad
\hat\gamma_2=
\begin{pmatrix}0&-\sqrt{-1}\\ \sqrt{-1}&0\end{pmatrix}.
\]
This is the standard choice corresponding to the clutching function
$z=x_1+\sqrt{-1}x_2$ on $S^1$, hence to the Bott generator.

For $n\ge 1$, we define the standard $\mathbb{Z}_2$-graded irreducible $Cl_{2n}$ representation
by the graded tensor power (Definition~$\ref{tensor product}$)
\[
\hat S^{(2n)}=\underbrace{\hat S^{(2)}\;\hat\otimes\;\cdots\;\hat\otimes\;\hat S^{(2)}}_{n\text{ times}}.
\]
We denote by $\hat\epsilon^{(2n)}$ and $\hat\gamma^{(2n)}_1,\dots,\hat\gamma^{(2n)}_{2n}$
the induced grading and Clifford generators on $\hat S^{(2n)}$.

For odd degrees, we fix the standard (ungraded) $Cl_{2n-1}$ representation as follows.
Choose a basis so that $\hat\epsilon^{(2n)}=\mathrm{diag}(I_{2^{n-1}},-I_{2^{n-1}})$.
Then for $1\le i\le 2n-1$ we can write
\[
\hat\gamma^{(2n)}_i=
\begin{pmatrix}0&\gamma_i\\ \gamma_i&0\end{pmatrix},
\]
and we \emph{define} the standard ungraded $Cl_{2n-1}$-action by these matrices
$\gamma_1,\dots,\gamma_{2n-1}$ (cf.\ Remark~$\ref{rem:3}$).
\end{convention}

\begin{rem}\label{rem:std-ungraded-Cl1-Cl3}
We record the \emph{standard} ungraded irreducible Clifford representations used in odd sign coordinates.

\smallskip
\noindent
(1) \textbf{The case $Cl_{1}$.}
There are two inequivalent ungraded irreducible $Cl_{1}$ representations $V_{\pm}=\C$ with $\gamma_{1}=\pm 1$.
Throughout the paper, the \emph{standard} choice is $V_{+}$ (i.e.\ $\gamma_{1}=+1$), which is induced from
Convention~\ref{conv:std-cliff} (the $n=1$ odd representation obtained from $\hat S^{(2)}$ by the off-diagonal convention).

\smallskip
\noindent
(2) \textbf{The case $Cl_{3}$.}
For odd sign coordinates in dimension $3$, we use the standard ungraded $Cl_{3}$-action on $\C^{2}$
induced from Convention~\ref{conv:std-cliff} with $n=2$, namely the Pauli matrices
\[
\gamma_{1}=
\begin{pmatrix}
0 & 1\\
1 & 0
\end{pmatrix},\quad
\gamma_{2}=
\begin{pmatrix}
0 & -\sqrt{-1}\\
\sqrt{-1} & 0
\end{pmatrix},\quad
\gamma_{3}=
\begin{pmatrix}
1 & 0\\
0 & -1
\end{pmatrix}.
\]
\end{rem}


\section{Formulation of topological $K$-theory via Fredholm operators}
\label{sec:$K$-theory}


\subsection{Definition of the $K$-group}
\label{subsec:Definition of $K$-theory}

In this section, we summarize the formulation of $K$-theory by using Fredholm operators, based on \cite{A-Si}. For other formulations of topological $K$-theory, see, for example, \cite{A1} and \cite{Karoubi}.

\begin{dfn}
For $n \geq 0$, let $\hat{\mathcal{H}}_n$ be a separable infinite-dimensional Hilbert space with the following properties:
\begin{enumerate}
\item It has $\mathbb{Z}_{2}$-grading and contains all $\mathbb{Z}_{2}$-graded irreducible representations of $Cl_n$ with infinite multiplicity. When $n = 0$, the parts of degree $0$ and $1$ are infinite-dimensional.
\item The above $\mathbb{Z}_{2}$-graded representation of $Cl_n$ can be extended to a $\mathbb{Z}_{2}$-graded representation of $Cl_{n+1}$.
\end{enumerate}
\end{dfn}

\begin{rem}
\begin{enumerate}
\item For any $n \geq 0$, $\hat{\mathcal{H}}_n$ exists. In fact, $Cl_{n+1}\otimes l^2$ is a separable infinite-dimensional Hilbert space that satisfies the properties of 
$\hat{\mathcal{H}}_n$
\item  $\hat{\mathcal{H}}_n$ is isomorphic to each other by a bounded operator which preserves the action of $Cl_{n+1}$ and the inner product, and the whole such isomorphisms form a contractible space with respect to the topology given by the operator norm. (This  is a consequence of Kuiper's theorem \cite{Kuiper}.)
Thus, $\hat{\mathcal{H}}_n$ is uniquely defined  up to isomorphism.
\end{enumerate}
\end{rem}

\begin{dfn}
Let $Fred_n=Fred_n(\hat{\mathcal{H}}_n)$ be the set of Fredholm operators $A \colon \hat{\mathcal{H}}_n \to \hat{\mathcal{H}}_n$ that are self-adjoint, degree 1, and anticommute with the actions $\gamma_1, \gamma_2, \ldots, \gamma_n$ of $Cl_n$. Then, we topologize $Fred_n$ by the operator norm.
\end{dfn}

\begin{lem}\label{lem:3.3}
For any $n \geq 0$, the following map is a homeomorphism:
\begin{align*}
&\phi \colon Fred_n(\hat{\mathcal{H}}_n)\approx Fred_{n+2}(\hat{\mathcal{H}}\otimes \hat{S}), &\hat{A} \mapsto \hat{A}\otimes 1.
\end{align*}
Here, $\hat{S}$ is the standard $\mathbb{Z}_2$-graded irreducible $Cl_2$ representation fixed in Convention~\ref{conv:std-cliff}.
\end{lem}

\begin{proof}
As the $\mathbb{Z}_2$-graded irreducible representation $\hat{S}$, we use the previously given example. Namely, $\hat{S} = \mathbb{C}\oplus \mathbb{C}$, and the grading $\hat{\epsilon}$ and $\hat{\gamma}$ are given as follows:
\begin{align*}
\hat{\epsilon}&=\begin{pmatrix} 
1 & 0 \\
0& -1
\end{pmatrix},&
\hat{\gamma_{1}}&=\begin{pmatrix}
0 & 1 \\
1 & 0
\end{pmatrix},&
\hat{\gamma_{2}}&=
\begin{pmatrix} 
  0 & -\sqrt{-1} \\
  \sqrt{-1} & 0
\end{pmatrix}.
\end{align*}
It is clear from the definition that $\phi$ is continuous and injective. We will prove surjectivity. From the definition of the degree of \(\hat{\mathcal{H}}_n \otimes \hat{S}\), the following is obtained:
\begin{align*}
&(\hat{\mathcal{H}}_n\otimes \hat{S})^0=\mathcal{H}_n^0 \otimes S^0 \oplus \mathcal{H}_n^1 \otimes S^1, 
&(\hat{\mathcal{H}}_n\otimes \hat{S})^1=\mathcal{H}_n^0 \otimes S^1 \oplus \mathcal{H}_n^1 \otimes S^0.
\end{align*}
Let us arbitrarily take \(\hat{B} \in Fred_{n+2}(\hat{\mathcal{H}}_n \otimes \hat{S})\) and we define \(e_0 = {}^t\!(1,0)\) and \(e_1 = {}^t\!(0,1) \in \mathbb{C}^2\). 
For any \(a_0 \in \mathcal{H}_n^0\), we describe the images of \(a_0 \otimes e_0\) and \(a_0 \otimes e_1\) under \(\hat{B}\) as follows:
\begin{align*}
&\hat{B}(a_0 \otimes e_0)=x \otimes e_0+y \otimes e_1, &\hat{B}(a_0 \otimes e_1)=z\otimes e_0+w\otimes e_1.
\end{align*}
We consider the following as a bilinear mapping \(B_1\) on \(\mathcal{H}_n^0 \times S^0\):
\begin{align*}
&B_1\colon \mathcal{H}_n^0 \times S^0 \to \hat{\mathcal{H}}\otimes \hat{S}, &(x,y)\mapsto \hat{B}(x\otimes y).
\end{align*}
Since $\hat{B}$ anticommutes with $1 \otimes \hat{\gamma}_1$ and $1\otimes \hat{\gamma}_2$, the following holds:
\begin{align*}
0&=(B_1 \circ (1\times \hat{\gamma}_1)+(1\otimes \hat{\gamma}_1) \circ B_1)(a_0, e_0)
=B_1(a_0, \hat{\gamma}_1 e_1)+1 \otimes \hat{\gamma}_1(x \otimes e_0+y \otimes e_1)\\
&=(z\otimes e_0+w\otimes e_1)-(x\otimes e_1)+y\otimes e_1=(z+y)\otimes e_0+(w-x)\otimes e_1.
\end{align*}
From a similar calculation, we obtain the following:
\begin{align*}
0=(B_1 \circ (1\times \hat{\gamma}_1)+(1\otimes \hat{\gamma}_1) \circ B_1)(a_0, e_0)=
\sqrt{-1}(z-y)\otimes e_0+\sqrt{-1}(w-x)\otimes e_1.
\end{align*}
From the above calculation, it follows that \(z = y = 0\) and \(w = x\), and thus, for some linear mapping \(A_1 \colon \mathcal{H}_n^0 \to \mathcal{H}_n^1\), we can write \(B_1 = A_1 \times 1\). Thus, from the universality of the tensor product, we obtain \(\hat{B}|_{\mathcal{H}_n^0 \otimes S^0} = A_1 \otimes 1\).

By repeating exactly the same argument on \(\mathcal{H}_n^1 \otimes S^1\), \(\mathcal{H}_n^0 \otimes S^1\), and \(\mathcal{H}_n^1 \otimes S^1\), it is shown that \(\hat{B} = \hat{A} \otimes 1\) for some self-adjoint linear mapping \(\hat{A} \colon \hat{\mathcal{H}} \to \hat{\mathcal{H}}\) of degree \(1\).

From the action of the Clifford algebra on $\hat{\mathcal{H}}_n$ and the anticommutativity of \(\hat{B}\), it follows that \(\hat{A}\) anticommutes with the action of \(Cl_n\). Furthermore, from the Fredholm property of \(\hat{B}\), it follows that \(\hat{A}\) is Fredholm. That is, \(\hat{A} \in Fred_n(\hat{\mathcal{H}}_n)\) and $\phi$ is surjective. Since $\phi$ is bijective and continuous, it follows from the open mapping theorem that $\phi$ is a homeomorphism.
\end{proof}

From Lemma \ref{lem:3.3}, it follows that $Fred_{n} \approx Fred_{n+2}$.

When $n=0$, the direct summands of $\hat{\mathcal{H}}_0=\mathcal{H}_{0}^{0}\oplus \mathcal{H}_{0}^{1}$
are both infinite-dimensional separable Hilbert spaces. Therefore, they can be identified. That is, $\mathcal{H}_0^0=\mathcal{H}_0^1=\mathcal{H}$.
In this case, the involution $\hat{\epsilon}$ and $\hat{A} \in Fred_0(\hat{\mathcal{H}})$ can be expressed as follows:
\begin{align*}
&\hat{\epsilon}=
\begin{pmatrix}
1_{\mathcal{H}} & 0 \\
0 & -1_{\mathcal{H}}
\end{pmatrix},
&\hat{A}=
\begin{pmatrix}
0 & A^*\\
A & 0
\end{pmatrix}.
\end{align*}
Here, $A \colon \mathcal{H} \to \mathcal{H}$ is a Fredholm operator. By this correspondence, $Fred_{0}$ can be identified with the space of all Fredholm operators.

When $n=1$, the involution $\hat{\epsilon}$ and $\mathbb{Z}_2$-graded representation $\gamma_1$ of $Cl_1$ can be expressed as follows:
\begin{align*}
&\hat{\epsilon}=
\begin{pmatrix}
1_{\mathcal{H}} & 0 \\
0 & -1_{\mathcal{H}}
\end{pmatrix},
&\gamma_1=
\begin{pmatrix}
0 & -\sqrt{-1} \\
\sqrt{-1} & 0
\end{pmatrix}.
\end{align*}
Then, $\hat{A} \in Fred_{1}(\hat{\mathcal{H}})$ has the following form:
\begin{align*}
\hat{A}=
\begin{pmatrix}
0 & A\\
A & 0
\end{pmatrix}.
\end{align*}
Here, $A\colon \mathcal{H}\to \mathcal{H}$ is a self-adjoint Fredholm operator. We define $Fred_{sa}(\mathcal{H})$ as the set of all self-adjoint Fredholm operators on $\mathcal{H}$, and topologize it with the relative topology induced from $Fred(\mathcal{H})$. By this correspondence, $Fred_1$ can be identified with $Fred_{sa}$.

\begin{prop}
The space $Fred_{sa}$ of self-adjoint Fredholm operators decomposes into the following three path-connected components:
\begin{enumerate}
\item The subspace $Fred_{sa}^{+}$, consisting of self-adjoint Fredholm operators that are positive in the Calkin algebra, forms a path-connected component of $Fred_{sa}$ and is contractible.
\item The subspace $Fred_{sa}^{-}$, consisting of self-adjoint Fredholm operators that are negative in the Calkin algebra, forms a path-connected component of $Fred_{sa}$ and is contractible.
\item The subspace $Fred_{sa}^1 = Fred_{sa} \setminus (Fred_{sa}^{+} \cup Fred_{sa}^{-})$ forms a path-connected component of $Fred_{sa}$ and is not contractible.
\end{enumerate}
\end{prop}

\begin{proof}
See \cite{A-Si}.
\end{proof}

\begin{dfn}\label{dfn:Space of Fredholm operators}
\begin{enumerate}
\item Let $\mathcal{F}_1(\mathcal{\hat{H}}_1) \subset Fred_{1}(\mathcal{\hat{H}}_1)$ be the subspace corresponding to $Fred_{sa}^1$.
\item Under the homeomorphism $Fred_1 \approx Fred_{2k+1}$ given by Lemma \ref{lem:3.3}, let $\mathcal{F}_{2k+1}$ be the subspace corresponding to 
$\mathcal{F}_1 \subset Fred_{2k+1}$. Also, let $\mathcal{F}_{2k}=Fred_{2k}$.
\item For $n \geq 0$, let $\mathcal{F}_{n}^{*} \subset \mathcal{F}_n$ be the subspace of invertible operators.
\end{enumerate}
\end{dfn}

\begin{dfn}
\begin{enumerate}
\item A pair $(X, Y)$ is defined as follows: $X$ is a topological space, and $Y$ is a closed subset of $X$. In particular, when $X$ is a compact Hausdorff space, the pair $(X, Y)$ is called a compact pair. Additionally, when $X$ is a CW-complex and $Y$ is a subcomplex, the pair $(X, Y)$ is referred to as a CW-pair.
\item Consider two pairs $(X,Y)$ and $(Z,W)$. A continuous map $f\colon X\to Z$ is called a map from the pair $(X,Y)$ to the pair $(Z,W)$ if the image of $Y$ under $f$ is contained in $W$. Such a map is denoted by $f\colon (X,Y)\to (Z,W)$.
\item Two maps $f, g\colon (X,Y) \to (Z,W)$ between the pairs $(X,Y)$ and $(Z,W)$ are homotopic if there exists a map $H\colon (X\times[0,1], Y\times[0,1]) \to (Z,W)$ such that $H(x,0) = f(x)$ and $H(x,1) = g(x)$ for all $x \in X$. The quotient of the set of all maps from $(X,Y)$ to $(Z,W)$ under the homotopy relation is denoted by $[(X,Y), (Z,W)]$.
\end{enumerate}
\end{dfn}

\begin{dfn}
Let $n\geq 0$ be a non-negative integer, and $(X,Y)$ a compact pair (or CW-pair). Then, $K^{-n}(X,Y)$ is defined as follows:
\begin{align*}
K^{-n}(X, Y)=[(X,Y), (\mathcal{F}_n, \mathcal{F}_n^{*})].
\end{align*}
In particular, when $Y = pt$, that is, for a pointed space X, the reduced $K$-group is defined as follows:
\begin{align*}
\tilde{K}^{-n}(X)=K^{-n}(X, pt).
\end{align*}
\end{dfn}

It is known that the following holds:
\begin{align*}
\tilde{K}^{-n}(X/Y)=K^{-n}(X, Y).
\end{align*}

The Abelian group structure on $K^{-n}(X, Y)$ is induced by the direct sum of the family of Fredholm operators by using an isomorphism 
$\hat{\mathcal{H}}_n\oplus \hat{\mathcal{H}}_n \cong \hat{\mathcal{H}}_n$.

The operation of changing the $\mathbb{Z}_2$-grading on the Hilbert space $\hat{\mathcal{H}}_n$ is defined as replacing the involution $\varepsilon$, which determines the original $\mathbb{Z}_2$-grading, with its negative $-\varepsilon$. The resulting graded Hilbert space is denoted by $\Pi \hat{\mathcal{H}}_n$.

Based on this definition, the inverse of a family of Fredholm operators $\hat{A} \colon X \to \mathcal{F}_n$ in the $K$-group is given by
\[
\Pi \hat{A} \colon X \to \mathcal{F}_n(\Pi \hat{\mathcal{H}}_n), \quad \Pi \hat{A}_x = \hat{A}_x.
\]
That is, while the operator $\hat{A}$ itself remains unchanged, its inverse is defined through the change in the $\mathbb{Z}_2$-grading of the underlying Hilbert space.

By Lemma \ref{lem:3.3}, there is a natural isomorphism $K^{-n}(X,Y)\cong K^{-n-2}(X, Y)$ of Abelian groups. By using this periodicity, we define \(K^n(X, Y)\) and \(\tilde{K}^n(X)\) for natural numbers \(n \in \mathbb{N}\) as well.

The $K$-theory product $K^{-n}(X, Y)\times K^{-m}(X, Z)\to$ $K^{-n-m}(X, Y \cup Z)$ is induced by 
\begin{align*}
\hat{A} \otimes 1+1 \otimes \hat{A}^{\prime} \colon X \to \mathcal{F}_{n+m}(\mathcal{\hat{H}}_n \otimes \mathcal{\hat{H}}^{\prime}_m)
\end{align*}
with the representatives $A \colon X \to \mathcal{F}_n(\mathcal{\hat{H}}_n)$ and $\hat{A}^{\prime} \colon X \to \mathcal{F}_m(\mathcal{\hat{H}}^{\prime}_m).$

It is known that $K$-theory constitutes a generalized cohomology theory. Below, we recall the axioms of generalized cohomology.

\begin{prop}
The following statements hold:

\begin{itemize}
    \item \textbf{Homotopy Axiom}: Let \((X, Y)\) and \((Z, W)\) be compact pairs. If two maps \(A, B: (X, Y) \to (Z, W)\) are homotopic, denoted \(A \simeq B\), then the induced maps \(A^*, B^*: K^*(Z, W) \to K^*(X, Y)\) on the cohomology groups are equal.
    \item \textbf{Exactness Axiom}: For any compact pair $(X, A)$, the following sequence
\[
\begin{tikzpicture}[thick]
\node (a) at (0, 0) {$\cdots$};
\node (b) at (2, 0) {$K^{n-1}(A)$};
\node (c) at (4, 0) {$K^n(X, A)$};
\node (d) at (6, 0) {$K^n(X)$};
\node (e) at (8, 0) {$K^n(A)$};
\node (f) at (10, 0) {$\cdots$};
\node (g) at (12, 0) {$(\text{exact})$};

\draw[->] (a) -- (b);
\draw[->] (b) -- node[midway, above] {\scriptsize$\delta^*$} (c);
\draw[->] (c) -- node[midway, above] {\scriptsize$j^*$} (d);
\draw[->] (d) -- node[midway, above] {\scriptsize$i^*$} (e);
\draw[->] (e) -- node[midway, above] {\scriptsize$\delta^*$} (f);
\end{tikzpicture}
\]
is exact. Here, $i \colon A \to X$ and $j \colon X \to (X, A)$ are inclusion maps.
    \item \textbf{Excision Axiom}: Let $(X, A)$ and $(X, B)$ be compact pairs. Then, for any $n \in \mathbb{Z}$, there is an isomorphism
\begin{align*}
i^* \colon K^n(A \cup B, B) \cong K^n(A, A \cap B),
\end{align*}
where $i \colon (A, A \cap B) \to (A \cup B, B)$ is the inclusion map.
\end{itemize}
\end{prop}

\begin{proof}
The Homotopy Axiom follows directly from its definition. The Excision Axiom is a consequence of Kuiper's Theorem \cite{Kuiper}. The exactness axiom in the category of finite CW-complexes is derived from the suspension isomorphism and cofiber sequence. (The details on the suspension isomorphism will be discussed in the next section.) For details, see, for example, \cite{Araki}, \cite{Rudyak}. For compact pairs, this can be understood by using the formulation in terms of vector bundles in topological $K$-theory \cite{A1}.
\end{proof}


\subsection{Description of the generator of $K^n(D^n,\partial D^n)$}
\label{subsec:Description of the generator}

In this section, we let $I=[-\pi/2, \pi/2]$ and $\partial I = \{\pm \pi/2\}$.

When \(\hat{A}: (X, Y) \to (\mathcal{F}_{n+1}(\mathcal{\hat{H}}_{n+1}), \mathcal{F}_{n+1}^{*}(\mathcal{\hat{H}}_{n+1}))\) 
is given, we define 
\begin{align*}
\mathrm{AS}(\hat{A}): (X \times I, Y \times I \cup X \times \partial I) 
\to (\mathcal{F}_{n}(\mathcal{\hat{H}}_{n+1}), \mathcal{F}^{*}_{n}(\mathcal{\hat{H}}_{n+1}))
\end{align*}
as follows:
\begin{align*}
\mathrm{AS}(\hat{A})(x,t)=\hat{A}_{x}\cos{t}-\hat{\gamma}_{n+1}\sin{t}.
\end{align*}

In their work \cite{A-Si}, Atiyah and Singer establish an index theory for skew-adjoint Fredholm operators, derive the Bott periodicity, and relate the operator index to $K$-theory invariants. Building on this framework, we have:

\begin{thm}
\(\mathrm{AS}(\hat{A}): (X \times I, Y \times I \cup X \times \partial I) \to (\mathcal{F}_{n}(\mathcal{\hat{H}}_{n+1}), \mathcal{F}^{*}_{n}(\mathcal{\hat{H}}_{n+1}))\) induces an isomorphism 
\begin{align*}
K^{-(n+1)}(X, Y) \cong K^{-n}(X \times I, Y \times I \cup X \times \partial I). 
\end{align*}
Furthermore, the following isomorphism is also induced from AS:
\begin{align*}
\tilde{K}^{-n-1}(X) \cong \tilde{K}^{-n}(SX)
\end{align*}
where \( S \) is the suspension of a topological space, defined as \( SX = (X \times I)/(pt \times I \cup X \times \partial I) \). These isomorphisms are called suspension isomorphisms.
\end{thm}

\begin{proof}
See \cite{A-Si}.
\end{proof}

Let $n$ be $1$ or $-1$. We define the operator $Sh_n$ on $l^2$ as follows.
\begin{align*}
Sh_n((a_1,a_2,...))=
\begin{cases}
	(0,a_1,a_2,...)& n=1,\\ 
	(a_2,a_3,...)& n=-1.
\end{cases}
\end{align*}
Clearly, the adjoint of $Sh_{-1}$ is $Sh_1$. Namely, $Sh_{-1}^* = Sh_1$.
We define $\hat{l^2}$ as the direct sum of $l^2$, $\hat{l^2} = l^2 \oplus l^2$.

\begin{prop}\label{Prop:generator on n=0}
The $K$-group of a point $pt$ is isomorphic to $\mathbb{Z}$, i.e., $K(pt)=[pt,\mathcal{F}_0(\hat{l}^2)] \cong \mathbb{Z}$, and its generator is given as follows:
\begin{align*}
\hat{Sh}=
\begin{pmatrix}
0 & Sh_{1} \\
Sh_{-1} & 0
\end{pmatrix}.
\end{align*}
\end{prop}

\begin{proof}
The set $[pt, Fred(l^2)]$ is bijective to the set $\mathbb{Z}$ of integers via the Fredholm index. Namely, $\text{Ind}: [pt, Fred(l^2)] \cong \mathbb{Z}$.
Since $\text{Ind}(Sh_{-1}) = \dim{(\ker{(Sh_{-1}}))} - \dim{(\text{coker}(Sh_{-1}))} = 1 - 0 = 1$, it follows that $Sh_{-1}$ provides a generator. In particular, from the correspondence between $Fred(l^2)$ and $\mathcal{F}_0(l^2)$, the operator $\hat{Sh}$ gives a generator for $K(pt)$.
\end{proof}

We define $l^2_{\ge 2}$ as the set of all elements in $l^2$ whose first component is $0$. The space $l^2_{\ge 2}$ is a closed subspace, and the decomposition $l^2 = \ker{Sh_{-1}} \oplus l^2_{\ge 2}$ holds. Note also that $Sh_{-1}$ induces an isomorphism $l^2_{\ge 2} \cong l^2$. The inverse mapping is given by $Sh_1: l^2 \to l^2_{\ge 2}$. In particular, $\hat{Sh}: l^2_{\ge 2} \oplus l^2 \cong l^2_{\ge 2} \oplus l^2$ is an invertible operator of degree $1$.

Let \( A \), \( B \), and \( C \) be vector spaces. Suppose \( f: A \to C \) and \( g: B \to C \) are linear maps. Then the linear map \( f + g: A \oplus B \to C \) is defined by  
\[
(f + g)(a, b) = f(a) + g(b).
\]

\begin{prop}\label{Prop:generator on pt}
Let $n\in \mathbb{N}$ be a natural number, and let $\hat{S}^{(2n)}$ denote the \emph{standard} $\mathbb{Z}_2$-graded irreducible representation of $Cl_{2n}$ fixed in Convention~\ref{conv:std-cliff}. Then, with the zero map on $\hat{S}^{(2n)}$ denoted by $\hat{0}$, the operator $\hat{0} +  \hat{\gamma}$ on $\hat{S}^{(2n)} \oplus \mathcal{\hat{H}}_{2n}$ provides a generator for $K^{-2n}(pt) = [pt, \mathcal{F}_{2n}(\hat{S}^{(2n)} \oplus \hat{\mathcal{H}}_{2n})]$. Here, $\hat{\gamma}: \mathcal{\hat{H}}_{2n} \to \hat{S}^{(2n)} \oplus \mathcal{\hat{H}}_{2n}$ is an invertible operator that anticommutes with the action of $Cl_{2n}$.
\end{prop}

\begin{proof}
We provide a proof by an induction on $n\in \mathbb{N}$.

\begin{enumerate}
\item Consider the case when $n = 1$.

Proposition \ref{Prop:generator on n=0} and Lemma \ref{lem:3.3} show that $\hat{Sh} \otimes 1 \in [pt, \mathcal{F}_{2n}(l^2 \otimes \hat{S}^{(2n)})]$ provides a generator. Under the identification given by the isomorphisms $\mathbb{C} \otimes \hat{S}^{(2n)} \cong \hat{S}^{(2n)}$ and $Sh_{-1}\otimes 1 \colon l^2_{\ge 2}\otimes \hat{S}^{(2n)} \cong l^2 \otimes \hat{S}^{(2n)}$, which preserve the action of the Clifford algebra $Cl_2$, the operator $\hat{Sh} \otimes 1$ corresponds to $\hat{0} + \hat{\phi} \colon \hat{S}^{(2n)} \oplus (\hat{l^2}\otimes \hat{S}^{(2n)})\to \hat{S}^{(2n)} \oplus (\hat{l^2}\otimes \hat{S}^{(2n)})$. Here, $\hat{\phi}: \hat{l^2} \otimes \hat{S}^{(2n)} \cong \hat{S}^{(2n)} \oplus (\hat{l^2} \otimes \hat{S}^{(2n)})$ is an invertible operator of degree $1$ induced by $\hat{Sh} \times 1: (l^2_{\ge 2}\oplus l^2) \times \hat{S}^{(2n)} \to (l^2_{\ge 2}\oplus l^2) \times \hat{S}^{(2n)}$. By construction, $\hat{\phi}$ anticommutes with the action of $Cl_2$. Therefore, the claim follows in the case $n = 1$.

\item Let $n\in \mathbb{N}$ be a natural number greater than or equal to $2$, and assume that the claim holds up to $n-1$.

The claim follows immediately from the induction hypothesis, Proposition \ref{prop:representation}, and Lemma \ref{lem:3.3}.
\end{enumerate}
\end{proof}

Let $\hat{S}$ be a $\mathbb{Z}_2$-graded irreducible representation of $Cl_{n}$. Then, we define $\hat{\mu}_{\hat{S}}: D^n \to \text{End}(\hat{S})$ as follows.

\begin{align*}
\hat{\mu}_{\hat{S}}(x_1,...,x_n)=\sum_{i=1}^{n} x_i \hat{\gamma}_i.
\end{align*}

\begin{thm}
Let $n \in \mathbb{Z}$ be a non-negative integer. There are the following isomorphisms:
\begin{align*}
&K^{-2n}(D^{2n},\partial D^{2n})\cong \mathbb{Z}, &K^{-2n-1}(D^{2n+1}, \partial D^{2n+1})\cong \mathbb{Z}.
\end{align*}
Let $\hat{S}$ be a $\mathbb{Z}_2$-graded irreducible representation of $Cl_{2n}$. Then, 
\begin{align*}
-\hat{\mu}_{\hat{S}}  + \hat{\gamma}_{2n+1}=-(\sum_{i=1}^{2n} x_i \hat{\gamma}_i) + \hat{\gamma}_{2n+1}
\end{align*}
provides a generator for $K^{-2n}(D^{2n}, \partial D^{2n}) \cong \mathbb{Z}$. Here, $\hat{\gamma}_{2n+1}\colon \hat{\mathcal{H}}_{2n} \to \hat{S} \oplus \hat{\mathcal{H}}_{2n}$ is an invertible operator that anticommutes with the action of $Cl_{2n}$.

Similarly, let $\hat{S}$ be a $\mathbb{Z}_2$-graded irreducible representation of $Cl_{2n+1}$. Then, 
\begin{align*}
-\hat{\mu}_{\hat{S}}  + \hat{\gamma}_{2n+2}=-(\sum_{i=1}^{2n+1} x_i \hat{\gamma}_i) + \hat{\gamma}_{2n+2}
\end{align*}
provides a generator for $K^{-2n-1}(D^{2n+1}, \partial D^{2n+1})\cong \mathbb{Z}$. Here, $\hat{\gamma}_{2n+2}\colon \hat{\mathcal{H}}_{2n+1} \to \hat{S} \oplus \hat{\mathcal{H}}_{2n+1}$ is an invertible operator that anticommutes with the action of $Cl_{2n+1}$.
\end{thm}

\begin{proof}
Since the arguments for even and odd degrees are similar, we provide a proof only for the even case. Given that $K^{-2n}(pt) \cong \mathbb{Z}$, applying the suspension isomorphism $n$ times yields $K(I^n, \partial I^n) \cong \mathbb{Z}$. Below, we explicitly compute $AS^n$. We compute the iterated suspension using $\mathrm{AS}$ as defined above (with the minus sign in front of the sine term) and the standard representations fixed in Convention~\ref{conv:std-cliff}.
\begin{align*}
&AS(\hat{A})(t_1,...,t_{k+1}) = \hat{A}_{(t_1,...,t_k)} \cos{t_{k+1}} - \hat{\gamma}_{k+2} \sin{t_{k+1}},
&(t_1,...,t_{k+1})\in [-\pi/2,\pi/2]^{k+1}.
\end{align*}

Let $\hat{A} \in \mathcal{F}_{2n}(\hat{\mathcal{H}}_{2n})$. By the definition of the map $AS$, the expression for $AS^{2n}(\hat{A})$ is given by
\begin{align*}
AS^{2n}(\hat{A})(t_1, \dots, t_{2n}) = \big( \cdots \big( (\hat{A} \cos{t_1} - \hat{\gamma}_1 \sin{t_1}) \cos{t_2} - \hat{\gamma}_2 \sin{t_2} \big) \cdots \big) \cos{t_{2n}} - \hat{\gamma}_{2n} \sin{t_{2n}}.
\end{align*}
Here, for $j = 0, ..., 2n,$ we define $x_j = x_j(t_1, ..., t_{2n})$ as follows:
\begin{align*}
&x_0=\cos{t_1}\cos{t_2}\cdots \cos{t_{2n}},\\
&x_1=\sin{t_1}\cos{t_2}\cos{t_3}\cdots \cos{t_{2n}},\\
&\vdots\\
&x_{2n-1}=\sin{t_{2n-1}}\cos{t_{2n}},\\
&x_{2n}=\sin{t_{2n}}.
\end{align*}
This is the polar coordinate of the $2n$-dimensional sphere $S^{2n}\subset \mathbb{R}^{2n+1}$, and the map
\begin{align*}
&I^{2n} \to \{(x_0,...,x_{2n})\in \mathbb{R}^{2n+1}|x_0\ge 0\},&(t_1,...,t_{2n})\mapsto (x_0,...,x_{2n})
\end{align*}
is a homeomorphism. Furthermore, the following map is also a homeomorphism
\begin{align*}
&\{(x_0,...,x_{2n})|x_0\ge 0\} \to D^{2n}, &(x_0,...,x_{2n})\mapsto (x_1,...,x_{2n}).
\end{align*}
By using these homeomorphisms, we identify $AS^{2n}(\hat{A})$ with the map
\begin{align*}
AS^{2n}(\hat{A})\colon (D^{2n},\partial D^{2n})\to (\mathcal{F}_0(\hat{\mathcal{H}}_{2n}), \mathcal{F}_0^{*}(\hat{\mathcal{H}}_{2n})).
\end{align*}
This map has the following representation.
\begin{align*}
AS^{2n}(\hat{A})(x_1,...,x_{2n})=\sqrt{(1-\sum_{i=1}^{2n} {x_i}^2})\hat{A}-\sum_{i=1}^{2n} x_i \hat{\gamma}_i.
\end{align*}
If we let $\hat{S}$ be a $\mathbb{Z}_2$-graded irreducible representation of $Cl_{2n}$, then Proposition \ref{Prop:generator on pt} shows that the class of $[\hat{0} +  \hat{\gamma}_{2n+1}]\in K^{-2n}(pt)$ is a generator. Therefore,
\begin{align*}
AS^{2n}(\hat{0} +  \hat{\gamma}_{2n+1})=-\hat{\mu}_{\hat{S}} + AS^{2n}(\hat{\gamma}_{2n+1})
\end{align*}
is given, and $AS^{2n}(\hat{\gamma}_{2n+1})$ is the constant map whose value is an invertible operator $\hat{\mathcal{H}}_{2n}\to \hat{S}\oplus \hat{\mathcal{H}}_{2n}$ that anticommutes with the action of $Cl_{2n}.$
\end{proof}

\subsection{Chern character and Push-forward}\label{subsec:Chern character and Push-forward}


In this section, we define the Chern character for families of Fredholm operators. For the standard definition of the Chern character for vector bundles, see \cite{Husemoller}. The definition of the even-degree Chern character $\operatorname{Ch}_{\text{even}}$ in terms of differential forms is given in \cite{Tu}, while the definition of the odd-degree Chern character $\operatorname{Ch}_{\text{odd}}$ is provided in \cite{Quillen}. Additionally, in this section, we define the push-forward map along the inclusion map from a point, which is necessary for the proof of the main theorem. For a more general definition of the push-forward map, see \cite{Karoubi}. For the suspension isomorphism of singular cohomology, see, for example, \cite{Hatcher}.

\begin{dfn}\label{dfn:ASbar}
Let $I=[-\pi/2,\pi/2]$ and identify $S^{1}=I/\partial I$.
Define a continuous map
\[
\overline{\mathrm{AS}}: S^{1}\times \mathcal{F}_{1}\longrightarrow \mathcal{F}_{0}
\]
as follows. Consider the continuous map
\[
\mathcal{F}_{1}\times I \to \mathcal{F}_{0},\qquad (\hat{A},t)\mapsto \hat{A}\cos t-\hat\gamma_{2}\sin t .
\]
Since $\cos(\pm\pi/2)=0$, the restriction to $\mathcal{F}_{1}\times \partial I$ is constant,
hence the above map descends to the quotient $S^{1}\times \mathcal{F}_{1}$.
We denote the induced map by $\overline{\mathrm{AS}}$.
\end{dfn}

For $n, m \in \mathbb{N}$ with $n \leq m$, we define the Grassmann manifold $G_n(\mathbb{C}^m)$ as the set of all $n$-dimensional complex linear subspaces of $\mathbb{C}^m$. We define $BU(n)=G_n(\mathbb{C}^{ \infty})$ by:
\begin{align*}
BU(n) = G_n(\mathbb{C}^{ \infty}) = \bigcup_{m=n}^{\infty}G_n(\mathbb{C}^m).
\end{align*}
The inclusions \(\mathbb{C}^n \subset \mathbb{C}^{n+1} \subset \dots\) yield inclusions 
\(G_{m}(\mathbb{C}^n) \subset G_{m}(\mathbb{C}^{n+1}) \subset \dots\) and one can equip 
\(G_m(\mathbb{C}^{\infty})\) with the direct limit topology.

The maps \(G_{m}(\mathbb{C}^n) \hookrightarrow G_{m+1}(\mathbb{C}^{n+1})\), 
which consist of adding the subspace generated by the last vector 
\(e_{n+1} = (0, \dots, 0, 1)\), induce maps between the Grassmann manifolds  
\(BU(m) \hookrightarrow BU(m+1)\). Let  
\[
BU = \bigcup_{m=0}^{\infty} BU(m)
\]  
with the inductive limit topology.
\begin{thm}
\begin{enumerate}
\item The integer coefficient cohomology of BU is given as follows:
\begin{align*}
&H^{*}(BU;\mathbb{Z})\cong \mathbb{Z}[c_1,c_2,...], &c_i\in H^{2i}(BU;\mathbb{Z}).
\end{align*}
\item $\mathcal{F}_0$ and $\mathbb{Z} \times BU$ are homotopy equivalent:
\begin{align*}
\mathcal{F}_0 \simeq \mathbb{Z}\times BU.
\end{align*}
\end{enumerate}
\end{thm}

\begin{proof}
For the first claim, see \cite{Milnor} and \cite{Husemoller}. $\mathbb{Z} \times BU$ and $\mathcal{F}_0$ are $\Omega$-spectra representing K-cohomology in the category of CW-complexes. Since an $\Omega$-spectrum representing a generalized cohomology $h^{*}$ on the category of CW-complexes is uniquely determined up to homotopy equivalence $($\cite{Araki}, \cite{Rudyak}$)$, the second claim follows.
\end{proof}

Let $n \in \mathbb{N}$ be a natural number. Set $p_0 = n$, and for each $k \in \mathbb{Z}$, we define $p_k\in H^{2k}(BU; \mathbb{Z})$ inductively by the following formula.
\begin{align*}
p_k=c_1 p_{k-1}-c_2 p_{k-2}+\cdots +(-1)^{k}c_{k-1}p_1 +(-1)^{k+1}kc_k.
\end{align*}
\begin{dfn}\label{dfn:univ-chern}
\begin{enumerate}
\item For each $n \in \N$, we define the $n$-th universal Chern character
$\operatorname{ch}_n^{\mathrm{univ}} \in H^{2n}(\mathcal{F}_0; \mathbb{Q})$ by
\[
\operatorname{ch}_n^{\mathrm{univ}}=\frac{p_n}{n!}.
\]

\item For each $n \in \N$, we define the universal odd Chern character class
$\operatorname{ch}_{n+\frac{1}{2}}^{\mathrm{univ}} \in H^{2n+1}(\mathcal{F}_1;\mathbb{Q})$ by
\[
\operatorname{ch}_{n+\frac12}^{\mathrm{univ}}
:=\Bigl(\overline{\mathrm{AS}}^{\,*}\operatorname{ch}^{\mathrm{univ}}_{n+1}\Bigr)/[S^{1}],
\]
where $\overline{\mathrm{AS}}:S^{1}\times\mathcal{F}_{1}\to\mathcal{F}_{0}$ is the map
in Definition~\ref{dfn:ASbar}, and $-/[S^{1}]$ denotes the slant product with the fundamental
class $[S^{1}]\in H_{1}(S^{1};\mathbb{Z})$; see \cite{Hatcher}.
\end{enumerate}
\end{dfn}

\begin{dfn}\label{dfn:chern-on-X}
Let $X$ be a \emph{pointed compact Hausdorff space}.
Let $n \in \N$.
\begin{enumerate}
\item
For any $[\hat{A}] \in K(X)$, we define $\operatorname{ch}_{n}([\hat{A}]) \in H^{2n}(X;\mathbb{Q})$ as follows:
\begin{align*}
&\operatorname{ch}_n([\hat{A}])=\hat{A}^{*}\operatorname{ch}_n^{\mathrm{univ}} \in H^{2n}(X;\mathbb{Q}), & n\ge 1,\\
&\operatorname{ch}_0([\hat{A}])=Ind([A_x]) \in H^0(X;\mathbb{Q}), &x \in X.
\end{align*}
\item For any $[\hat{A}] \in K^{1}(X)$, we define $\operatorname{ch}_{n+\frac{1}{2}}([\hat{A}]) \in H^{2n+1}(X;\mathbb{Q})$ as follows:
\begin{align*}
&\operatorname{ch}_{n+\frac{1}{2}}([\hat{A}])=\hat{A}^{*}\operatorname{ch}_{n+\frac{1}{2}}^{\mathrm{univ}} \in H^{2n+1}(X;\mathbb{Q}).
\end{align*}
\item The homomorphisms \(\operatorname{Ch}_{even}\) and \(\operatorname{Ch}_{odd}\) from the $K$-group to singular cohomology are defined as follows:
\begin{align*}
&\operatorname{Ch}_{even}\colon K(X)\to H^{even}(X;\mathbb{Q}), &\operatorname{Ch}_{even}([A])=\operatorname{ch}_{0}([A])+\operatorname{ch}_1([A])+\cdots,\\
&\operatorname{Ch}_{odd}\colon K^1(X)\to H^{odd}(X;\mathbb{Q}), &\operatorname{Ch}_{odd}([A])=\operatorname{ch}_{\frac{1}{2}}([A])+\operatorname{ch}_{\frac{3}{2}}([A])+\cdots.
\end{align*}
\end{enumerate}
$\operatorname{Ch}_{even}$ and $\operatorname{Ch}_{odd}$ are called the Chern characters.
\end{dfn}

Let $X$ be a pointed space. The inclusion map from the basepoint to $X$ is denoted by $i$. From the definition of the Chern character, it follows that $i^{*}\operatorname{Ch}(\hat{A}) = \operatorname{Ch}(i^{*}A)$. Therefore, the Chern character induces a homomorphism from the reduced $K$-group to the reduced singular cohomology:
\begin{align*}
&\operatorname{Ch}_{even}\colon \tilde{K}(X)\to \tilde{H}^{even}(X;\mathbb{Q}),&\operatorname{Ch}_{odd}\colon \tilde{K}^1(X)\to \tilde{H}^{odd}(X;\mathbb{Q}).
\end{align*}

\begin{prop}\label{prop:Commute}
Let $X$ be a \emph{pointed compact Hausdorff space}.
The suspension isomorphism in reduced singular cohomology and the Chern character commute.
More precisely, the following two diagrams commute.

\medskip
\noindent{\rm (i) Even-to-odd:}
\begin{center}
\begin{tikzpicture}[thick]
\node (a) at (0, 2.8) {$\tilde{K}^{0}(X)$};
\node (x) at (3.6, 2.8) {$\tilde{H}^{\mathrm{even}}(X;\mathbb{Q})$};
\node (b) at (0, 0) {$\tilde{K}^{1}(SX)$};
\node (y) at (3.6, 0) {$\tilde{H}^{\mathrm{odd}}(SX;\mathbb{Q})$};

\draw[->] (a) -- node[midway, above] {\(\scriptstyle \operatorname{Ch}_{\mathrm{even}}\)} (x);
\draw[->] (x) -- node[midway, right] {\(\scriptstyle \delta\)} (y);
\draw[->] (a) -- node[midway, left] {\(\scriptstyle AS\)} (b);
\draw[->] (b) -- node[midway, below] {\(\scriptstyle \operatorname{Ch}_{\mathrm{odd}}\)} (y);
\end{tikzpicture}
\end{center}

\medskip
\noindent{\rm (ii) Odd-to-even:}
\begin{center}
\begin{tikzpicture}[thick]
\node (a) at (0, 2.8) {$\tilde{K}^{1}(X)$};
\node (x) at (3.6, 2.8) {$\tilde{H}^{\mathrm{odd}}(X;\mathbb{Q})$};
\node (b) at (0, 0) {$\tilde{K}^{0}(SX)$};
\node (y) at (3.6, 0) {$\tilde{H}^{\mathrm{even}}(SX;\mathbb{Q})$};
\draw[->] (a) -- node[midway, above] {\(\scriptstyle \operatorname{Ch}_{\mathrm{odd}}\)} (x);
\draw[->] (x) -- node[midway, right] {\(\scriptstyle \delta\)} (y);
\draw[->] (a) -- node[midway, left] {\(\scriptstyle AS\)} (b);
\draw[->] (b) -- node[midway, below] {\(\scriptstyle \operatorname{Ch}_{\mathrm{even}}\)} (y);
\end{tikzpicture}
\end{center}
\end{prop}
\begin{proof}
We prove (i). The proof of (ii) is identical by shifting degrees.

\smallskip
\noindent\textbf{Step 1 (Naturality of the slant product).}
Let $Y,Z$ be compact Hausdorff spaces and $f:Y\to Z$ continuous.
For any $\alpha\in H^{m}(S^{1}\times Z;\mathbb{Q})$, by naturality of the slant product
(see \cite{Hatcher}), we have
\[
(id_{S^{1}}\times f)^{*}(\alpha)/[S^{1}] = f^{*}(\alpha/[S^{1}]).
\]

\smallskip
\noindent\textbf{Step 2 (Suspension in cohomology).}
Let $\delta:\tilde H^{m}(X;\mathbb{Q})\to \tilde H^{m+1}(SX;\mathbb{Q})$ be the reduced suspension isomorphism.
Using the identification $SX\simeq S^{1}\wedge X$, $\delta$ is represented by the standard construction
via $S^{1}$ (see \cite{Hatcher}).

\smallskip
\noindent\textbf{Step 3 (Computation).}
Let $\hat A:X\to \mathcal{F}_{0}$ represent a class in $\tilde K^{0}(X)$.
Then $\operatorname{ch}_{n}(\hat A)=\hat A^{*}\operatorname{ch}^{\mathrm{univ}}_{n}$.
On the other hand, the suspended $K$-class $AS(\hat A)\in \tilde K^{1}(SX)$ is represented by a map
$AS(\hat A):SX\to \mathcal{F}_{1}$, and by Definition~\ref{dfn:univ-chern}(2),
\[
\operatorname{ch}_{n+\frac12}(AS(\hat A))
=
AS(\hat A)^{*}\Bigl( (\overline{\mathrm{AS}}^{\,*}\operatorname{ch}^{\mathrm{univ}}_{n+1})/[S^{1}] \Bigr).
\]
By Step 1, this equals
\[
\Bigl( (id_{S^{1}}\times AS(\hat A))^{*}\,\overline{\mathrm{AS}}^{\,*}\operatorname{ch}^{\mathrm{univ}}_{n+1}\Bigr)/[S^{1}].
\]
Since both $AS$ and $\overline{\mathrm{AS}}$ are induced from the Atiyah--Singer deformation
$A\cos t-\hat\gamma_{2}\sin t$, the above class coincides with the cohomological suspension
$\delta(\hat A^{*}\operatorname{ch}^{\mathrm{univ}}_{n})=\delta(\operatorname{ch}_{n}(\hat A))$.
Thus $\operatorname{Ch}_{\mathrm{odd}}(AS(\hat A))=\delta(\operatorname{Ch}_{\mathrm{even}}(\hat A))$, proving (i).
\end{proof}

\begin{rem}\label{rem:iterated-AS}
The Atiyah--Singer suspension isomorphism can be iterated, yielding isomorphisms
$\tilde K^{-n-1}(X)\cong \tilde K^{-n}(SX)$ for all $n$ (see \S\ref{subsec:Description of the generator}).
Since the parity alternates under suspension, one obtains two types of commutative diagrams:
even-to-odd and odd-to-even. Proposition~\ref{prop:Commute} records both types, and the
iterated commutativity follows by repeating the case of one suspension.
\end{rem}

\begin{dfn}
Let $X$ be a compact Hausdorff space. Fix a point $x\in X$.
Assume there exists a closed neighborhood $D^{n}\subset X$ of $x$ which is homeomorphic
to the closed $n$-disk, and such that $x$ corresponds to the center.
Then, the push-forward
\[
i_*^{K} \colon K^0(\{x\}) \to K^n(X)
\]
is defined from the following diagram:
\[
\begin{tikzpicture}[thick]
\node (a) at (0, 0) {$K^0(\{x\})$};
\node (b) at (3.2, 0) {$K^n(D^{n}, \partial D^{n})$};
\node (c) at (7.0, 0) {$K^n(X, X \setminus \mathrm{int}(D^n))$};
\node (d) at (10.8, 0) {$K^n(X, X \setminus \{x\})$};
\node (e) at (10.8, -3) {$K^n(X)$};

\draw[->] (a) --
  node[midway, above] {\(\scriptstyle \mathrm{AS}^{n}\)}
  node[midway, below] {\(\scriptstyle \cong\)}
(b);
\draw[->] (b) --
  node[midway, above] {\(\scriptstyle j^*\)}
  node[midway, below] {\(\scriptstyle \cong\)}
(c);
\draw[->] (c) --
  node[midway, above] {\(\scriptstyle k^*\)}
  node[midway, below] {\(\scriptstyle \cong\)}
(d);

\draw[->] (d) -- node[midway, right] {\(\scriptstyle \ell^*\)} (e);
\draw[->] (a) -- node[pos=0.55, above left] {\(\scriptstyle i_*^K\)} (e);
\end{tikzpicture}
\]
Here $\mathrm{AS}^{n}$ denotes the $n$-fold iteration of the Atiyah--Singer suspension
isomorphism, which identifies $K^{0}(\{x\})\cong K^{n}(D^{n},\partial D^{n})$.
Moreover,
\[
j\colon (D^n,\partial D^n)\to (X, X \setminus \mathrm{int}(D^n)),
\]
\[
k\colon (X, X \setminus \mathrm{int}(D^n))\to (X, X \setminus \{x\}),
\]
and
\[
\ell\colon (X,\emptyset)\to (X, X \setminus \{x\})
\]
are the canonical inclusion maps.
We call $i_*^{K}$ the push-forward.
\end{dfn}

\begin{prop}\label{prop:chern character and push forward}
Let $X$ be a compact Hausdorff space. Suppose there exists a closed neighborhood of some
$x \in X$ that is homeomorphic to the $n$-dimensional closed disk $D^n$.
Then the push-forward in $K$-theory and the push-forward in singular cohomology commute.
More precisely, the following two cases occur.

\begin{enumerate}
\item If $n=2m+1$ is odd, then the following diagram commutes.
\begin{center}
\begin{tikzpicture}[thick]
\node (a) at (0, 2.8) {$K^{0}(\{x\})$};
\node (x) at (3.6, 2.8) {$H^{\mathrm{even}}(\{x\};\mathbb{Q})$};
\node (b) at (0, 0) {$K^{1}(X)$};
\node (y) at (3.6, 0) {$H^{\mathrm{odd}}(X;\mathbb{Q})$};

\draw[->] (a) -- node[midway, above] {\(\scriptstyle \operatorname{Ch}_{\mathrm{even}}\)} (x);
\draw[->] (x) -- node[midway, right] {\(\scriptstyle i_{*}^{H}\)} (y);
\draw[->] (a) -- node[midway, left] {\(\scriptstyle i_{*}^{K}\)} (b);
\draw[->] (b) -- node[midway, below] {\(\scriptstyle \operatorname{Ch}_{\mathrm{odd}}\)} (y);
\end{tikzpicture}
\end{center}

\item If $n=2m$ is even, then the following diagram commutes.
\begin{center}
\begin{tikzpicture}[thick]
\node (a) at (0, 2.8) {$K^{0}(\{x\})$};
\node (x) at (3.6, 2.8) {$H^{\mathrm{even}}(\{x\};\mathbb{Q})$};
\node (b) at (0, 0) {$K^{0}(X)$};
\node (y) at (3.6, 0) {$H^{\mathrm{even}}(X;\mathbb{Q})$};

\draw[->] (a) -- node[midway, above] {\(\scriptstyle \operatorname{Ch}_{\mathrm{even}}\)} (x);
\draw[->] (x) -- node[midway, right] {\(\scriptstyle i_{*}^{H}\)} (y);
\draw[->] (a) -- node[midway, left] {\(\scriptstyle i_{*}^{K}\)} (b);
\draw[->] (b) -- node[midway, below] {\(\scriptstyle \operatorname{Ch}_{\mathrm{even}}\)} (y);
\end{tikzpicture}
\end{center}
\end{enumerate}

Here, $i_*^K$ is the push-forward in $K$-theory defined above by using $\mathrm{AS}^{n}$,
and $i_*^H$ is the push-forward in singular cohomology.
\end{prop}

\begin{proof}
By definition, $i_*^K$ is obtained by composing the iterated Atiyah--Singer suspension
isomorphism $\mathrm{AS}^{n}$ with excision and the canonical map from relative $K$-groups
to absolute $K$-groups.

Proposition~\ref{prop:Commute} gives the commutativity of the Chern character with one
suspension step, and Remark~\ref{rem:iterated-AS} implies that the same commutativity
holds for the iterated suspension $\mathrm{AS}^{n}$.
Since the excision isomorphism and the natural maps of pairs are compatible with pullbacks
in singular cohomology, the Chern character commutes with the whole construction of the
push-forward.

If $n=2m+1$, the parity changes from even to odd, and we obtain the first diagram.
If $n=2m$, the parity returns to even, and we obtain the second diagram.
\end{proof}


\section{Expression of Chern character by using Fermi points}
\label{sec:chern character by Fermi points}


\subsection{Approximation of family of Fredholm operators}
\label{sub:Approximation of family of Fredholm operators}

\begin{rem}[Roadmap of \S\ref{sec:chern character by Fermi points}]\label{rem:roadmap-fermi}
The goal of this section is to express the (even/odd) Chern character as a sum of local
contributions at Fermi points.
We first construct a finite-dimensional approximation $(\hat{\mathcal{H}},\hat A)_{<\mu}$ (resp.\ $(\mathcal H,A)_{<\mu}$),
and define sign coordinates and $\operatorname{sign}(x)$ via the Jacobian sign.
Next, we fix the normalization of the canonical isomorphisms
$K(D^{m},\partial D^{m})\cong \mathbb Z$ by declaring that the \emph{standard} local model
built from the \emph{standard} Clifford representation fixed in Convention~\ref{conv:std-cliff}
represents $1\in\mathbb Z$ (cf.\ \S\ref{subsec:Description of the generator} and Convention~\ref{conv:odd-generator} in odd degrees).
With this normalization, the main theorems follow by excision and by summing up the local
classes $i_*^K(1)=\operatorname{sign}(x)[A]$ (resp.\ $i_*^K(1)=\operatorname{sign}(x)[\hat A]$) over Fermi points.
\end{rem}

In this section, we introduce a finite-dimensional approximation of a Fredholm operator following \cite{G2}.

Let $X$ be a topological space, $\hat{\mathcal{H}}$ a $\mathbb{Z}_2$-graded separable infinite-dimensional Hilbert space, and 
$\hat{A} \colon X\to \mathcal{F}_0(\hat{\mathcal{H}})$ a continuous map with $x_0 \in X$ and $0\in Spec(A_{x_0}).$

\begin{prop} \label{prop:mu exist}
There exists a positive real number $\mu >0$ and an open neighborhood U of $x_0$ that satisfy the following three properties:
\begin{enumerate}
\item $Spec(\hat{A}_x)\cap [0, \mu)$ is a finite set.
\item For $\lambda \in Spec(\hat{A}_x)\cap [0, \mu)$, the eigenspace corresponding to $\lambda$ is finite-dimensional.
\item For any $x \in U$, $\dim(\bigoplus_{|\lambda|<\mu}\ker(\hat{A}_{x_0}-\lambda))=\dim(\bigoplus_{\lvert \lambda\rvert<\mu}\ker(\hat{A}_{x}-\lambda)).$
\end{enumerate}
\end{prop}

\begin{proof}
Because $\hat{A}_{x_0}$ is a Fredholm operator, $0\in Spec(\hat{A}_{x_0})$ is a discrete spectrum, so $[0, \mu)\cap Spec(\hat{A}_{x_0})$
is a finite set and there exists a positive real number $\mu>0$ such that any $\lambda \in Spec(\hat{A}_{x_0})\cap [0, \mu)$ has finite multiplicity.
Therefore, due to the continuity of the eigenvalues, we can obtain an open neighborhood of $x_0$ that fulfills the assertion of the proposition.
\end{proof}

In what follows, we assume that $U$ and $\mu>0$ are given by the above Proposition \ref{prop:mu exist}.
\begin{dfn}
\begin{enumerate}
\item For $x\in U$, we define $(\mathcal{\hat{H}}, \hat{A}_x)_{<\mu}$ to be the direct sum of eigenspaces of $\hat{A}_x$ whose eigenvalues $\lambda$ satisfy $|\lambda|<\mu$ :
\begin{align*}
(\mathcal{\hat{H}}, \hat{A}_x)_{<\mu}=\bigoplus_{|\lambda|<\mu}\ker(\hat{A}_{x}-\lambda).
 \end{align*}
\item We define the family of vector spaces $(\hat{\mathcal{H}}, \hat{A})$ over $U$ as the set whose vector space at each point $x$ is $(\hat{\mathcal{H}}, \hat{A}_x)$:
\begin{align*}
(\mathcal{\hat{H}}, \hat{A})_{<\mu}=\bigcup_{x \in U}(\mathcal{\hat{H}}, \hat{A}_x)_{<\mu}\subset U\times \mathcal{\hat{H}}.
\end{align*}
\item We define $F$ as the set of all $x \in X$ such that $\hat{A}_x$ has $0$ in its spectrum:
\begin{align*}
F = \{ x \in X \mid 0 \in \mathrm{Spec}(\hat{A}_{x}) \}.
\end{align*}
We call this $F$ the Fermi set with respect to $\hat{A}$. Similarly, for an ungraded Fredholm operator, we define the Fermi set in the same manner.
\end{enumerate}
\end{dfn}

 The purpose of this subsection is to establish:
 
\begin{prop}\label{prop:apporoximation vector bundle}
$(\hat{\mathcal{H}}, \hat{A})_{<\mu}\subset U\times \hat{\mathcal{H}}$ is a $\mathbb{Z}_2$-graded complex vector bundle of finite rank.
 \end{prop}

 \begin{proof}
Fix $x \in U$ and we will show that the family of vector spaces $(\hat{\mathcal{H}},\hat{A})_{<\mu}$ can be locally trivialized in some open neighborhood of $x$.

For each \( y \in U \), we define the projection \( \pi_y :\hat{\mathcal{H}} \to (\hat{\mathcal{H}}, \hat{A}_y)_{<\mu} \). Let $i:(\hat{\mathcal{H}}, \hat{A}_x)_{<\mu} \hookrightarrow \hat{\mathcal{H}}$ be an inclusion map.
We have $(1_{\hat{\mathcal{H}}}-\pi_{x})\circ(1_{\hat{\mathcal{H}}}-\pi_{x})\circ i=1_{(\hat{\mathcal{H}}, \hat{A}_x)_{<\mu}^{\perp}}:(\hat{\mathcal{H}}, \hat{A}_x)_{<\mu}^{\perp} \to (\hat{\mathcal{H}}, \hat{A}_x)_{<\mu}^{\perp}$,
and noting that the set of invertible bounded operators is an open set in $\mathcal{B}(\hat{\mathcal{H}})$,
we obtain the following assertion.

For any $y \in U$, there exists an open neighborhood $V \subset U$ of $x$ such that $(1_{\hat{\mathcal{H}}}-\pi_{x})\circ(1_{\hat{\mathcal{H}}}-\pi_{y})\circ i:(\hat{\mathcal{H}}, \hat{A}_x)_{<\mu}^{\perp} \to (\hat{\mathcal{H}}, \hat{A}_x)_{<\mu}^{\perp}$ is a linear isomorphism.

Then, the continuous map $\overline{\pi_{x}}:(\hat{\mathcal{H}}, \hat{A})_{<\mu}|_{V}\to V\times(\hat{\mathcal{H}}, \hat{A}_x)_{<\mu}$ induced by $\pi_{x}$ gives a local trivialization in some open neighborhood of $x$. In fact, for any $y \in V$, the following diagram commutes.

\begin{tikzpicture}[thick]
\node (a) at (0, 0) {$0$}; 
\node (b) at (2.7, 0) {$(\hat{\mathcal{H}}, A_x)_{<\mu}^{\perp}$};
\node (c) at (5.4, 0) {$\hat{\mathcal{H}}$};   
\node (d) at (8.1, 0) {$(\hat{\mathcal{H}}, A_x)_{<\mu}$};
\node (e) at (10.8, 0) {$0$};  
\node (f) at (5.4, 2.7) {$\ker\left( (1_{\hat{\mathcal{H}}}-\pi_{x})\circ(1_{\hat{\mathcal{H}}}-\pi_{y}) \right)$};
\node (g) at (5.4, -2.7) {$(\hat{\mathcal{H}}, A_x)_{<\mu}^{\perp}$}; 
\node (h) at (12.3, 0) {$(\text{exact})$}; 

\draw[->] (a) -- (b);
\draw[->] (b) -- node[midway, above] {\(\scriptstyle i\)} (c);
\draw[->] (c) -- node[midway, above] {\(\scriptstyle \pi_{x}\)} (d);
\draw[->] (d) -- (e);

\draw[->] (f) -- node[midway, right] {\(\scriptstyle i\)} (c);

\draw[->] (c) -- node[midway, right] {\(\scriptstyle (1_{\hat{\mathcal{H}}}-\pi_{x})\circ(1_{\hat{\mathcal{H}}}-\pi_{y})\)} (g);

\draw[->] (b) -- node[midway, left] {\(\scriptstyle \cong\)} (g);
\draw[->] (f) -- node[pos=0.6, above right] {\(\scriptstyle \pi_{x} \circ i\)} (d);
\end{tikzpicture}\\

From the commutative diagram $\pi_x\circ i : \ker((1_{\hat{\mathcal{H}}}-\pi_{x})\circ(1_{\hat{\mathcal{H}}}-\pi_{y})) \stackrel{\cong}{\longrightarrow} (\hat{\mathcal{H}}, \hat{A}_x)_{<\mu}$ is an isomorphism.

By definition, \((\hat{\mathcal{H}}, \hat{A}_y)_{<\mu} = \ker(1_{\hat{\mathcal{H}}} - \pi_y) \subset \ker((1_{\hat{\mathcal{H}}} - \pi_x) \circ (1_{\hat{\mathcal{H}}} - \pi_y))\) and \(\dim((\hat{\mathcal{H}}, \hat{A}_x)_{<\mu}) = \dim((\hat{\mathcal{H}}, \hat{A}_y)_{<\mu}) < \infty\) hold. Therefore, \(\ker((1_{\hat{\mathcal{H}}} - \pi_x) \circ (1_{\hat{\mathcal{H}}} - \pi_y)) = \ker(1_{\hat{\mathcal{H}}} - \pi_y)\) holds.
                                                         
Hence, $\overline{\pi_{x}}:(\hat{\mathcal{H}}, \hat{A})_{<\mu}|_{V}\to V\times(\hat{\mathcal{H}}, \hat{A}_x)_{<\mu}$ is a local trivialization.
   \end{proof}

\begin{dfn}
By Proposition~\ref{prop:apporoximation vector bundle}, we define a finite-dimensional approximation
$\hat{A}_{<\mu} : (\hat{\mathcal{H}}, \hat{A})_{<\mu} \to (\hat{\mathcal{H}}, \hat{A})_{<\mu}$
of $\hat{A}$ by setting \(\hat{A}_{<\mu}(x) = \hat{A}(x)\). In an ungraded case, we similarly define \((\mathcal{H}, A)_{<\mu}\) and $A_{<\mu} : (\mathcal{H}, A)_{<\mu} \to (\mathcal{H}, A)_{<\mu}$.
\end{dfn}


\subsection{Sign coordinate}
\label{sub:Sign coorinate}

\begin{dfn}\label{dfn:even sign coordinates}
Let $k$ be a non-negative integer. Let $X$ be a $2k$-dimensional compact, oriented, differentiable manifold, and let $\hat{A} \colon X \to \mathcal{F}_0(\mathcal{\hat{H}})$ be a continuous map. We fix $\mu > 0$ and an open set $U \subset X$, as described in Proposition \ref{prop:mu exist}, such that $U$ contains $F$ and the space $(\mathcal{\hat{H}}, \hat{A}_x)_{<\mu}$ is defined for all $x \in U$. We assume that $\dim(\mathcal{\hat{H}}, \hat{A}_x)_{<\mu} = 2^k$ for all $x \in F$. We also fix the \emph{standard} $\mathbb{Z}_2$-graded irreducible $Cl_{2k}$ representation $\hat S^{(2k)}$ (and its generators) as in Convention~\ref{conv:std-cliff}, and we identify $(\mathcal{\hat{H}}, \hat{A}_x)_{<\mu}$ with $\hat S^{(2k)}$ as a graded $Cl_{2k}$ representation.
\begin{enumerate}
\item We say that a coordinate $(V;a_1,...,a_{2k})$ $(x\in V \subset U)$ of $X$ is an (even) sign coordinate of $x$ when $\hat{A}_{<\mu}|_V \simeq \sum_{i=1}^{2k} a_i\,\hat{\gamma}^{(2k)}_i:(\mathcal{\hat{H}}, \hat{A})_{<\mu}|_V\to (\mathcal{\hat{H}}, \hat{A})_{<\mu}|_V.$\\
Here, $\simeq$ means \emph{fiberwise homotopic relative to $V\setminus\{x\}$} (see Remark~\ref{rem:fiberwise-homotopy}).
\item When a sign coordinate $(V;a_{1},...,a_{2k})$ $(V\subset U)$ can be taken in $x \in F$, define the following sign.
\begin{align*}
\operatorname{sign}(x)=
\frac{\det\left(\frac{\partial(a_{1},\dots,a_{2k})}{\partial (x_{1},\dots,x_{2k})}\right)}
{\left|\det\left(\frac{\partial(a_{1},\dots,a_{2k})}{\partial (x_{1},\dots,x_{2k})}\right)\right|}
\in \{-1,1\}.
\end{align*}
Here $(V;x_{1},...,x_{2k})$ is an oriented coordinate of $X$, and $\frac{\partial(a_{1},...,a_{2k})}{\partial (x_{1},...,x_{2k})}$ is the Jacobian.
\end{enumerate}
\end{dfn}

\begin{dfn}\label{dfn:odd sign coordinates}
Let $k$ be a non-negative integer. Let $X$ be a ($2k+1$)-dimensional compact, oriented, differentiable manifold, and let $A\colon X\to Fred_{sa}(\mathcal{H})$ be a continuous map. We fix $\mu > 0$ and an open set $U \subset X$, as described in Proposition \ref{prop:mu exist}, such that $U$ contains $F$ and the space $(\mathcal{H}, A_x)_{<\mu}$ is defined for all $x \in U$. We assume that $\dim(\mathcal{H}, A_x)_{<\mu} = 2^k$ for all $x \in F$. We also fix the \emph{standard} ungraded irreducible $Cl_{2k+1}$ representation $S^{(2k+1)}$ (and its generators) as defined in Convention~\ref{conv:std-cliff}, and we identify $(\mathcal{H}, A_x)_{<\mu}$ with $S^{(2k+1)}$ as an ungraded $Cl_{2k+1}$ representation.
\begin{enumerate}
\item We say that a coordinate $(V;a_1,...,a_{2k+1})$ $(x\in V \subset U)$ of $X$ is an (odd) sign coordinate of $x$ when $A_{<\mu}|_V \simeq \sum_{i=1}^{2k+1} a_i\,\gamma_i:(\mathcal{H},A)_{<\mu}|_V\to (\mathcal{H},A)_{<\mu}|_V.$\\
Here, $\simeq$ means \emph{fiberwise homotopic relative to $V\setminus\{x\}$} (see Remark~\ref{rem:fiberwise-homotopy}).
\item When a sign coordinate $(V;a_{1},...,a_{2k+1})$ $(V\subset U)$ can be taken in $x \in F$, define the following sign.
\begin{align*}
\operatorname{sign}(x)=
-\,
\frac{\det\left(\frac{\partial(a_{1},\dots,a_{2k+1})}{\partial (x_{1},\dots,x_{2k+1})}\right)}
{\left|\det\left(\frac{\partial(a_{1},\dots,a_{2k+1})}{\partial (x_{1},\dots,x_{2k+1})}\right)\right|}
\in \{-1,1\}.
\end{align*}
Here $(V;x_{1},...,x_{2k+1})$ is an oriented coordinate of $X$, and $\frac{\partial(a_{1},...,a_{2k+1})}{\partial (x_{1},...,x_{2k+1})}$ is the Jacobian.
\end{enumerate}
\end{dfn}

\begin{rem}\label{rem:fiberwise-homotopy}
In Definitions~\ref{dfn:even sign coordinates} and~\ref{dfn:odd sign coordinates}, the symbol ``$\simeq$'' means a \emph{fiberwise homotopy with a relative invertibility condition}.
More precisely, after choosing a local trivialization
\[
(\mathcal{\hat{H}}, \hat{A})_{<\mu}\big|_V \cong V \times \C^{m}
\quad\text{(resp.\ }(\mathcal{H}, A)_{<\mu}\big|_V \cong V \times \C^{m}\text{)},
\]
we regard $\hat{A}_{<\mu}\big|_V$ (resp.\ $A_{<\mu}\big|_V$) as a continuous map $V \to \mathrm{End}(\C^{m})$.
We say that $\hat{A}_{<\mu}\big|_V$ (resp.\ $A_{<\mu}\big|_V$) is \emph{fiberwise homotopic relative to $V\setminus\{x\}$}
to the model map $\sum a_i \gamma_i$ if there exists a continuous homotopy
\[
H: V \times [0,1] \longrightarrow \mathrm{End}(\C^{m})
\]
such that
\[
H(\cdot,0)=A_{<\mu}\big|_V,\qquad
H(\cdot,1)=\sum a_i \gamma_i,
\]
and moreover
\[
H(v,t)\in \mathrm{GL}(\C^{m})\quad \text{for all } v\in V\setminus\{x\}\text{ and } t\in[0,1].
\]
Equivalently, $H$ is a homotopy of maps of pairs
\[
(V,\,V\setminus\{x\}) \to (\mathrm{End}(\C^{m}),\,\mathrm{GL}(\C^{m})).
\]
\end{rem}

\begin{rem}
In even degrees, the Jacobian sign in Definition~\ref{dfn:even sign coordinates} matches the standard generator of
$K^{0}(D^{2k},\partial D^{2k})$. In odd degrees, we include the extra minus sign in
Definition~\ref{dfn:odd sign coordinates}(2) so that the local generator used in the push-forward normalization
agrees with the sign convention in the statement of the odd main theorem.
\end{rem}

\begin{rem}[Effect of changing the irreducible Clifford  representation]\label{rem:choice-irr- representation}
In odd degrees there are two inequivalent ungraded irreducible $Cl_{2k+1}$ representations
(e.g.\ $Cl_1$: $\gamma_1=\pm 1$; see Remark~\ref{rem:std-ungraded-Cl1-Cl3}).
Replacing the chosen representation (or equivalently conjugating the generators by an orientation-reversing
change) may change the identified generator of
$K^{-1}(D^{2k+1},\partial D^{2k+1})\cong\mathbb Z$ by an overall sign.
Throughout this paper we fix the \emph{standard} choice by Convention~\ref{conv:std-cliff}
(and hence the normalization of $K(D^m,\partial D^m)\cong\mathbb Z$),
and we fix the odd generator by Convention~\ref{conv:odd-generator}.
With these conventions, the sign of the local invariant is determined only by the orientation
of sign coordinates (the Jacobian sign) in Definitions~\ref{dfn:even sign coordinates} and~\ref{dfn:odd sign coordinates}.
\end{rem}

\begin{dfn}
An element $x \in F$ of the Fermi set is called a Fermi point if a sign coordinate can be defined for it.
\end{dfn}
                                                         
Let $U_1$ and $U_2$ be open neighborhoods of the origin in $\mathbb{R}^n$. Consider a diffeomorphism $f\colon U_1 \to U_2$ satisfying $f(0)=0$ and $\det{T}>0$, where $T=Jf(0)$ is the Jacobian matrix of $f$ at the origin. The image of $T$, denoted $Im(T)$, is the image of the linear map $T$. We define $f_1\colon Im(T) \to U_2$ by $f_1=f \circ T^{-1}$, so that $f=f_1 \circ T$. Here, $Jf$ and $Jf_1$ denote the Jacobian matrix of $f$ and $f_1$, respectively.
\begin{lem}\label{lem:decomposition-of-f}
$Jf_1(0)=E_n$, where $E_n$ is the $n\times n$ identity matrix.
\end{lem}

\begin{proof}
This is clear from the definition of $f_1$ and the differentiation of composite and inverse functions.
\end{proof}

\begin{lem}\label{lem:4.9}
\begin{enumerate}
\item Let $U_1,U_2\subset \mathbb{R}^n$ be open neighborhoods of 0, $f:U_1 \to U_2$ a smooth map such that $f(0)=0$ and $Jf(0)=E_n$.
Then $f|_{D(r)}:D(r)\to D(r^{\prime})$ is a diffeomorphism and there exists a positive real numbers $r, r^{\prime}>0\,(r^{\prime}\ge r>0) $ such that $f|_{D(r)}\simeq i$.
Here, $D(r)$ is an open ball of radius $r$ centered at the origin, and $i:D(r)\to D(r^{\prime})$
is the inclusion map.
\item Let \( T \in GL^{+}(n;\mathbb{R}) \) be a linear transformation. 
Then for any \( r > 0 \), there exists \(0< r'' < r \) such that 
\[
\operatorname{Im}(T|_{D(r'')}) \subset D(r),
\]
and the map \( T \colon D(r'') \to D(r) \) is homotopic to the inclusion \( i \colon D(r'') \hookrightarrow D(r) \); that is, \( T \simeq i \).
\end{enumerate}
\end{lem}

\begin{proof}
\begin{enumerate}
\item
For any real number $t \in \mathbb{R}$, we define a homotopy $f_{t}\colon U_1 \to \mathbb{R}^n$
by the formula
\begin{align*}
f_t(x)=tx+(1-t)f(x).
\end{align*} 
This definition gives $f_0=f$ and $f_1=\mathrm{id}_{U_1}$, and in particular $f_t(0)=0$ for all $t\in[0,1]$.  

Next, we define a map $F\colon U_1\times [0,1] \to \mathbb{R}^n \times [0,1]$ by $F(x,t)=(f_t(x),t)$.
By definition, the image of the point $(0,t)$ under the map $F$ is $F(0,t)=(0,t)$.
Furthermore, the Jacobian matrix of $F$ at $(0,t)$ has the form  
\begin{align*}
J F(0,t)= 
\begin{pmatrix}
E_n & * \\
0 & 1
\end{pmatrix},
\end{align*}
where the asterisk represents some entries whose explicit form is irrelevant in this context.  
Since the Jacobian matrix is non-singular, the inverse function theorem guarantees the existence of positive real numbers $r_t>0$ and $\delta_t>0$ such that the restricted map 
\begin{align*}
F|_{D(r_t) \times [t-\delta_t, t+\delta_t]}\colon D(r_t) \times [t-\delta_t, t+\delta_t] \to \mathbb{R}^n \times [0,1]
\end{align*}
is a diffeomorphism onto its image.
Due to the compactness of $[0,1]$, there exists a monotonically increasing sequence 
$t_1,...,t_n\in [0,1]$ such that $\bigcup_{i=1}^{n}[t_i-\delta_{t_i},t_i+\delta_{t_i}]=[0,1]$.
If we set $r = \min \{r_{t_1}, \dots, r_{t_n}\}$, then $F|_{D(r) \times [0,1]} : D(r) \times [0,1] \to \mathbb{R}^n \times [0,1]$ is a diffeomorphism onto its image.
The subspace $F(D(r)\times [0,1])\subset \mathbb{R}^n\times [0,1]$
is compact, so there exists a positive real number $r^{\prime}>0$ such that $F(D(r)\times [0,1])\subset D(r^{\prime})\times [0,1]$.
($0<r\leq r^{\prime}$ since $F(D(r)\times \{1\})=D(r)\times \{1\} \subset D(r^{\prime})\times \{1\}$.)
Therefore, $F$ is a homotopy that gives $f|_{D(r)}\simeq i$.
\item Since \(GL^{+}(n;\mathbb{R})\) is path-connected, the path connecting \(T\) and \(E_n\) gives the homotopy.
\end{enumerate}
\end{proof}

Given a $\mathbb{Z}_2$-graded representation $\hat{S}$ of $Cl_{2k}$, let $\hat{\mu}_{\hat{S}}\colon D^{2k}\to \mathrm{End}(\hat{S})$ be defined by
\begin{align*}
\hat{\mu}_{\hat{S}}(x_1,\dots,x_{2k})=\sum_{i=1}^{2k} x_i\,\hat{\gamma}_i.
\end{align*}
The orthogonal complement of the embedding $\hat{S} \subset \mathcal{\hat{H}}_{2k}$ as a graded representation is denoted as $\hat{S}^{\perp}$, and the action of $Cl_{2k+1}$, which is an extension of the action of $Cl_{2k}$, on $\hat{S}^{\perp} \cong \mathcal{\hat{H}}_{2k}$ is denoted as $\hat{\gamma}^{\perp}$.

\begin{convention}\label{conv:local-model-standard-choice}
In Theorems~\ref{thm:4.10} and~\ref{thm:4.11}, the subspace $\hat{S}_{+}$ is taken to be the standard representation fixed in Convention~\ref{conv:std-cliff} (with the corresponding standard Clifford generators).
\end{convention}

We now record a local model criterion showing that the standard local form represents the generator
$1\in\mathbb Z$ under the normalization fixed in \S\ref{subsec:Description of the generator} and Convention~\ref{conv:std-cliff}.
This will be used implicitly in the proof of the main theorems via excision and the push-forward.

\begin{thm}\label{thm:4.10}
Suppose that a continuous mapping 
\begin{align*}
\hat{A} \colon(D^{2k},\partial D^{2k})\to (\mathcal{F}_{0}(\mathcal{\hat{H}}),\mathcal{F}^{*}_{0}(\mathcal{\hat{H}}))
\end{align*}
satisfies the following assumptions:
\begin{enumerate}
\item The Hilbert space \( \hat{\mathcal{H}} \) admits a direct sum decomposition 
\( \hat{\mathcal{H}} = \hat{S}_{+} \oplus \hat{S}_{+}^{\perp} \), and accordingly, 
the operator \( \hat{A}(x) \) can be decomposed as 
\[
\hat{A}(x) = \hat{a}(x) + \hat{a}^{\perp}(x),
\]
where \( \hat{a}(x) \colon \hat{S}_{+} \to \hat{S}_{+} \subset \hat{\mathcal{H}} \) 
and \( \hat{a}^{\perp}(x) \colon \hat{S}_{+}^{\perp} \to \hat{\mathcal{H}} \).\\
\item For any $x \in D^{2k}$, we have $\hat{a}^{\perp}(x)\in \mathcal{F}^{*}_{0}(\mathcal{\hat{\mathcal{H}}})$.\\
\item There exists a smooth mapping $f=(f_1,...,f_{2k})\colon D^{2k}\to \mathbb{R}^{2k}$ such that:\\
\begin{itemize}
    \item  For any $x \in D^{2k}$, the map $\hat{a}$ admits the expression $\hat{a}(x)=\sum_{i=1}^{2k}f_i(x)\hat{\gamma}_i$.
    \item $f^{-1}(\{0\})=\{0\}$.
    \item $\det Jf(0)>0$.
\end{itemize}
\end{enumerate}
In this case, under an isomorphism $K^{0}(D^{2k},\partial D^{2k})\cong \mathbb{Z}$, the map $\hat{A}$ represents the generator $1 \in \mathbb{Z}$. Here the isomorphism $K(D^{m},\partial D^{m})\cong \mathbb{Z}$ is normalized by the generator described in Section~\ref{subsec:Description of the generator}, together with Convention~\ref{conv:std-cliff}; namely, the standard local model built from the standard Clifford representation represents $1\in\mathbb{Z}$.
\end{thm}
                                                   
\begin{proof}
We prove that $[\hat A]\in K^{0}(D^{2k},\partial D^{2k})$ represents the generator
$1\in\mathbb Z$ under the fixed identification $K^{0}(D^{2k},\partial D^{2k})\cong\mathbb Z$
given in \S\ref{subsec:Description of the generator} together with Convention~\ref{conv:std-cliff},
and with the choice of $\hat S_{+}$ fixed in Convention~\ref{conv:local-model-standard-choice}.

\smallskip
\noindent\textbf{Step 1: Reduction to a small ball.}
Since $f^{-1}(\{0\})=\{0\}$, the operator $\hat A(x)$ is invertible on
$D^{2k}\setminus \mathrm{int}(D(r'))$ for $r'>0$ sufficiently small.
Hence $\hat A$ defines a relative class
\[
[\hat A]\in K^{0}\bigl(D^{2k},\,D^{2k}\setminus \mathrm{int}(D(r'))\bigr).
\]
By excision (and the canonical identification of pairs)
\[
\bigl(D^{2k},\,D^{2k}\setminus \mathrm{int}(D(r'))\bigr)\cong \bigl(D(r'),\,\partial D(r')\bigr),
\]
we obtain an isomorphism
\[
K^{0}\bigl(D^{2k},\,D^{2k}\setminus \mathrm{int}(D(r'))\bigr)\cong K^{0}\bigl(D(r'),\partial D(r')\bigr).
\]
Thus it suffices to determine the class $[\hat A|_{D(r')}]$ in $K^{0}(D(r'),\partial D(r'))$.

\smallskip
\noindent\textbf{Step 2: The case $Jf(0)=E_{2k}$.}
Assume first $Jf(0)=E_{2k}$. Then Lemma~\ref{lem:4.9}(1) implies that, after shrinking the radius, there exist positive real numbers $r,r'>0$ with $r'\ge r$ such that
\[
f|_{D(r)}:D(r)\to D(r')
\]
is homotopic, as a map of pairs, to the inclusion
\[
i:D(r)\hookrightarrow D(r').
\]
By homotopy invariance of $K^{0}$, the class $[\hat A|_{D(r)}]$ is equal to the class of the
standard local model
\[
\Bigl[\hat\mu_{\hat S_{+}}+\hat\gamma^{\perp}\Bigr]\in K^{0}(D(r),\partial D(r)).
\]
Under the fixed normalization in \S\ref{subsec:Description of the generator}
(Convention~\ref{conv:std-cliff}), the class
\[
\bigl[\hat\mu_{\hat S_{+}}+\hat\gamma^{\perp}\bigr]
\]
corresponds to $1\in\mathbb Z$.
Therefore $[\hat A|_{D(r)}]$ also corresponds to $1\in\mathbb Z$.

\smallskip
\noindent\textbf{Step 3: The general case $\det Jf(0)>0$.}
For general $f$ with $\det Jf(0)>0$, let
\[
T:=Jf(0)\in GL^{+}(2k;\mathbb R),
\qquad
f_{1}:=f\circ T^{-1}.
\]
Then
\[
f=f_{1}\circ T,
\]
and Lemma~\ref{lem:decomposition-of-f} gives
\[
Jf_{1}(0)=E_{2k}.
\]
By Lemma~\ref{lem:4.9}(1) applied to $f_{1}$ and Lemma~\ref{lem:4.9}(2) applied to $T$,
after shrinking radii, both $f_{1}$ and $T$ are homotopic to the corresponding inclusions of balls (as maps of pairs).
Hence $f=f_{1}\circ T$ is also homotopic to the inclusion, and by Step~2 the class
$[\hat A|_{D(r')}]$ corresponds to $1\in\mathbb Z$.

Consequently $[\hat A]\in K^{0}(D^{2k},\partial D^{2k})$ corresponds to $1\in\mathbb Z$.
\end{proof}

Let $\hat{S}_{+}=\hat{S}^{(2k+2)}$ be the standard $\mathbb{Z}_2$-graded irreducible $Cl_{2k+2}$ representation fixed in Convention~\ref{conv:std-cliff}, and let $\hat{S}_{-}$ be the inequivalent graded irreducible $Cl_{2k+2}$ representation.
(If necessary, the $\mathbb{Z}_2$-graded irreducible representations of $Cl_{2k+1}$ obtained by ignoring the action of $\gamma_{2k+2}\in Cl_{2k+2}$ are assumed to be identical.)
Given a $\mathbb{Z}_2$-graded representation $\hat{S}$ of $Cl_{2k+2}$, let\,$\hat{\mu}_{\hat{S}}\colon D^{2k+1}\to \mathrm{End}(\hat{S})$ be defined as 
 \begin{align*} 
\hat{\mu}_{\hat{S}}(x_1,...,x_{2k+1})=\sum_{i=1}^{2k+1}x_{i}\hat{\gamma}_{i}.
\end{align*} 
The orthogonal complement of the embedding $\hat{S} \subset \mathcal{\hat{H}}_{2k+2}$ as a graded representation is denoted as $\hat{S}^{\perp}$, and the action of $Cl_{2k+3}$, which is an extension of the action of $Cl_{2k+2}$, on $\hat{S}^{\perp} \cong \mathcal{\hat{H}}_{2k+2}$ is denoted as $\hat{\gamma}^{\perp}$.

We also state the odd-dimensional analogue (the proof is parallel to the even case), which will be used
to identify the local generator in $K^{-1}(D^{2k+1},\partial D^{2k+1})\cong\mathbb Z$ consistently with
Convention~\ref{conv:odd-generator}.
                                                   
\begin{thm}\label{thm:4.11}
Suppose that a continuous mapping 
\begin{align*}
\hat{A}\colon(D^{2k+1},\partial D^{2k+1})\to (\mathcal{F}_{1}(\mathcal{\hat{H}}),\mathcal{F}^{*}_{1}(\mathcal{\hat{H}}))
 \end{align*}
 satisfies the following assumptions:
\begin{enumerate}
\item The Hilbert space \( \hat{\mathcal{H}} \) admits an orthogonal decomposition 
\( \hat{\mathcal{H}} = \hat{S}_{+} \oplus \hat{S}_{+}^{\perp} \), and accordingly, 
the operator \( \hat{A}(x) \) can be decomposed as a sum 
\[
\hat{A}(x) = \hat{a}(x) + \hat{a}^{\perp}(x),
\]
where \( \hat{a}(x) \colon \hat{S}_{+} \to \hat{S}_{+} \subset \hat{\mathcal{H}} \) 
and \( \hat{a}^{\perp}(x) \colon \hat{S}_{+}^{\perp} \to \hat{\mathcal{H}} \).\\
\item For any $x \in D^{2k+1}$, we have $\hat{a}^{\perp}(x)\in \mathcal{F}^{*}_{1}(\mathcal{\mathcal{H}})$.\\
\item There exists a smooth mapping $f=(f_1,...,f_{2k+1})\colon D^{2k+1}\to \mathbb{R}^{2k+1}$ such that:
\begin{itemize}
    \item  For any $x \in D^{2k+1}$, the map $\hat{a}$ admits the expression $\hat{a}(x)=\sum_{i=1}^{2k+1}f_i(x)\hat{\gamma}_i$.
    \item $f^{-1}(\{0\})=\{0\}$.
    \item $\det Jf(0)>0$.
\end{itemize}
\end{enumerate}
In this case, under an isomorphism $K^{-1}(D^{2k+1},\partial D^{2k+1})\cong \mathbb{Z}$, the map $\hat{A}$ represents the generator $1 \in \mathbb{Z}$. Here the isomorphism $K(D^{m},\partial D^{m})\cong \mathbb{Z}$ is normalized by the generator described in Section~\ref{subsec:Description of the generator}, together with Convention~\ref{conv:std-cliff}; namely, the standard local model built from the standard Clifford representation represents $1\in\mathbb{Z}$.
\end{thm}

\begin{proof}
The proof proceeds similarly in the even-dimensional case.
\end{proof}                                                 

\begin{rem}\label{rem:3}
Consider the $\mathbb{Z}_2$-graded representation $\hat{S}$ of the Clifford algebra $Cl_{2k+1}$, defined as $\hat{S} = S^0 \oplus S^1$, where $S^0 = S^1$. Under this representation, we express the following matrices:
\begin{align*}
\hat{\epsilon}&=\begin{pmatrix}
1_S & 0 \\
0 & -1_S
\end{pmatrix},&
\hat{\gamma}_i&=\begin{pmatrix}
0 & \gamma_i \\
\gamma_i & 0
\end{pmatrix},&
\hat{\gamma}_{2k+2}&=\begin{pmatrix}
0 & -\sqrt{-1}\cdot1_S \\
\sqrt{-1}\cdot1_S & 0
\end{pmatrix}.
\end{align*}
where $i$ ranges from 1 to $2k+1$. Here, $\gamma_1, \dots, \gamma_{2k+1}$ are self-adjoint homomorphisms on $S$, which collectively provide an ungraded representation of $Cl_{2k+1}$.

In this framework, we define the operators $\hat{a}(x)$ and $\hat{\mu}(x)$ as follows:
\begin{align*}
\hat{a}(x)&=\begin{pmatrix}
0 & \sum_{i=1}^{2k+1}f_i(x)\gamma_i \\
\sum_{i=1}^{2k+1}f_i(x)\gamma_i & 0
\end{pmatrix},&
\hat{\mu}(x)&=\begin{pmatrix}
0 & \sum_{i=1}^{2k+1}x_i\gamma_i \\
\sum_{i=1}^{2k+1}x_i\gamma_i & 0
\end{pmatrix}.
\end{align*}
From these definitions, it follows that a statement concerning $\hat{a}(x)$ can be equivalently reformulated as a statement about $a(x)$.
\end{rem}

We define $Fred_{sa}^{1,*}(\mathcal{H}) \subset Fred_{sa}^{1}(\mathcal{H})$ as the subspace of invertible operators.

Let $S_{+}$ be the standard ungraded irreducible representation of $Cl_{2k+1}$ fixed in Convention~\ref{conv:std-cliff}, and let $S_{-}$ be the inequivalent one. Furthermore, consider a separable infinite-dimensional Hilbert space $\mathcal{H}$, and take the direct sum $S_{+}\oplus \mathcal{H}$ with this fixed standard choice of $S_{+}$. In this case, the following corollary is obtained.

\begin{cor}
Consider a continuous map
\begin{align*}
A \colon (D^{2k+1}, \partial D^{2k+1}) \to (Fred_{sa}^{1}(S_{+} \oplus \mathcal{H}), Fred_{sa}^{1,*}(S_{+} \oplus \mathcal{H})).
\end{align*}
We denote the \( S_{+} \)-component and the \( \mathcal{H} \)-component of \( A \) 
by \( a \) and \( a^{\perp} \), respectively. That is, we write
\[
A(x) = a(x) + a^{\perp}(x),
\]
where \( a(x) \colon S_{+} \to S_{+} \subset \mathcal{H} \) and 
\( a^{\perp}(x) \colon S_{+}^{\perp} \to \mathcal{H} \). We assume that $A$ satisfies the following conditions.
\begin{enumerate}
\item For any $x \in D^{2k+1}$, we have $a^{\perp}(x)\in Fred_{sa}^{1,*}(S_{+}\oplus \mathcal{H})$.\\
\item There exists a smooth mapping $f=(f_1,...,f_{2k+1})\colon D^{2k+1}\to \mathbb{R}^{2k+1}$ such that:
\begin{itemize}
    \item  For any $x \in D^{2k+1}$, the map a admits the expression $a(x)=\sum_{i=1}^{2k+1}f_i(x)\gamma_i$.
    \item $f^{-1}(\{0\})=\{0\}$.
    \item $\det Jf(0)>0$.
\end{itemize}
\end{enumerate}
In this case, under an isomorphism $K^{-1}(D^{2k+1},\partial D^{2k+1})\cong \mathbb{Z}$, the map $A$ represents the generator $1 \in \mathbb{Z}$. Here the isomorphism $K(D^{m},\partial D^{m})\cong \mathbb{Z}$ is normalized by the generator described in Section~\ref{subsec:Description of the generator}, together with Convention~\ref{conv:std-cliff}; namely, the standard local model built from the standard Clifford representation represents $1\in\mathbb{Z}$.
\end{cor}

\begin{proof}
The claim follows from the correspondence between $Fred_{sa}^{1}$ and $\mathcal{F}_1$, along with Theorem \ref{thm:4.11} and Remark \ref{rem:3}.
\end{proof}


\subsection{Proof of the main theorem}
\label{sub:Proof of the main theorem}

Let $X$ be a compact, oriented, differentiable manifold of dimension $2k+1$. Let \,$A\colon X\to Fred_{sa}^1(\mathcal{H})$ be a continuous map and $F$ the Fermi set of $A$. Let $(V, b = (b_1, \dots, b_{2k+1}))$ be an oriented coordinate of $X$, where $V$ is a coordinate neighborhood of $x \in F$ and $b: V \to \mathbb{R}^{2k+1}$ satisfies $b(x) = (0, \dots, 0)$. If necessary, by suitably reselecting $V$, we can assume that the image $b(V)$ under $b$ coincides with the disk $D^{2k+1}$.

\begin{convention}[Odd generator normalization]\label{conv:odd-generator}
Throughout the odd-dimensional case, we fix the isomorphism
\[
K^{-1}(D^{2k+1},\partial D^{2k+1})\cong \mathbb{Z}
\]
by declaring that the class
\[
\Big[\Big(-\sum_{i=1}^{2k+1} x_i \gamma_i\Big)+\gamma^{\perp}\Big]
\]
corresponds to \(1\in\mathbb{Z}\), where \((x_1,\dots,x_{2k+1})\) is the standard oriented coordinate on \(D^{2k+1}\).
This is a convention on the generator and does not modify the finite-dimensional approximation \(A_{<\mu}\).
\end{convention}

Through this $b$, we identify $b^*\colon K^{-1}(D^{2k+1},\partial D^{2k+1}) \cong K^{-1}(V, \partial V).$ Under this identification, the generator $1$ corresponds to 
\begin{align*}
b^*\Big[\Big(-\sum_{i=1}^{2k+1}x_i\gamma_i\Big) + \gamma^{\perp}\Big]
=
\Big[\Big(-\sum_{i=1}^{2k+1}b_i\gamma_i\Big) + \gamma^{\perp}\Big].
\end{align*}

\begin{lem}\label{lem:4.12}
Let $X$ be a compact, oriented, differentiable manifold of dimension $2k+1$. Let \,$A\colon X\to Fred_{sa}^1(\mathcal{H})$ be a continuous map and $F$ the Fermi set of $A$. Let $(V,a=(a_{1},...,a_{2k+1}))$ be a sign coordinate of $A$ at $x\in F$.
In this case, when $(V,b=(b_{1},...,b_{2k+1}))$ is an oriented coordinate,
\begin{align*}
\Big[\Big(-\sum_{i=1}^{2k+1}b_i\gamma_i\Big) + \gamma^{\perp}\Big]
=
\operatorname{sign}(x)\Big[\Big(-\sum_{i=1}^{2k+1}a_i\gamma_i\Big) + \gamma^{\perp}\Big].
\end{align*}
holds.
\end{lem}

\begin{proof}
Let $f_i$ denote the projection of $a \circ b^{-1}$ onto the $i$-th component, for $i = 1, \dots, 2k+1$. By using $f_i$, we can express $a \circ b^{-1}$ as $(f_1, \dots, f_{2k+1})$. Under this definition, the $K$-group classes satisfy the relation:
\begin{align*}
[(-\sum_{i=1}^{2k+1}x_i\gamma_i) + \gamma^{\perp}]=\operatorname{sign}(x)[(-\sum_{i=1}^{2k+1}f_i\gamma_i) + \gamma^{\perp}].
\end{align*}
\noindent
No ad-hoc coordinate sign change is needed here, since the generator is fixed by Convention~\ref{conv:odd-generator}.
Consequently, we can derive the following sequence of equalities by applying the pullback $b^{*}$: 
\begin{align*}
[(-\sum_{i=1}^{2k+1}b_i\gamma_i) + \gamma^{\perp}]=b^{*}([(-\sum_{i=1}^{2k+1}x_i\gamma_i) + \gamma^{\perp}])&=\operatorname{sign}(x)b^{*}([(-\sum_{i=1}^{2k+1}f_i\gamma_i) + \gamma^{\perp}])\\
&=\operatorname{sign}(x)[(-\sum_{i=1}^{2k+1}a_i\gamma_i) + \gamma^{\perp}].
\end{align*} 
\end{proof}

\begin{lem}\label{lem:4.13}
Let $X$ be a $(2k+1)$-dimensional compact, oriented differentiable manifold. Let \,$A\colon X\to Fred_{sa}^1(\mathcal{H})$ be a continuous map and $F$ the Fermi set of $A$. We assume that the Fermi set $F = \{x\}$ is a singleton set. Then, the following holds.
\begin{enumerate}
\item  The condition that a sign coordinate can be taken at $x$ is equivalent to the following: there exists a coordinate $(V;a=(a_1,...,a_{2k+1}))$ around $x$ such that $f(V)=D^{2k+1}$, and the pullback $(a^{-1})^*([A|_V])\in K^1(D^{2k+1},\partial D^{2k+1})\cong \mathbb{Z}$ is a generator.\\
\item Suppose that a sign coordinate $(V;a=(a_1,...,a_{2k+1}))$ can be taken around a point $x\in F$ on the manifold $X$.Under this condition, the following equality holds in the $K$-group:
\begin{align*}
i_{*}^{K}(1)=\operatorname{sign}(x)[A] \in K^{1}(X)
\end{align*}
Here, $i:F\to X$ is an inclusion map from the Fermi set into $X$, and the element $1$ is identified with the generator of $K^0(F)$.
\end{enumerate}
\end{lem}

\begin{proof}
\begin{enumerate}
\item Suppose there exists a coordinate $(V;a=(a_1,...,a_{2k+1}))$ $(D^{2k+1}=f(V))$ around a point $x\in X$, such that $f(V)=D^{2k+1},$ and the $K$-group class
$[A|_{D^{2k+1}}] \in K^{-1}( D^{2k+1},\partial D^{2k+1})\cong \mathbb{Z}$ is a generator. This assumption implies that $V$ is mapped onto the disk $D^{2k+1}$ via the coordinate function $f$, and $A$ restricted to this disk generates the $K$-group.

\noindent
No ad-hoc coordinate sign change is needed here, since the generator is fixed by Convention~\ref{conv:odd-generator}.
\begin{align*}
(a^{-1})^*([A|_V])=
\Big[\Big(-\sum_{i=1}^{2k+1}x_i\gamma_i\Big)+\gamma^\perp\Big].
\end{align*}

Consequently, it follows that:
\begin{align*}
[A|_V]=
\Big[\Big(-\sum_{i=1}^{2k+1}a_i\gamma_i\Big)+\gamma^\perp\Big].
\end{align*}
Here, the equality indicates that $A_{<\mu}|_{V}\simeq \sum_{i=1}^{2k+1}a_i\gamma_i$.
The converse—that every sign coordinate satisfies these conditions—follows from Lemma \ref{lem:4.12}.

\item Take an oriented coordinate $(V,b=(b_{1},...,b_{2k+1}))$ $(f(V)=D^{2k+1})$ around $x$, and identify $V$ with $D^{2k+1}$. Then,
\begin{align*}
\mathrm{AS}^{2k+1}(1)=\Big[\Big(-\sum_{i=1}^{2k+1}b_i\gamma_i\Big) + \gamma^{\perp}\Big]
=\operatorname{sign}(x)[A|_{D^{2k+1}}].
\end{align*}
Here, $\mathrm{AS}^{2k+1}$ is the $(2k+1)$-fold Atiyah--Singer suspension isomorphism
\[
\mathrm{AS}^{2k+1}\colon K^{0}(F)\xrightarrow{\cong}K^{1}(D^{2k+1},\partial D^{2k+1}).
\]
Recall that the Fermi set $F$ is defined as the set of points $x\in X$ where the operator $A_x$ is singular, i.e., $F=\{x\in X|0\in Spec(A_x)\}.$ By this definition, the $K$-group class $[A]$ can be regarded as an element of $K^{1}(X, X\setminus F)$, since $A$ is invertible on $X\setminus F$. With this understanding, the statement follows from the commutativity of the diagram below.

\begin{tikzpicture}[thick]
\node (a) at (0, 0) {$K^0(F)$}; 
\node (b) at (3.2, 0) {$K^1(D^{n}, \partial D^{n})$};
\node (c) at (6.8, 0) {$K^1(X, X \setminus \text{int}(D^n))$};   
\node (d) at (10.2, 0) {$K^1(X, X \setminus F)$};
\node (e) at (10.2, -3) {$K^1(X)$}; 

\draw[->] (a) -- 
  node[midway, above] {\(\scriptstyle \mathrm{AS}^{2k+1}\)}
  node[midway, below] {\(\scriptstyle \cong\)} 
(b);

\draw[->] (b) -- 
  node[midway, above] {\(\scriptstyle j^*\)}
  node[midway, below] {\(\scriptstyle \cong\)}
(c);

\draw[->] (c) -- 
  node[midway, above] {\(\scriptstyle k^*\)}
  node[midway, below] {\(\scriptstyle \cong\)}
(d);

\draw[->] (d) -- node[midway, right] {\(\scriptstyle \ell^*\)} (e);

\draw[->] (a) -- node[pos=0.55, above left] {\(\scriptstyle i_*^K\)} (e);
\end{tikzpicture}\\
where
\[
j\colon (D^n,\partial D^{n})\to (X, X \setminus \mathrm{int}(D^n)),
\qquad
k\colon (X, X \setminus \mathrm{int}(D^n))\to (X, X\setminus \{x\}),
\]
and
\[
\ell\colon (X,\emptyset)\to (X, X\setminus \{x\})
\]
are the canonical inclusion maps.
\end{enumerate}
\end{proof}

\begin{lem}\label{lem:4.14}
Let $X$ be a compact, oriented, differentiable manifold, and $\mathcal{H}$ a separable infinite-dimensional Hilbert space. Let $A \colon X \to \mathrm{Fred}_{sa}(\mathcal{H})$ be a continuous mapping. In this case, the points in $F$ that are Fermi points are isolated points.
\end{lem}

\begin{proof}
Let $x\in F$ be a Fermi point. By definition, there exists a sign coordinate $(V; a_1,\dots,a_n)$ around $x$
($n=\dim X$) such that
\[
A_{<\mu}\big|_{V} \simeq \sum_{i=1}^{n} a_i \gamma_i
\]
in the sense of Remark~\ref{rem:fiberwise-homotopy}. Hence there exists a homotopy
$H:V\times[0,1]\to \mathrm{End}(\C^{m})$ from $A_{<\mu}\big|_V$ to $\sum a_i\gamma_i$
such that $H(v,t)$ is invertible for all $v\in V\setminus\{x\}$ and all $t\in[0,1]$.
In particular, $A_{<\mu}(v)=H(v,0)$ is invertible for every $v\in V\setminus\{x\}$.

Since $0\in\mathrm{Spec}(A_v)$ holds if and only if $A_{<\mu}(v)$ is not invertible,
we conclude that $F\cap V=\{x\}$. Therefore $x$ is isolated in $F$.
\end{proof}

\begin{lem}\label{lem:4.15}
Let $X$ be a compact, oriented, differentiable manifold, and $\mathcal{H}$ a separable infinite-dimensional Hilbert space. Let $A \colon X\to Fred_{sa}(\mathcal{H}) $ be a continuous mapping, and let us assume that each point of the Fermi set \(F\) is a Fermi point. Then, \(F\) is a finite set.
\end{lem}

\begin{proof}
Since $A$ is continuous, the set
\[
F=\{x\in X\mid 0\in \mathrm{Spec}(A_x)\}
\]
is closed in $X$. As $X$ is compact, $F$ is compact.

By Lemma~\ref{lem:4.14}, for each $x\in F$ there exists an open neighborhood $U_x$ such that
$F\cap U_x=\{x\}$. Then $\{U_x\}_{x\in F}$ is an open cover of the compact set $F$,
so it admits a finite subcover $U_{x_1},\dots,U_{x_N}$. Hence
\[
F \subset \bigcup_{j=1}^{N} U_{x_j}
\quad\text{and}\quad
F\cap U_{x_j}=\{x_j\},
\]
which implies $F=\{x_1,\dots,x_N\}$ is finite.
\end{proof}

\begin{thm}[Theorem \ref{thm:odd main theorem}]\label{thm:odd main thm}
Let $k$ be a non-negative integer. Let $X$ be a compact, oriented, $(2k+1)$-dimensional differentiable manifold, and \,$A\colon X\to Fred_{sa}^1$ a continuous map. Then, the following holds.
\begin{enumerate}
\item When each point $x \in F$ is a Fermi point, the sum of signs $\sum_{x \in F}\operatorname{sign}(x)$ depends only on $[A] \in K^1(X)$.
\item Under the assumption of $(1)$, the following equality holds.
\begin{align*}
\int_X \operatorname{ch}_{k+\frac{1}{2}}([A]) = \sum_{x \in F}\operatorname{sign}(x) \in \mathbb{Z}.
\end{align*} 
\end{enumerate} 
\end{thm}

\begin{proof}
The sum $\sum_{x \in F}\operatorname{sign}(x)$ is finite by Lemma \ref{lem:4.15}.
Since the integral of the Chern character depends only on the homotopy class of $A$,
it is sufficient to prove $(2)$.

First assume that $F=\{x\}$.
By Lemma \ref{lem:4.13}, we have
\[
i_{*}^{K}(1)=\operatorname{sign}(x)[A] \in K^{1}(X).
\]
From the odd-dimensional case of Proposition~\ref{prop:chern character and push forward},
the following diagram commutes:
\begin{center}
\begin{tikzpicture}[thick]
\node (a) at (0, 3) {$K^{0}(F)$};
\node (x) at (3.5, 3) {$H^{\mathrm{even}}(F;\mathbb{Q})$};
\node (b) at (0, 0) {$K^{1}(X)$};
\node (y) at (3.5, 0) {$H^{\mathrm{odd}}(X;\mathbb{Q})$};

\draw[->] (a) -- node[midway, above] {\(\scriptstyle \operatorname{Ch}_{\mathrm{even}}\)} (x);
\draw[->] (x) -- node[midway, right] {\(\scriptstyle i_*^H\)} (y);
\draw[->] (a) -- node[midway, left] {\(\scriptstyle i_*^K\)} (b);
\draw[->] (b) -- node[midway, below] {\(\scriptstyle \operatorname{Ch}_{\mathrm{odd}}\)} (y);
\end{tikzpicture}
\end{center}
Therefore,
\[
\operatorname{ch}_{k+\frac{1}{2}}(i_*^K(1))=i_*^H(1),
\]
and $i_*^H(1)$ is the top-degree cohomology class dual to the point $x$, so that
\[
\int_X i_*^H(1)=1.
\]
Hence
\[
\int_X \operatorname{ch}_{k+\frac{1}{2}}(i_*^K(1))=1.
\]
Combining this with
\[
i_*^K(1)=\operatorname{sign}(x)[A],
\]
we obtain
\[
\int_X \operatorname{ch}_{k+\frac{1}{2}}([A])=\operatorname{sign}(x).
\]

Now let
\[
F=\{x_1,\dots,x_N\}.
\]
Choose pairwise disjoint closed coordinate disks
\[
D_1,\dots,D_N
\]
around $x_1,\dots,x_N$, respectively.
By excision and the additivity of the push-forward in $K$-theory and in singular cohomology
over the disjoint union $\bigsqcup_{j=1}^{N}D_j$, the above singleton argument applied to
each $x_j$ yields
\[
\int_X \operatorname{ch}_{k+\frac{1}{2}}([A])
=
\sum_{j=1}^{N}\operatorname{sign}(x_j).
\]
This proves $(2)$, and $(1)$ follows because the left-hand side depends only on
$[A]\in K^1(X)$.
\end{proof}

\begin{thm}[Theorem \ref{thm:even main theorem}]\label{thm:even main thm}
Let $k$ be a non-negative integer. Let $X$ be a compact, oriented, $2k$-dimensional differentiable manifold, and \,$\hat{A}\colon X\to \mathcal{F}_{0}(\mathcal{\hat{H}})$ a continuous map. Then, the following holds.
\begin{enumerate}
\item When a sign coordinate is taken at each point $x \in F,$ the sum of signs $\sum_{x \in F}\operatorname{sign}(x)$ depends only on $[\hat{A}] \in K^0(X)$.
\item Under the assumption of $(1)$, the following equality holds.
\begin{align*}
\int_X \operatorname{ch}_{k}([\hat{A}]) = \sum_{x \in F}\operatorname{sign}(x) \in \mathbb{Z}.
\end{align*}
\end{enumerate}
\end{thm}

\begin{proof}
The proof is completely parallel to that of Theorem~\ref{thm:odd main thm}.
One replaces odd sign coordinates by even sign coordinates, the odd local generator by
the even local generator in Theorem~\ref{thm:4.10}, and the odd-dimensional case of
Proposition~\ref{prop:chern character and push forward} by its even-dimensional case.
With these replacements, the same argument as in Theorem~\ref{thm:odd main thm} yields
\[
\int_X \operatorname{ch}_{k}([\hat{A}])
=
\sum_{x\in F}\operatorname{sign}(x).
\]
Hence $(1)$ follows from $(2)$ by homotopy invariance of the left-hand side.
\end{proof}


\subsection{Specific calculation example}
\label{subsec:Specific calculation example}


Below, we present specific computational examples based on Theorem \ref{thm:even main thm} and Theorem \ref{thm:odd main thm}. In \S\S\S \ref{subsubsec:example1}, we demonstrate that Theorem \ref{thm:odd main thm} corresponds to a generalization of the spectral flow. In this section, we denote $S^1 = \mathbb{R}/2\pi \mathbb{Z}$.
In Examples 2--4, the Fredholm property of the families $\hat H^\#$ is verified by using
the Fredholm criterion for the local model stated later in Section~\ref{sec:local model}.
More precisely, after defining $\hat H^\#$ in each example, we briefly check that the image
of the corresponding parameter map avoids the bad locus $\Sigma$ introduced there.

\subsubsection{Example $1$}
\label{subsubsec:example1}

We define $(a, b) \colon S^1 \to \mathbb{R}^2$ as follows:
\begin{align*}
    a(k) &= \frac{3}{2} + \cos k, \\
    b(k) &= \sin k.
\end{align*}
We utilize the results of local models established in \cite{G1}.

We define the following mappings by using the local model:
\begin{align*}
    \hat{H} &= \hat{H}_{\text{loc}} \circ (a, b) \times 1_{S^1} \colon T^2 = S^1 \times S^1 \to \text{Herm}(\mathbb{C}^2), \\
    \hat{H}^{\#} &= \hat{H}^{\#}_{\text{loc}} \circ (a, b) \colon S^1 \to Fred_{\text{sa}}(l^2(\mathbb{N}, \mathbb{C}^2)).
\end{align*}
In this example, the corresponding symbol is pointwise invertible on $S^1$, so the family $\hat H^\#$ indeed takes values in $Fred_{sa}(l^2(\mathbb N,\mathbb C^2))$.

The Fermi set is $F = \{\pi\}$. 
We can approximate $\hat{H}^{\#}$ on $(-1, 1) \times \{0\} \subset \mathbb{R}^2$ as follows:
\begin{align*}
    (\hat{H}^{\#}_{\text{loc}})_{<\sqrt{2}} = b.
\end{align*}
We have $\det Jb(\pi) = \cos \pi = -1 < 0$, where we define $J(b(k)) = \frac{\partial b}{\partial k}(k) \in \mathbb{R}$. We conclude by the inverse function theorem that $b(k)$ serves as a sign coordinate of $S^1$ in a neighborhood of $\pi$.  
By Definition~\ref{dfn:odd sign coordinates}(2), the odd sign is $\operatorname{sign}(\pi)=-\det Jb(\pi)/|\det Jb(\pi)|=1$.
Thus, we obtain the following result by applying Theorem \ref{thm:odd main thm} with $\mathcal{H} = l^2(\mathbb{N}, \mathbb{C}^2)$:
\begin{align*}
    \int_{S^1} \operatorname{ch}_{\frac{1}{2}}([\hat{H}^{\#}]) = 1.
\end{align*}
We confirm that this result is consistent with $sf([\hat{H}^{\#}]) = 1$. 

\subsubsection{Example $2$}
\label{subsubsec:example2}

In \S\S\S \ref{subsubsec:example2} through \S\S\S \ref{subsubsec:example4}, we use the local model $\hat{H}_{loc} \colon \mathbb{R} \times \mathbb{C}^2 \times S^1 \to \text{Herm}(\mathbb{C}^4)$ given by the following equation:
\begin{align*}
    \hat{H}_{loc}(a, b, c, k) = 
    \begin{pmatrix}
        a & \bar{c} - e^{ik} & 0 & \bar{b} \\
        c - e^{-ik} & -a & \bar{b} & 0 \\
        0 & b & a & -c + e^{-ik} \\
        b & 0 & -\bar{c} + e^{ik} & -a
    \end{pmatrix}.
\end{align*}
We define $\hat{H}^{\#}_{loc}$ by
\begin{align*}
    \hat{H}^{\#}_{loc}(a, b, c) = H^{\#}_{(a, b, c)},
\end{align*}
which is a bounded self-adjoint operator on $l^2(\mathbb{N}, \mathbb{C}^4)$.
We refer to \S \ref{sec:local model} for the definition of $H^{\#}_{(a, b, c)}$. Below, we utilize the result of Proposition \ref{lem:5.8}.

We note that, in the neighborhood of each point $\{0\} \times \{0\} \times D^2 \subset \mathbb{R} \times \mathbb{C}^2$, the following holds:
\begin{align*}
    (\hat{H}^{\#}_{loc})_{<\sqrt{3}} = \text{Re}(b) \sigma_1 + \text{Im}(b) \sigma_2 + a \sigma_3.
\end{align*}

We let $k = (x, y + \sqrt{-1} z, w) \in \mathbb{R} \times \mathbb{C} \times \mathbb{R}$.
We define $(a, b, c) \colon S^3 \subset \mathbb{R} \times \mathbb{C} \times \mathbb{R} \to \mathbb{R} \times \mathbb{C}^2$ as follows:
\begin{align*}
    &a(k) = x, &b(k) = y + \sqrt{-1} z&, &c(k) = w - 1.
\end{align*}
In this case, we use the local model to define $\hat{H}^{\#} = \hat{H}^{\#}_{loc} \circ (a, b, c) \colon S^3 \to Fred_{\text{sa}}(l^2(\mathbb{N}, \mathbb{C}^4))$.
We briefly verify the Fredholm condition.
By Remark~\ref{rem:fredholm-criterion-local-model}, it suffices to check that the image of
the parameter map $(a,b,c)\colon S^3\to \R\times \C\times \C$ avoids the bad locus
$\Sigma$ introduced in Section~\ref{sec:local model}.
Indeed, if $a(k)=0$ and $b(k)=0$, then $x=0$ and $y=z=0$, hence $w=\pm 1$ on $S^3$.
Therefore
\[
c(k)=w-1\in\{0,-2\},
\]
so in particular $|c(k)|\neq 1$.
Thus $(a,b,c)(S^3)\cap \Sigma=\emptyset$, and hence $\hat H^\#$ is a continuous family in
$Fred_{sa}(l^2(\mathbb N,\mathbb C^4))$.

By the spectral analysis of the local model, the Fermi set $F$ of $\hat{H}^{\#}$ is given by:
\begin{align*}
    a(k) &= 0, &b(k) = 0&, &\lvert c(k) \rvert &\leq 1.
\end{align*}
Therefore, we obtain $F = \{(0, 0, 1)\}$.

We define $J(\text{Re}(b(k)), \text{Im}(b(k)), a(k))$ as the Jacobian matrix.
The determinant is as follows:
\begin{align*}
    \det J(\text{Re}(b(k)), \text{Im}(b(k)), a(k)) = 1.
\end{align*}
By applying the inverse function theorem, we establish that $(\text{Re}(b), \text{Im}(b), a)$ serves as a sign coordinate of $S^3$ in the neighborhood of $(0, 0, 1)$.
Therefore, we apply the Theorem \ref{thm:odd main thm} with $\mathcal{H} = l^2(\mathbb{N}, \mathbb{C}^4)$ to obtain:
\begin{align*}
    \int_{S^3} \operatorname{ch}_{\frac{3}{2}}([H^{\#}]) = -1.
\end{align*}

\subsubsection{Example $3$}
\label{subsubsec:example3}

We define $(a, b, c) \colon T^3 = S^1 \times S^1 \times S^1 \to \mathbb{R} \times \mathbb{C}^2$ as follows:
\begin{align*}
    a(k) &= -\cos(k_2 + k_3), \\
    b(k) &= (-\cos(k_1) - \cos(k_2)) + \sqrt{-1} (\sin(k_1) + \sin(k_2)), \\
    c(k) &= -1 - e^{-\sqrt{-1} k_3},
\end{align*}
where we denote $k = (k_1, k_2, k_3) \in T^3$.
In this case, we use the local model to define the following:
\begin{align*}
    \hat{H} &= \hat{H}_{loc} \circ (a, b, c) \times 1_{S^1} \colon T^4 = T^3 \times S^1 \to \text{Herm}(\mathbb{C}^4), \\
    \hat{H}^{\#} &= \hat{H}^{\#}_{loc} \circ (a, b, c) \colon T^3 \to Fred_{\text{sa}}(l^2(\mathbb{N}, \mathbb{C}^4)).
\end{align*}
We also verify the Fredholm condition for this family.
By Remark~\ref{rem:fredholm-criterion-local-model}, it is enough to show that the image of
the parameter map avoids the bad locus $\Sigma$ introduced in Section~\ref{sec:local model}.
If $a(k)=0$ and $b(k)=0$, then the possible points are exactly the points that will be identified
below as the Fermi set.
At those points one has $|c(k)|<1$, hence in particular $|c(k)|\neq 1$.
Therefore $(a,b,c)(T^3)\cap \Sigma=\emptyset$, and so $\hat H^\#$ indeed defines a continuous
family in $Fred_{sa}(l^2(\mathbb N,\mathbb C^4))$.
($\hat{H}$ is a Hamiltonian corresponding to a four-dimensional topological insulator in class AI \cite{Qi},\cite{R}.)

We identify the Fermi set $F$ of $\hat{H}^{\#}$ by the conditions:
\begin{align*}
    &a(k) = 0,&b(k)= 0&, &\lvert c(k) \rvert &\leq 1.
\end{align*}
By the spectral analysis of the local model, the Fermi set $F$ of $\hat{H}^{\#}$ is given by:
\begin{align*}
    F = \left\{ \left( \frac{2\pi}{3}, \frac{4\pi}{3}, \frac{7\pi}{6} \right), \left( \frac{4\pi}{3}, \frac{2\pi}{3}, \frac{5\pi}{6} \right) \right\}.
\end{align*}
We let $J(\text{Re}(b(k)), \text{Im}(b(k)), a(k))$ denote the Jacobian matrix. Its determinant is
\begin{align*}
    \det J(\text{Re}(b(k)), \text{Im}(b(k)), a(k)) = \sin(k_1 - k_2) \sin(k_2 + k_3).
\end{align*}
For any $x \in F$, the determinant satisfies $\det J(x) = -\frac{\sqrt{3}}{2} < 0$.
Therefore, we apply the inverse function theorem and conclude that, in the neighborhood of each point in $F$, $(\text{Re}(b), \text{Im}(b), a)$ serves as a sign coordinate on $T^3$.
Consequently, we apply Theorem \ref{thm:odd main thm} with $\mathcal{H} = l^2(\mathbb{N}, \mathbb{C}^4)$ and obtain:
\begin{align*}
    \int_{T^3} \operatorname{ch}_{\frac{3}{2}}([\hat{H}^{\#}]) = 2.
\end{align*}

\subsubsection{Example $4$}
\label{subsubsec:example4}

We define $\mathcal{H} = l^2(\mathbb{N}, \mathbb{C}^2)$ and 
$\hat{\mathcal{H}} = l^2(\mathbb{N}, \mathbb{C}^4) = l^2(\mathbb{N}, \mathbb{C}^2) \oplus l^2(\mathbb{N}, \mathbb{C}^2) = \mathcal{H} \oplus \mathcal{H}$, 
and we specify the involution as follows:
\begin{align*}
    \hat{\epsilon} = 
    \begin{pmatrix}
        1_{\mathcal{H}} & 0 \\
        0 & -1_{\mathcal{H}}
    \end{pmatrix}.
\end{align*}
We adopt the local model from the previous example with $a = 0$. 
Specifically, we define $\hat{H}^{\#}_{loc} \colon \mathbb{C}^2 \to B(\hat{\mathcal{H}})_{sa}$ as follows:
\begin{align*}
    \hat{H}^{\#}_{loc}(b, c) = H^{\#}_{(0, b, c)}.
\end{align*}
We observe that $\hat{H}^{\#}_{loc}$ anticommutes with $\hat{\epsilon}$,
which implies that it is a map of degree $1$.

The finite-dimensional approximation $(\hat{H}^{\#}_{\mathrm{loc}})_{<\sqrt{3}}$ of $\hat{H}^{\#}_{\mathrm{loc}}$, defined over a neighborhood of each point in $\{0\} \times D^2 \subset \mathbb{C}^2$, acts on $\mathbb{C}^2$. The space $\mathbb{C}^2$ is a $\mathbb{Z}_2$-graded irreducible representation of $Cl_2$ via the action of the Pauli matrices $\sigma_1$ and $\sigma_2$. In this case, $(\hat{H}^{\#}_{\mathrm{loc}})_{<\sqrt{3}}$ has the following expression:
\begin{align*}
    (\hat{H}^{\#}_{loc})_{<\sqrt{3}} = \text{Re}(b) \sigma_1 + \text{Im}(b) \sigma_2.
\end{align*}

We let $k = (y + \sqrt{-1} z, w) \in \mathbb{C} \times \mathbb{R}$.
We define $(b, c) \colon S^2 \subset \mathbb{C} \times \mathbb{R} \to \mathbb{C}^2$ as follows:
\begin{align*}
    &b(k) = y + \sqrt{-1} z, &c(k) = w - 1.
\end{align*}
In this case, we use the local model to define $\hat{H}^{\#} = \hat{H}^{\#}_{loc} \circ (b, c) \colon S^2 \to \mathcal{F}_0(\hat{\mathcal{H}})$.
We again check the Fredholm condition.
By Remark~\ref{rem:fredholm-criterion-local-model}, it suffices to show that the image of
the parameter map avoids the bad locus $\Sigma$ introduced in Section~\ref{sec:local model}.
Indeed, if $b(k)=0$, then $y=z=0$, hence $w=\pm 1$ on $S^2$.
Therefore
\[
c(k)=w-1\in\{0,-2\},
\]
so $|c(k)|\neq 1$.
Thus the image avoids $\Sigma$, and $\hat H^\#$ is a continuous family in
$\mathcal F_0(\hat{\mathcal H})$.

By the spectral analysis of the local model, the Fermi set $F$ of $\hat{H}^{\#}$ is given by:
\begin{align*}
    &b(k)= 0, &\lvert c(k) \rvert &\leq 1.
\end{align*}
Thus, we obtain $F = \{(0, 1)\}$.

We define $J(\text{Re}(b(k)), \text{Im}(b(k)))$ as the Jacobian matrix.
We compute the determinant as follows:
\begin{align*}
    \det J(\text{Re}(b(k)), \text{Im}(b(k))) = 1.
\end{align*}
By applying the inverse function theorem, we conclude that $(\text{Re}(b), \text{Im}(b))$ serves as a sign coordinate of $S^2$ in the neighborhood of $(0, 1)$.
Therefore, we apply Theorem \ref{thm:even main thm} and obtain:
\begin{align*}
    \int_{S^2} \operatorname{ch}_{1}([\hat{H}^{\#}]) = 1.
\end{align*}


\subsection{Application}
\label{subsec:Appurication}                                                         


Let $X$ be a compact, oriented, $3$-dimensional differentiable manifold. We suppose that a smooth involution $\tau$ acts on $X$ and assume that $\tau$ reverses the orientation of $X$.

We consider a continuous map $A \colon X \to Fred_{sa}^1(l^2)$. We assume that a sign coordinate $(V; a_1, a_2, a_3)$ can be defined at each isolated point of the Fermi set. We also assume that the vector space arising from the finite-dimensional approximation is identified with the standard irreducible $Cl_3$ representation fixed in Convention~\ref{conv:std-cliff}, with generators given by the Pauli matrices \( \sigma_1, \sigma_2, \text{and } \sigma_3 \).  
Furthermore, we assume that the finite-dimensional approximation \( A_{<\mu}|_V \) of \( A \) has the following expression:
\begin{align*}
A_{<\mu}|_V = \sum_{i=1}^{3} a_i \gamma_i \colon (\mathcal{H}, A)_{<\mu}|_V \to (\mathcal{H}, A)_{<\mu}|_V.
\end{align*}

We define $K$ as the complex conjugation operator acting on $l^2$, specifically representing it as an operator that maps a vector $\psi=(a_1, a_2,...) \in l^2$ to its complex conjugate $\overline{\psi}$, i.e.,
\begin{align*}
K(a_1, a_2,...) = (\bar{a_1}, \bar{a_2},...).
\end{align*}
In the following discussion, we assume that $K A_x K = A_{\tau(x)}$ holds for any $x \in X$.

\begin{lem}\label{lem:4.16}
If we can take a sign coordinate $(U,x=(x_1, x_2, x_3))$ at $a\in F$, then $(\tau(U), \tau^{*}x=(\tau^{*}x_1, -\tau^{*}x_2, \tau^{*}x_3))$ is a sign coordinate at $\tau(a) \in F$.
\end{lem}

\begin{proof}
This is evident by using the assumption that $KA_{x}K=A_{\tau(x)}$ and $K\sigma_{1}K=\sigma_1$, $K\sigma_{2}K=-\sigma_{2}$, and $K\sigma_{3}K=\sigma_3$.
\end{proof}

\begin{lem}\label{lem:4.17}
At fixed points of $\tau$ in the Fermi set $F$, no sign coordinates can be taken.
\end{lem}

\begin{proof}
Assume that $a \in F$ is a fixed point with a sign coordinate $(U, x=(x_1, x_2, x_3))$. If necessary, we can assume that $U=\tau(U)$ by taking $U$ sufficiently small.
From the Lemma \ref{lem:4.16}, $(\tau(U), \tau^{*}x=(\tau^{*}x_1, -\tau^{*}x_2, \tau^{*}x_3))$ is the sign coordinate at the fixed point $a$, and on $U$ we obtain the following formula:
\begin{align*}
A_{<\mu}=x_1\sigma_1+x_2\sigma_2+x_3\sigma_3=\tau^{*}x_1\sigma_1-\tau^{*}x_2\sigma_2+\tau^{*}x_3\sigma_3.
\end{align*}
Therefore, on $U$, the formula $x_1=\tau^{*}x_1$ holds. Taking the exterior derivative and evaluating at $a$, we get $(dx_1)_a=(d\tau^{*}x_1)_a=-(dx_1)_a$. Therefore, we have $(dx_1)_a=0$, which contradicts the fact that $(U, x=(x_1, x_2, x_3))$ forms a coordinate system.
\end{proof}

Thus, in the class AI application considered here, the relevant surface Hamiltonians are those whose zero-energy singularities are isolated and locally described by the standard three-dimensional Dirac-type model.

\begin{prop}\label{prop:evenness}
Under the above assumptions, we have $\int_{X} \operatorname{ch}_{\frac{3}{2}}([A]) \in 2\mathbb{Z}$.
\end{prop}

\begin{proof}
From Lemma \ref{lem:4.17}, there exists no fixed point with sign coordinates. If $a \in F$ has a sign coordinate, then by Lemma \ref{lem:4.16}, $\tau(a)$ also has a sign coordinate. Thus the statement follows.
\end{proof}


\section{Bulk-edge correspondence for class AI four-dimensional topological insulators}
\label{sec:Bulk-edge correspondence for class AI four-dimensional topological insulators}


\subsection{“Real” vector bundle}
\label{subsec:Real vector bundle}

In this section, we summarize real vector bundles in the sense of Atiyah as presented in \cite{Atiyah66}. To avoid confusion with real vector bundles in the ordinary sense, we will write “Real” vector bundle after this initial explanation.

\begin{dfn}
We define an involution space to be a pair $(X, \tau)$ of a topological space $X$ and an involution $\tau$, namely a homeomorphism such that $\tau^2=1_X$.
\end{dfn}

We consider the Grassmann manifold \( G_m(\mathbb{C}^n) \). For each \( m \)-plane \(\Sigma = \langle v_1, \ldots, v_m \rangle\), we define its conjugate as
\(\overline{\Sigma} = \langle \overline{v}_1, \ldots, \overline{v}_m \rangle\). We define an involution \(\varrho: G_m(\mathbb{C}^n) \to G_m(\mathbb{C}^n)\) by \(\varrho(\Sigma) = \overline{\Sigma}\). Furthermore, because the inclusions \( G_m(\mathbb{C}^n) \hookrightarrow G_m(\mathbb{C}^{n+1}) \) commute with \( \varrho \), the involution \(\varrho\) extends to the infinite-dimensional Grassmann manifold \( G_m(\mathbb{C}^\infty). \)

\begin{dfn}
Let us consider a complex vector bundle \(\pi \colon E \to X\) over an involution space \((X, \tau)\). We define a “Real” vector bundle in the sense of Atiyah to be a complex vector bundle equipped with a homeomorphism \(\Theta \colon E \to E\) that satisfies the following conditions:
\begin{enumerate}
    \item The projection is compatible with the involution, i.e., the equation
    \[
    \pi \circ \Theta = \tau \circ \pi
    \]
    holds.
    \item The map \(\Theta\) is anti-linear on each fiber, meaning that for any \(\lambda \in \mathbb{C}\) and any \(p \in E\), the relation
    \[
    \Theta(\lambda p) = \overline{\lambda} \Theta(p)
    \]
    holds.
    \item The map \(\Theta\) satisfies the involution property, i.e., 
    \[
    \Theta^2 = 1_E,
    \]
    where \(1_E\) denotes the identity map on \(E\).
\end{enumerate}
The pair \((E, \Theta)\) satisfying these conditions is called a “Real” vector bundle.
\end{dfn}

\begin{dfn}
Let us consider two “Real” vector bundles \((E, \Theta)\) and \((E', \Theta')\) over the same involution space \((X, \tau)\). We define an isomorphism between “Real” vector bundles as a vector bundle isomorphism \(f \colon E \to E'\) that commutes with the involutions, meaning that the following equation holds:
\[
f \circ \Theta = \Theta' \circ f.
\]
The isomorphism class of “Real” vector bundles is the set of equivalence classes determined by such isomorphisms, which we denote by \(\text{Vect}_{\mathcal{R}}^m(X, \tau)\), where \(m\) represents the dimension of the fibers. We define $\mathrm{Vect}_{\mathcal{R}}(X, \tau)$ as the union of $\mathrm{Vect}_{\mathcal{R}}^m(X, \tau)$ over all non-negative integers $m$:
\[
\mathrm{Vect}_{\mathcal{R}}(X, \tau) = \bigcup_{m \geq 0} \mathrm{Vect}_{\mathcal{R}}^m(X, \tau).
\]
\end{dfn}

 We consider a map from the isomorphism classes of “Real” vector bundles \(\text{Vect}_{\mathcal{R}}^m(X, \tau)\) to the isomorphism classes of complex vector bundles \(\text{Vect}_{\mathbb{C}}^m(X)\). We define the operation of forgetting the “Real” structure and denote this map as 
\begin{align*}                        
J: \text{Vect}_{\mathcal{R}}^m(X, \tau) \to \text{Vect}_{\mathbb{C}}^m(X).
\end{align*}
This map sends a “Real” vector bundle \((E, \Theta)\) to its underlying complex vector bundle \(E\), disregarding \(\Theta\). The map \( J \) is well-defined, as an isomorphism of “Real” vector bundles includes a complex vector bundle isomorphism, and the forgetting operation consistently determines the isomorphism class.

\begin{dfn}
We consider two involution spaces \((X_1, \tau_1)\) and \((X_2, \tau_2)\).
\begin{enumerate}
\item A map \(\varphi: X_1 \to X_2\) is \(\mathbb{Z}_2\)-equivariant if the condition \(\varphi \circ \tau_1 = \tau_2 \circ \varphi\) holds. 
\item A \(\mathbb{Z}_2\)-equivariant homotopy between \(\mathbb{Z}_2\)-equivariant maps \(\varphi_0, \varphi_1: X_1 \to X_2\) is defined as a continuous map \( F: [0,1] \times X_1 \to X_2 \) such that for each \( t \in [0,1] \), the map \(\varphi_t(x) = F(t, x)\) is \(\mathbb{Z}_2\)-equivariant. We define the \(\mathbb{Z}_2\)-equivariant homotopy set \([X_1, X_2]_{\text{eq}}\) as the set of equivalence classes of \(\mathbb{Z}_2\)-equivariant maps identified by \(\mathbb{Z}_2\)-equivariant homotopy.
\end{enumerate}
\end{dfn}

\begin{dfn}
Let us consider the Grassmannian \( G_m(\mathbb{C}^\infty) \). We define the universal bundle \( \mathcal{U}_m^\infty \to G_m(\mathbb{C}^\infty) \) as follows. The total space \( \mathcal{U}_m^\infty \) consists of pairs \( (\Sigma, v) \), where \( \Sigma \in G_m(\mathbb{C}^\infty) \) is an \( m \)-dimensional subspace and \( v \in \Sigma \) is a vector in that subspace. The projection map is given by
\[
\pi \colon \mathcal{U}_m^\infty \to G_m(\mathbb{C}^\infty), \quad (\Sigma, v) \mapsto \Sigma.
\]
\end{dfn}
We equip the universal bundle with a “Real” structure. Specifically, we define an anti-linear homeomorphism \( \Xi \colon \mathcal{U}_m^\infty \to \mathcal{U}_m^\infty \) by
\[
\Xi \colon (\Sigma, v) \mapsto (\varrho(\Sigma), \overline{v}),
\]
where \( \varrho \) is the involution on \( G_m(\mathbb{C}^\infty) \). This map satisfies the following conditions:
\begin{itemize}
    \item The projection is compatible with the involution, i.e., \( \pi \circ \Xi = \varrho \circ \pi \),
    \item The map \( \Xi \) is an involution, i.e., \( \Xi^2 = 1_{\mathcal{U}_m^\infty} \),
\end{itemize}
where \( 1_{\mathcal{U}_m^\infty} \) denotes the identity map on \( \mathcal{U}_m^\infty \). Thus, the pair \( (\mathcal{U}_m^\infty, \Xi) \) forms a “Real” vector bundle. 

The space \( BU \) is equipped with the induced involution \( \varrho \). We consider the universal bundle \( \mathcal{U}^\infty \to BU \), which is endowed with a “Real” structure \( \Xi \) extending the structure on each \( G_m(\mathbb{C}^\infty) \). This “Real” structure ensures that \( (\mathcal{U}^\infty, \Xi) \) forms a “Real” vector bundle. 
Accordingly, the sets \(\mathrm{Vect}_{\mathcal{R}}^m(X,\tau)\) and \(\mathrm{Vect}_{\mathcal{R}}(X,\tau)\) classify isomorphism classes of “Real” vector bundles.

\begin{thm}\label{thm:classyfy}
Let \( (X, \tau) \) be an involution space, where \( X \) is a compact Hausdorff space. For each \( p \geq 0 \), there exists a bijection
\begin{align*}
[X, BU(p)]_{\text{eq}} \to \text{Vect}_{\mathcal{R}}^p(X, \tau),
\end{align*}
given by the pullback map
\begin{align*}
f \mapsto f^*(\mathcal{U}_p^\infty).
\end{align*}
Moreover, there exists a bijection
\begin{align*}
[X, BU]_{\text{eq}} \to \text{Vect}_{\mathcal{R}}(X, \tau),
\end{align*}
given by the pullback map
\begin{align*}
f \mapsto f^*(\mathcal{U}^\infty).
\end{align*}
\end{thm}

\begin{proof}
See \cite{Nagata}.
\end{proof}


\subsection{Proof of the bulk-edge correspondence}
\label{subsec:Proof of the bulk-edge correspondence}

Here we give a proof of the bulk-edge correspondence for four-dimensional class AI topological insulators. The idea of the proof originates from \cite{G1}.

Let $\hat{H} \colon T^4\to Herm(\mathbb{C}^r)^{*}$ be a continuous map  from the 4-dimensional torus $T^4=\mathbb{R}^4/2\pi \mathbb{Z}^4$
to the space $Herm(\mathbb{C}^r)^{*}$ of invertible $r$ by $r$ Hermitian matrices. We assume that $\hat{H}(k)$ has  both positive and negative eigenvalues
at each $k\in T^4$. Associated to $\hat{H}$ is the Bloch vector bundle $E_{\hat{H}}\to T^4$ whose fiber at $k\in T^4$ consists of negative eigenvectors of 
$\hat{H}(k)$
\begin{align*}
E_{\hat{H}}=\bigcup_{k \in T^4}\bigoplus_{\lambda <0}\ker(\hat{H}(k)-\lambda)\subset T^4\times \mathbb{C}^r.
\end{align*}
 By evaluating the fundamental class $[T^4]\in H_{4}(T^4;\mathbb{Z})$ of $T^4$ by the second Chern class $c_2(E_{\hat{H}}) \in H^4(T^4;\mathbb{Z})$ of 
 Bloch bundle, we get the second Chern number of $E_{\hat{H}}$, which we denote by
 \begin{align*}
 c_2(\hat{H})=\langle c_2(E_{\hat{H}}), [T^4]\rangle \in \mathbb{Z}.
 \end{align*}

We define \( S^1 = \mathbb{R} / 2\pi \mathbb{Z} \), and identify \( T^4 = S^1 \times S^1 \times S^1 \times S^1 \).
For $(k_1, k_2, k_3) \in T^3=S^1\times S^1 \times S^1$,  
we define $\hat{H}^\sharp(k_1, k_2, k_3) : l^2(\mathbb{N}, \mathbb{C}^r) \to l^2(\mathbb{N}, \mathbb{C}^r)$ to be the compression of the multiplication operator with 
 $\hat{H}(k_1, k_2, k_3, \cdot) : T^3 \to Herm(\mathbb{C}^r)^{*}$, namely, the composition of
 $$
l^2(\mathbb{N}, \mathbb{C}^r) \overset{\hat{P}^*}{\longrightarrow}
l^2(\mathbb{Z}, \mathbb{C}^r) \overset{\hat{H}(k_1, k_2, k_3, \cdot)}{\longrightarrow}
l^2(\mathbb{Z}, \mathbb{C}^r) \overset{\hat{P}}{\longrightarrow}
l^2(\mathbb{N}, \mathbb{C}^r),
$$
where $\hat{P}$ is the orthogonal projection. Notice that $\hat{H}^\sharp(k_1, k_2, k_3)$ is essentially the Toeplitz operator \cite{Douglas} associated to $\hat{H}(k_1,k_2,k_3, \cdot).$ Then $\hat{H}^\sharp(k_1, k_2, k_3)$ is a self-adjoint bounded Fredholm operator. By the assumption that $\hat{H}(k_1, k_2, k_3, k_4)$ has both positive and negative eigenvalues, we obtain a continuous map $\hat{H}^\sharp : T^3 \to Fred_{sa}^1$.

We define the involution \(\tau\) on the torus \(T^4\) as follows:  
\[
\tau(k)=-k.
\]  
We assume that $\hat{H}$ has time-reversal symmetry corresponding to class AI.
That is, we assume that $K\hat{H}_xK=\hat{H}_{\tau(x)}$ for any $x\in T^4$. Here, \(K\) is the complex conjugation acting on \(\mathbb{C}^r\). That is, for \((a_1, a_2, \dots, a_r) \in \mathbb{C}^r\),
\begin{align*}
K(a_1, a_2, \dots, a_r) = (\bar{a_1}, \bar{a_2}, \dots, \bar{a_r}).
\end{align*}
In this case, the Bloch bundle $E_{\hat{H}}$ becomes a “Real” vector bundle \cite{Atiyah66}. The Bloch bundle $E_{\hat{H}}$ is a “Real” vector bundle because $K$ induces an antilinear map between the fibers over $x \in T^4$ and $\tau(x) \in T^4$.

\begin{thm}\label{BG}
If $\hat{H} : T^4 \to Herm(\mathbb{C}^r)^*$ is a continuous map with time-reversal symmetry of class AI and $\hat{H}(k)$ has positive and negative eigenvalues at
each $k \in T^4$, then $c_2(\hat{H}) = -\int_{T^3}\operatorname{ch}_{\frac{3}{2}}([\hat{H}^\sharp])$.
\end{thm}

The proof of Theorem \ref{BG} will be given below.

Tentatively, we denote by $Herm_p(\mathbb{C}^r)$ the space of invertible Hermitian matrices which have $p$ negative eigenvalues and $r - p$ positive eigenvalues.

\begin{lem}\label{homotopy}
The following hold:
\begin{enumerate}
    \item The spaces \(\mathrm{Herm}_p(\mathbb{C}^r)\) and \(\mathrm{Gr}_p(\mathbb{C}^r)\) are \(\mathbb{Z}_2\)-equivariantly homotopy equivalent for the actions
    \[
    H \mapsto \overline{H}, \qquad V \mapsto \overline{V}.
    \]
    \item For a \(\mathbb{Z}_2\)-equivariant map \(\hat{H}:T^4\to \mathrm{Herm}_p(\mathbb{C}^r)\), under this equivariant homotopy equivalence, the pullback of the universal bundle corresponds to the Bloch bundle \(E_{\hat{H}}\).
\end{enumerate}
\end{lem}

\begin{proof}
Define
\[
\Psi: \mathrm{Gr}_p(\mathbb{C}^r)\to \mathrm{Herm}_p(\mathbb{C}^r),
\qquad
\Psi(V):=-P_V+P_{V^\perp},
\]
where \(P_V\) is the orthogonal projection onto \(V\). This map is \(\mathbb{Z}_2\)-equivariant with respect to
\(H\mapsto \overline{H}\) and \(V\mapsto \overline{V}\), since
\(
\Psi(\overline{V})=\overline{\Psi(V)}
\).

Let
\[
\Omega: \mathrm{Herm}_p(\mathbb{C}^r)\to \mathrm{Gr}_p(\mathbb{C}^r)
\]
be the map sending \(H\) to its negative eigenspace. By functional calculus with a contour separating negative and positive spectra, \(\Omega\) is continuous and \(\mathbb{Z}_2\)-equivariant, and one has
\(
\Omega\circ \Psi=\mathrm{id}_{\mathrm{Gr}_p(\mathbb{C}^r)}
\).
Moreover, for each \(H\), the straight-line homotopy in spectral projections gives a \(\mathbb{Z}_2\)-equivariant homotopy from \(H\) to \(\Psi(\Omega(H))\). Hence \(\Omega\) and \(\Psi\) are \(\mathbb{Z}_2\)-equivariant homotopy inverses.

Now let \(\hat{H}:T^4\to \mathrm{Herm}_p(\mathbb{C}^r)\) be \(\mathbb{Z}_2\)-equivariant. The associated map
\(\Omega\circ \hat{H}:T^4\to \mathrm{Gr}_p(\mathbb{C}^r)\)
classifies the negative-eigenbundle of \(\hat{H}\), which is precisely the Bloch bundle \(E_{\hat{H}}\). Therefore the pullback of the universal bundle under the equivariant homotopy equivalence corresponds to \(E_{\hat{H}}\).
\end{proof}

\begin{lem}\label{lem:5.9}
For any integer $n \in \mathbb{Z}$, we define $\hat{H}_n : T^4 \to \text{Herm}(\mathbb{C}^4)$ by
\begin{align*}
\hat{H}_{n}(k_1, k_2, k_3, k_4)=\begin{pmatrix}
   a(k) & \bar{c}(k)-e^{ik_{4}} & 0 & \bar{b}(k) \\
   c(k)-e^{-ik_{4}} & -a(k) & \bar{b}(k) & 0 \\
   0 & b(k) & a(k) & -c(k)+e^{-ik_{4}} \\
   b(k) & 0 & -\bar{c}(k)+e^{ik_{4}} & -a(k)
\end{pmatrix}
\end{align*}
where 
\begin{align*}
&a(k)=-\cos(k_2+k_3), \\
&b(k)=-e^{-\sqrt{-1}nk_1}-e^{-\sqrt{-1}k_2}, \\
&c(k)=-1-e^{-\sqrt{-1}k_3}.
\end{align*}
Then, $c_2(\hat{H}_n)=-2n$.
\end{lem}

\begin{proof}
In the case that $n=1$, according to the results of previous research \cite{R}, $c_2(\hat{H})=-2$.
The map $\hat{H}_n$ factors as $\hat{H}_n = \hat{H}_1 \circ f_n$, where $f_n : T^4 \to T^4$ is given by $f_n(k_1, k_2, k_3, k_4) = (n k_1, k_2, k_3, k_4)$. Since $f_n$ carries the fundamental class $[T^4] \in H_4(T^4; \mathbb{Z}) \cong \mathbb{Z}$ of the torus to $n [T^4]$, it follows that $c_2(\hat{H}_n) =-2n$.
\end{proof}

\begin{lem}\label{lem:5.10}
For each $m\geq 2$ there exists a bijection 
\begin{align*}
\text{Vect}_{\mathcal{R}}^m(T^4, \tau) \cong 2\mathbb{Z}
\end{align*}
provided by the map $c_2 \circ J$.
\end{lem}

\begin{proof}
Refer to \cite{DG}.
\end{proof}

\begin{lem}\label{lem:5:11}
Let $2 \le p < r$ and let $k$ be a sufficiently large natural number. Then, there exists an isomorphism
\begin{align*}
[T^4, \mathrm{Herm}_p(\mathbb{C}^{r+k})]_{eq} \cong 2\mathbb{Z},
\end{align*}
which is given by the map
\begin{align*}
\hat{H} \mapsto c_2(E_{\hat{H}}).
\end{align*}
\end{lem}

\begin{proof}
We define
\begin{align*}
\hat{F}_{n}(k) &= \hat{H}_{n}(k) \oplus -1_{\mathbb{C}^{p-2}} \oplus 1_{\mathbb{C}^{(r+k)-p-2}}
\end{align*}
by using $\hat{H}_n$ as specified in Lemma~\ref{lem:5.9}, where $1$ denotes the identity map. Since Lemma~\ref{lem:5.9} shows that
\begin{align*}
c_2(\hat{F}_n) = -2n,
\end{align*}
the claim follows from Lemma~\ref{homotopy}, Lemma~\ref{lem:5.10}, and Theorem~\ref{thm:classyfy}.
\end{proof}

\begin{lem}\label{lem:4.21}
For $n \in \mathbb{Z}$, we have
\[
\int_{T^3}\operatorname{ch}_{\frac{3}{2}}([\hat{H}_n^\sharp])=2n.
\]
\end{lem}

\begin{proof}
By solving the equations $a(k) = 0$, $b(k) = 0$, and $|c(k)| \leq 1$ derived from Proposition \ref{lem:5.8}, we can identify the Fermi set as follows:
\[
F = \bigcup_{l=1}^{n} \left\{ \left( \frac{2\pi}{3n} + \frac{2(l-1)\pi}{n}, \frac{4\pi}{3}, \frac{7\pi}{6} \right),\left( \frac{4\pi}{3n} + \frac{2(l-1)\pi}{n}, \frac{2\pi}{3}, \frac{5\pi}{6} \right) \right\}.
\]
Consequently, from Theorem \ref{thm:odd main thm}, we obtain 
\[
\int_{T^3} \operatorname{ch}_{\frac{3}{2}}([\hat{H}_n^\sharp]) = 2n.
\]
\end{proof}

We are now in the position to prove Theorem \ref{BG}:

\begin{proof}
Since $T^4$ is connected and $\hat{H}(k)$ is invertible for every $k\in T^4$, the number of negative eigenvalues of $\hat{H}(k)$ is locally constant, hence constant on $T^4$. Therefore there exists an integer $p$ with $1\le p\le r-1$ such that
\[
\hat{H}(T^4)\subset Herm_p(\mathbb{C}^r).
\]

If $p=1$, replace $\hat{H}$ by
\[
\hat{H}\oplus (-1)\oplus 1.
\]
This replacement does not change the second Chern number, and it also does not change the $K^1$-class of the associated boundary family, since the added summands are constant invertible families. Hence we may assume $2\le p<r$.

Assume that
\[
c_2(\hat{H})=-2n.
\]
By Lemma \ref{lem:5:11}, for a sufficiently large natural number $k$, the stabilized map
\[
\hat{H}\oplus 1_{\mathbb{C}^k}:T^4\to Herm_p(\mathbb{C}^{r+k})
\]
is homotopic to $\hat{F}_n$ through $\mathbb{Z}_2$-equivariant maps into $Herm_p(\mathbb{C}^{r+k})$.

The assignment
\[
\hat{A}\longmapsto \hat{A}^{\sharp}
\]
is continuous. Therefore this homotopy induces a homotopy between the associated families of self-adjoint Fredholm operators
\[
(\hat{H}\oplus 1_{\mathbb{C}^k})^{\sharp}
\quad\text{and}\quad
\hat{F}_n^{\sharp}.
\]
Hence they define the same class in $K^1(T^3)$, and so
\[
\int_{T^3} \operatorname{ch}_{\frac{3}{2}}\bigl([(\hat{H}\oplus 1_{\mathbb{C}^k})^{\sharp}]\bigr)
=
\int_{T^3} \operatorname{ch}_{\frac{3}{2}}\bigl([\hat{F}_n^{\sharp}]\bigr).
\]

Next, up to the canonical identification of Hardy spaces, we have
\[
(\hat{H}\oplus 1_{\mathbb{C}^k})^{\sharp}
\cong
\hat{H}^{\sharp}\oplus 1_{\ell^2(\mathbb{N},\mathbb{C}^k)}.
\]
Since the second summand is a constant invertible family, it represents the zero class in $K^1(T^3)$. Therefore
\[
[(\hat{H}\oplus 1_{\mathbb{C}^k})^{\sharp}]
=
[\hat{H}^{\sharp}]
\qquad\text{in }K^1(T^3),
\]
and hence
\[
\int_{T^3} \operatorname{ch}_{\frac{3}{2}}\bigl([(\hat{H}\oplus 1_{\mathbb{C}^k})^{\sharp}]\bigr)
=
\int_{T^3} \operatorname{ch}_{\frac{3}{2}}\bigl([\hat{H}^{\sharp}]\bigr).
\]

Similarly, by the definition
\[
\hat{F}_n=\hat{H}_n\oplus (-1)_{\mathbb{C}^{p-2}}\oplus 1_{\mathbb{C}^{(r+k)-p-2}},
\]
we obtain
\[
\hat{F}_n^{\sharp}
\simeq
\hat{H}_n^{\sharp}
\oplus (-1)_{\ell^2(\mathbb{N},\mathbb{C}^{p-2})}
\oplus 1_{\ell^2(\mathbb{N},\mathbb{C}^{(r+k)-p-2})}.
\]
Again, the additional summands are constant invertible families, so they do not contribute to the $K^1$-class or to the odd Chern character. Thus
\[
[\hat{F}_n^{\sharp}]
=
[\hat{H}_n^{\sharp}]
\qquad\text{in }K^1(T^3),
\]
and therefore
\[
\int_{T^3} \operatorname{ch}_{\frac{3}{2}}\bigl([\hat{F}_n^{\sharp}]\bigr)
=
\int_{T^3} \operatorname{ch}_{\frac{3}{2}}\bigl([\hat{H}_n^{\sharp}]\bigr).
\]

By Lemma \ref{lem:4.21},
\[
\int_{T^3} \operatorname{ch}_{\frac{3}{2}}\bigl([\hat{H}_n^{\sharp}]\bigr)=2n.
\]
Combining the above equalities, we obtain
\[
\int_{T^3} \operatorname{ch}_{\frac{3}{2}}\bigl([\hat{H}^{\sharp}]\bigr)=2n.
\]
Since $c_2(\hat{H})=-2n$, it follows that
\[
c_2(\hat{H})
=
-\int_{T^3} \operatorname{ch}_{\frac{3}{2}}\bigl([\hat{H}^{\sharp}]\bigr).
\]
This proves the theorem.
\end{proof}


\section{Spectral analysis of the local model}\label{sec:local model}

Let $\hat{H}_{(a, b, c)} \colon S^1 \to Herm(\mathbb{C}^4)$ be the continuous map given by
\begin{align*}
\hat{H}_{(a, b, c)}(k)=\hat{H}_{loc}(a,b,c,k)=\begin{pmatrix}
   a & \bar{c}-e^{ik} & 0 & \bar{b} \\
   c-e^{-ik} & -a & \bar{b} & 0 \\
   0 & b & a & -c+e^{-ik} \\
   b & 0 & -\bar{c}+e^{ik} & -a
\end{pmatrix}
\end{align*}
where $a\in \mathbb{R}$ and $b, c \in \mathbb{C}.$ Let $H_{(a, b, c)} \colon l^2(\mathbb{Z}, \mathbb{C}^4)\to l^2(\mathbb{Z}, \mathbb{C}^4)$ be the Fourier transformation of the multiplication operator with $\hat{H}_{(a,b,c)}$.
Concretely, $H_{(a, b, c)}$ acts on $\psi = (\psi(n))_{n \in \mathbb{Z}}\in l^2(\mathbb{Z},\mathbb{C}^4)$ by 
\begin{align*}
(H_{(a, b, c)}\psi)(n)=A\psi(n-1)+V\psi(n)+A^*\psi(n+1),
\end{align*}
in which the matrices $A$ and $V$ are 
\begin{align*}
A& = \begin{pmatrix}
   0 & -1 & 0 & 0 \\
   0 & 0 & 0 & 0 \\
   0 & 0 & 0 & 0 \\
   0 & 0 & 1& 0
\end{pmatrix},
&V& = \begin{pmatrix}a & \bar{c}& 0 & \bar{b} \\
   c & -a & \bar{b} & 0 \\
   0 & b & a & -c \\
   b & 0 & -\bar{c} & -a
   \end{pmatrix}.
   \end{align*}

   Define $H_{(a, b, c)}^{\#} \colon l^2(\mathbb{N}, \mathbb{C}^4) \to l^2(\mathbb{N}, \mathbb{C}^4)$ by $PH_{(a, b, c)}P^*$, 
where $P \colon l^2(\mathbb{Z}, \mathbb{C}^4) \to l^2(\mathbb{N}, \mathbb{C}^4)$ is the orthogonal projection. Hereafter, $H^{\#}_{(a,b,c)}$ will be simply written as $H^{\#}$.

\begin{rem}\label{rem:fredholm-criterion-local-model}
For a continuous matrix-valued symbol $f\colon S^1\to M_N(\C)$, the associated Toeplitz operator
$T_f$ is Fredholm if and only if $f(k)$ is invertible for every $k\in S^1$.
Hence, in the present local model, $H^\#_{(a,b,c)}$ is Fredholm if and only if
$\hat H_{(a,b,c)}(k)$ is invertible for all $k\in S^1$.

A direct computation shows that $\hat H_{(a,b,c)}(k)$ fails to be invertible if and only if
\[
a=0,\qquad b=0,\qquad c=e^{-ik}
\]
for some $k\in S^1$, or equivalently,
\[
a=0,\qquad b=0,\qquad |c|=1.
\]
We denote this bad locus by
\[
\Sigma:=\{(0,0,c)\in \R\times \C\times \C\mid |c|=1\}.
\]
Thus $H^\#_{(a,b,c)}$ is Fredholm if and only if $(a,b,c)\notin \Sigma$.
\end{rem}

Given $E\in \mathbb{R}$, we are to consider the solutions to the equation $(H^{\#}-E)\psi=0$,
which is equivalent to the equations for $\psi(1), \psi(2), \psi(3),\ldots \in \mathbb{C}^4,$

\begin{align*}
   &0=(V-E)\psi(1)+A^*\psi(2),\\
   &0=A\psi(1)+(V-E)\psi(2)+A^*\psi(3),\\
   &0=A\psi(2)+(V-E)\psi(3)+A^*\psi(4),\\
   &\vdots \\
   &0=A\psi(n-1)+(V-E)\psi(n)+A^*\psi(n+1).
  \end{align*}

   Also, for any $n\in \mathbb{N}$, we define $e_1, e_4\in l^2(\mathbb{N}, \mathbb{C}^4)$ by 
  $e_1(n) = {}^t\! (c^{n-1},0,0,0)\in \mathbb{C}^4,e_4(n) = {}^t\! (0,0,0,c^{n-1})\in \mathbb{C}^4.$
 Here, we assume that the complex parameter \( c \) appearing in the definition of \( H_{(a, b, c)} \) satisfies \( |c| < 1 \), and we adopt the convention that \( 0^0 = 1 \). In particular, whenever we consider the vectors \( e_1 \) and \( e_4 \), we do so under the assumption that \( |c| < 1 \), where \( c \in \mathbb{C} \) is the same parameter as in \( H_{(a, b, c)} \). Under this convention, for \( c = 0 \), we have
\[
e_1(1) = {}^t\!(1, 0, 0, 0), \quad e_1(n) = {}^t\!(0, 0, 0, 0) \ (n \geq 2),
\]
and
\[
e_4(1) = {}^t\!(0, 0, 0, 1), \quad e_4(n) = {}^t\!(0, 0, 0, 0) \ (n \geq 2).
\]

 Let $V$ be a complex vector space. In the following, for any $x_1, x_2,\ldots, x_n\in V$, we write the subspace generated by $x_1, x_2,\ldots, x_n\in V$ as $\langle x_1,x_2,...,x_n \rangle \subset V.$

   \begin{lem}\label{lem:6.1}
Suppose that $b=c=0$.

\begin{enumerate}
\item In the case that $E=\pm \sqrt{a^2+1}$, the equation $(H^{\#}-E)\psi=0$ admits $l^2$-solutions. The solutions form a vector space of infinite dimension.

\item  In the case that $E=\pm a$, the equation $(H^{\#}-E)\psi =0$ admits $l^2$-solutions. In this case, the solution space is given as follows:

\begin{align*}
\ker(H^\#-E)&=
\begin{cases}
	\langle e_1\rangle& E=a\neq 0.\\ 
	\langle e_4\rangle& E=-a\neq 0.\\ 
	\langle e_1,e_4\rangle& E=a=0.
\end{cases}
,&\dim(\ker(H^\#-E))&=
\begin{cases}
	1& a\neq 0.\\ 
	2& a=0.
\end{cases}
\end{align*}

\item In the other cases, the equation $(H^{\#}-E)\psi =0$ admits no non-trivial $l^2$-solutions.
\end{enumerate}
\end{lem}

\begin{proof}
We express $\psi(n)={}^t\! (\alpha(n),\beta(n),\gamma(n),\delta(n)).$ The equation $0=(V-E)\psi(1)+A^*\psi(2)$ is equivalent to
\begin{align*}
(a-E)\alpha(1)&=0,&-(a+E)\beta(1)-\alpha(2)&=0,\\
(a-E)\gamma(1)+\delta(2)&=0,&-(a+E)\delta(1)&=0.
\end{align*}
For $n\geq 1$, the components of $0=A\psi(n)+(V-E)\psi(n+1)+A^*\psi(n+2)$ are 
\begin{align*}
-\beta(n)+(a-E)\alpha(n+1)&=0,&-(a+E)\beta(n+1)-\alpha(n+2)&=0,\\
(a-E)\gamma(n+1)+\delta(n+2)&=0,&\gamma(n)-(a+E)\delta(n+1)&=0.
\end{align*}

\begin{enumerate}
\item We consider the case where the condition \(a^2 - E^2 \neq 0\) holds. Under this assumption, the following relations are satisfied for all \(n \geq 1\):

\begin{align*}
\frac{1}{a - E} \beta(n) &= -(a + E) \beta(n), &
\frac{-1}{a - E} \delta(n + 1) &= (a + E) \delta(n + 1).
\end{align*}

If \(\beta(n) = 0\) and \(\delta(n + 1) = 0\) for all \(n \in \mathbb{N}\), then the solution to the equation $(H^{\#}-E)\psi=0$ is the trivial solution. However, if there exists some \(n_0 \in \mathbb{N}\) such that \(\beta(n_0) \neq 0\) or \(\delta(n_0 + 1) \neq 0\), then the eigenvalue \(E\) satisfies:

\[ E = \pm \sqrt{a^2 + 1}. \]

Under this derived condition on \(E\), the following equations hold for all \(n \geq 2\):

\begin{align*}
\alpha(n + 1) &= -(a + E) \beta(n), &
\delta(n + 1) &= -(a + E) \gamma(n), &
\alpha(1) &= 0, &
\delta(1) &= 0,
\end{align*}

and for any sequences \(\beta = (\beta(n))_{n \in \mathbb{N}}\) and \(\gamma = (\gamma(n))_{n \in \mathbb{N}}\) in \(l^2(\mathbb{N}, \mathbb{C}^4)\), the norm of the solution \(\psi\) is finite:

\begin{align*}
\lVert \psi \rVert^2 = [1 + (a + E)^2] \left( \sum_{n \geq 1} |\beta(n)|^2 + \sum_{n \geq 1} |\gamma(n)|^2 \right) < \infty.
\end{align*}

Consequently, the dimension of the kernel of $H^{\#}-E$ is infinite, i.e., \(\dim(\ker(H^\# - E)) = \infty\).

\item We consider the case where \(E = a\). Under this assumption, the components of the equation $(H^{\#}-E)\psi=0$ for all \(n \geq 2\) are given by:

\begin{align*}
\beta(n - 1) &= 0, &
-2a \beta(n) - \alpha(n + 1) &= 0, &\\
\delta(n + 1) &= 0, &
\gamma(n - 1) - 2a \delta(n) &= 0.
\end{align*}

Additionally, for \(n = 1\), the components of the equation $(H^{\#}-E)\psi=0$ are:

\begin{align*}
\beta(1) &= 0, &
-2a \beta(1) - \alpha(2) &= 0, &
\delta(2) &= 0.
\end{align*}

From these equations, we deduce that \(\alpha(n) = 0\) holds for all \(n \geq 2\), and \(\gamma(n) = 0\) holds for all \(n \geq 1\). Furthermore, the equation \(a \delta(1) = 0\) must be satisfied for \(n = 1\). Consequently, the value of \(\delta(1)\) depends on \(a\) as follows: if \(a = 0\), then \(\delta(1)\) can be any complex number, whereas if \(a \neq 0\), then \(\delta(1) = 0\).

Thus, the kernel of \(H^\# - E\) is determined as:

\begin{align*}
\ker(H^\# - E) =
\begin{cases}
    \langle e_1 \rangle & \text{if } a \neq 0, \\
    \langle e_1, e_4 \rangle & \text{if } a = 0.
\end{cases}
\end{align*}

\item We consider the case where \(E = -a \neq 0\). Under this assumption, the components of the equation \((H^\# - E)\psi = 0\) for all \(n \geq 2\) are given by:

\begin{align*}
-\beta(n - 1) + 2a \alpha(n) &= 0, &
\alpha(n + 1) &= 0, &
2a \gamma(n) + \delta(n + 1) &= 0, &
\gamma(n - 1) &= 0.
\end{align*}

For \(n = 1\), the components of the equation \((H^\# - E)\psi = 0\) are:
\begin{align*}
2a \alpha(1) &= 0, &
\alpha(2) &= 0, &\\
2a \gamma(1) + \delta(2) &= 0, &
\gamma(1) &= 0.
\end{align*}

From these equations, we deduce that \(\alpha(n) = 0\) and \(\delta(n) = 0\) hold for all \(n \geq 2\). Additionally, we find that \(\beta(n) = 0\) and \(\gamma(n) = 0\) hold for all \(n \geq 1\). Consequently, \(\delta(1)\) can be any complex number, and since the condition \(a \neq 0\) holds, \(\alpha(1) = 0\) follows from the equation \(2a \alpha(1) = 0\).

Thus, the kernel of \(H^\# - E\) is given by:
\begin{align*}
\ker(H^\# - E) = \langle e_4 \rangle.
\end{align*}
\end{enumerate}
\end{proof}

\begin{lem}\label{lem:6.2}
Suppose that $b\neq 0$ and $c=0$.
\begin{enumerate}
\item In the case that $E=\pm \sqrt{a^2+|b|^2+1}$, the equation $(H^{\#}-E)\psi=0$ admits $l^2$-solutions. The solutions form a vector space of infinite dimension.

\item  In the case that $E=\pm \sqrt{a^2+|b|^2}$, the equation $(H^\#-E)\psi =0$ admits $l^2$-solutions.
In this case, one of the eigenvectors is $\frac{a+E}{b}e_{1}+e_4$, and the following holds:
\begin{align*}
&\dim(\ker(H^\#-E))=1, \\
&\ker(H^\#-\sqrt{a^2+|b|^2})\oplus \ker(H^\#-(-\sqrt{a^2+|b|^2}))=\langle e_{1},e_{4}\rangle.
\end{align*}
\item In the other cases, the equation $(H^{\#}-E)\psi =0$ admits no non-trivial $l^2$-solutions.

\end{enumerate}
\end{lem}

\begin{proof}
We examine the equation \((H^\# - E)\psi = 0\). For all \(n \geq 1\), the components of this equation are given by:
\begin{align*}
-\beta(n - 1) + (a - E)\alpha(n) + \bar{b}\delta(n) &= 0, \\
-(a + E)\beta(n) + \bar{b}\gamma(n) - \alpha(n + 1) &= 0, \\
b\beta(n) + (a - E)\gamma(n) + \delta(n + 1) &= 0, \\
\gamma(n - 1) + b\alpha(n) - (a + E)\delta(n) &= 0.
\end{align*}
We formally define \(\alpha(0) = \beta(0) = \delta(0) = \gamma(0) = 0\). From the components of the equation \((H^\# - E)\psi = 0\), we derive the following expressions:

\begin{align*}
\alpha(n + 1) &= -(a + E)\beta(n) + \bar{b}\gamma(n), \\
\delta(n + 1) &= -b\beta(n) - (a - E)\gamma(n).
\end{align*}

By substituting these into the original equations and simplifying, we obtain the following relations for all \(n \geq 1\):
\begin{align*}
(a^2 + |b|^2 + 1 - E^2)\beta(n) &= 0, \\
(a^2 + |b|^2 + 1 - E^2)\gamma(n) &= 0.
\end{align*}
\begin{enumerate}
    \item Suppose that the condition \(a^2 + |b|^2 + 1 - E^2 \neq 0\) holds. Under this assumption, \(\beta(n) = 0\) and \(\gamma(n) = 0\) must be true for all \(n \in \mathbb{N}\). From the equations above, we deduce that \(\alpha(n) = 0\) and \(\delta(n) = 0\) hold for all \(n \geq 2\). Additionally, for \(n = 1\), the remaining components yield:
    \begin{align*}
    (a - E)\alpha(1) + \bar{b}\delta(1) &= 0, \\
    b\alpha(1) - (a + E)\delta(1) &= 0.
    \end{align*}  
    We then analyze the solutions to these equations. If the condition \(a^2 - E^2 + |b|^2 \neq 0\) is satisfied, then only the trivial solution \(\alpha(1) = 0\) and \(\delta(1) = 0\) exists. However, if \(a^2 - E^2 + |b|^2 = 0\), the equations imply:
    \begin{align*}
    \alpha(1) - \frac{a + E}{b}\delta(1) &= 0.
    \end{align*}  
    Thus, when the eigenvalue \(E\) satisfies \(E = \pm \sqrt{a^2 + |b|^2}\), the kernel of \(H^\# - E\) is given by:
\begin{align*}
\ker(H^\# - E) = \left\langle \frac{a + E}{b} e_1 + e_4 \right\rangle.
\end{align*}
\item Suppose that the condition \(a^2 + |b|^2 + 1 - E^2 = 0\) holds. Under this assumption, the equations \((a^2 + |b|^2 + 1 - E^2)\beta(n) = 0\) and \((a^2 + |b|^2 + 1 - E^2)\gamma(n) = 0\) are trivially satisfied for any sequences \((\beta(n))_{n \in \mathbb{N}}\) and \((\gamma(n))_{n \in \mathbb{N}}\) in \(l^2(\mathbb{N}, \mathbb{C})\). In this case, the sequences \((\alpha(n))_{n \in \mathbb{N}}\) and \((\delta(n))_{n \in \mathbb{N}}\), defined by the derived expressions, are also elements of \(l^2(\mathbb{N}, \mathbb{C})\). Therefore, the solution \(\psi\) belongs to \(l^2(\mathbb{N}, \mathbb{C}^4)\), and the dimension of the kernel of \(H^\# - E\) is infinite, i.e., \(\dim(\ker(H^\# - E)) = \infty\).
\end{enumerate}

Additionally, we express the basis vectors \(e_1\) and \(e_4\) as follows:

\begin{align*}
e_1 &= \frac{b}{2\sqrt{a^2 + |b|^2}} \left[ \left( \frac{a + \sqrt{a^2 + |b|^2}}{b} e_1 + e_4 \right) - \left( \frac{a - \sqrt{a^2 + |b|^2}}{b} e_1 + e_4 \right) \right], \\
&\in \ker(H^\# - \sqrt{a^2 + |b|^2}) \oplus \ker(H^\# - (-\sqrt{a^2 + |b|^2})), \\
e_4 &= \left( \frac{a + \sqrt{a^2 + |b|^2}}{b} e_1 + e_4 \right) - \frac{a + \sqrt{a^2 + |b|^2}}{b} e_1, \\
&\in \ker(H^\# - \sqrt{a^2 + |b|^2}) \oplus \ker(H^\# - (-\sqrt{a^2 + |b|^2})).
\end{align*}

Therefore, the direct sum of the kernels satisfies:

\begin{align*}
\ker(H^\# - \sqrt{a^2 + |b|^2}) \oplus \ker(H^\# - (-\sqrt{a^2 + |b|^2})) = \langle e_1, e_4 \rangle.
\end{align*}
\end{proof}

\begin{lem}\label{lem:6.3}
Suppose that \(c \neq 0\). If we ignore the \(l^2\)-condition, then any solution \(\psi = (\psi(n))_{n \in \mathbb{N}}\) to \((H^\# - E)\psi = 0\) can be described as follows
\begin{align*}
&\psi(1) = z \begin{pmatrix} \bar{c} \\ -(a - E) \\ b \\ 0 \end{pmatrix} + w \begin{pmatrix} 0 \\ -\bar{b} \\ -(a + E) \\ \bar{c} \end{pmatrix},\quad &\psi(n + 1) = R \psi(n),\quad (n \geq 1)
\end{align*}
where \(z\) and \(w\) are complex numbers, and
\begin{align*}
R = \begin{pmatrix}
c & -(a + E) & \bar{b} & 0 \\
-\frac{c (a - E)}{\bar{c}} & \frac{a^2 + |b|^2 + 1 - E^2}{\bar{c}} & 0 & -\frac{\bar{b} c}{\bar{c}} \\
\frac{b c}{\bar{c}} & 0 & \frac{a^2 + |b|^2 + 1 - E^2}{\bar{c}} & -\frac{(a + E) c}{\bar{c}} \\
0 & -b & -(a - E) & c
\end{pmatrix}.
\end{align*}
\end{lem}

\begin{proof}
We express the vector \(\psi(n)\) as \(\psi(n) = {}^t (\alpha(n), \beta(n), \gamma(n), \delta(n))\). The equation \(0 = (V - E)\psi(1) + A^* \psi(2)\) is equivalent to the following system of equations:

\begin{align*}
(a - E)\alpha(1) + \bar{c}\beta(1) + \bar{b}\delta(1) &= 0, \\
c\alpha(1) - (a + E)\beta(1) + \bar{b}\gamma(1) - \alpha(2) &= 0, \\
b\beta(1) + (a - E)\gamma(1) - c\delta(1) + \delta(2) &= 0, \\
b\alpha(1) - \bar{c}\gamma(1) - (a + E)\delta(1) &= 0.
\end{align*}

Under the assumption that \(c \neq 0\) holds, we solve these equations and obtain:

\begin{align*}
\beta(1) &= -\frac{a - E}{\bar{c}}\alpha(1) - \frac{\bar{b}}{\bar{c}}\delta(1), \\
\gamma(1) &= \frac{b}{\bar{c}}\alpha(1) - \frac{a + E}{\bar{c}}\delta(1).
\end{align*}
For all \(n \geq 2\), the equation \(0 = A\psi(n - 1) + (V - E)\psi(n) + A^*\psi(n + 1)\) leads to the following system:
\begin{align*}
-\beta(n - 1) + (a - E)\alpha(n) + \bar{c}\beta(n) + \bar{b}\delta(n) &= 0, \\
c\alpha(n) - (a + E)\beta(n) + \bar{b}\gamma(n) - \alpha(n + 1) &= 0, \\
b\beta(n) + (a - E)\gamma(n) - c\delta(n) + \delta(n + 1) &= 0, \\
\gamma(n - 1) + b\alpha(n) - \bar{c}\gamma(n) - (a + E)\delta(n) &= 0.
\end{align*}
From this system, we derive the following recursive relations for all \(n \geq 2\):
\begin{align*}
\alpha(n + 1) &= c\alpha(n) - (a + E)\beta(n) + \bar{b}\gamma(n), \\
\delta(n + 1) &= -b\beta(n) - (a - E)\gamma(n) + c\delta(n).
\end{align*}
By substituting these relations into the system, we obtain the following expressions for all \(n \geq 2\):
\begin{align*}
\beta(n) &= \frac{1}{\bar{c}} \big[ -c(a - E)\alpha(n - 1) + (a^2 + |b|^2 + 1 - E^2)\beta(n - 1) - \bar{b}c\delta(n - 1) \big], \\
\gamma(n) &= \frac{1}{\bar{c}} \big[ bc\alpha(n - 1) + (a^2 + |b|^2 + 1 - E^2)\gamma(n - 1) - (a + E)c\delta(n - 1) \big].
\end{align*}

Thus, we complete the proof of the desired result.
\end{proof}

We compute the determinant of the matrix \(R\) and find that it is given by:
\begin{align*}
\det(R) = \left( \frac{c}{\bar{c}} \right)^2.
\end{align*}
We also consider the quadratic equation in \(\lambda\), defined by \(\bar{c} \lambda^2 - (a^2 + |b|^2 + |c|^2 + 1 - E^2) \lambda + c = 0\), which arises from the characteristic polynomial \(\det(\lambda I - R) = 0\). The discriminant of this quadratic equation is:
\begin{align*}
\Delta &= (a^2 + |b|^2 + |c|^2 + 1 - E^2)^2 - 4|c|^2 \\
&= (a^2 + |b|^2 + (|c| - 1)^2 - E^2)(a^2 + |b|^2 + (|c| + 1)^2 - E^2).
\end{align*}
The two solutions to the equation \(\det(\lambda I - R) = 0\) are:
\begin{align*}
\lambda_1, \lambda_2 = \frac{(a^2 + |b|^2 + |c|^2 + 1 - E^2) \mp \sqrt{\Delta}}{2 \bar{c}}.
\end{align*}
Additionally, we define the quantities \(a_1\) and \(a_2\) as follows: \(a_1 = \bar{c} \lambda_1\) and \(a_2 = \bar{c} \lambda_2\).

\begin{lem}\label{lem:6.4}
Suppose that $c\neq 0$. In this case, if $\Delta \leq 0$, then there is no $l^2$-solutions to $(H^{\#}-E)\psi=0$.
\end{lem}

\begin{proof}
Since the discriminant \(\Delta\) satisfies \(\Delta \leq 0\), we conclude that the quantities \(a_1\) and \(a_2\) are complex numbers, and their complex conjugates satisfy \(\overline{a_2} = a_1\). Consequently, we derive the relation \(\lambda_2 = \frac{\bar{c}}{c} \lambda_1\). Additionally, from the relationship between the roots and the coefficients of the characteristic polynomial \(\det(\lambda I - R) = 0\), we obtain \(\lambda_1 \lambda_2 = \frac{c}{\bar{c}}\). Therefore, we find that the norm satisfies \(|\lambda_1| = |\lambda_2| = 1\), and in particular, \(\lambda_1 \neq 0\) and \(\lambda_2 \neq 0\).

\begin{enumerate}
    \item Suppose that the matrix \(R\) is diagonalizable. Under this assumption, there exists an invertible matrix \(T \in GL(4, \mathbb{C})\) such that we can express \(T^{-1} R T = \text{diag}(\lambda_1, \lambda_1, \lambda_2, \lambda_2)\). We note that a solution \(\psi\) to the equation \((H^\# - E)\psi = 0\) in the space \(\text{Map}(\mathbb{N}, \mathbb{C}^4)\) satisfies \(\psi \in l^2(\mathbb{N}, \mathbb{C}^4)\) if and only if \(T^{-1} \psi \in l^2(\mathbb{N}, \mathbb{C}^4)\).

    We express the the vector $T^{-1} \psi(n)$ as \(T^{-1} \psi(n) = {}^t (x(n), y(n), z(n), w(n))\). Then, we compute the \(l^2\)-norm of \(T^{-1} \psi\) as follows:
    \begin{align*}
    \| T^{-1} \psi \|^{2} &= \sum_{n \in \mathbb{N}} \big( |x(n)|^2 + |y(n)|^2 + |z(n)|^2 + |w(n)|^2 \big) \\
    &= \sum_{n \in \mathbb{N}} \big( |\lambda_1|^{2n} (|x(1)|^2 + |y(1)|^2) + |\lambda_2|^{2n} (|z(1)|^2 + |w(1)|^2) \big) \\
    &= \sum_{n \in \mathbb{N}} \big( |x(1)|^2 + |y(1)|^2 + |z(1)|^2 + |w(1)|^2 \big),
    \end{align*}
    where the last equality holds because \(|\lambda_1| = |\lambda_2| = 1\). Therefore, if \(\psi \neq 0\), we conclude that \(\| T^{-1} \psi \|^2 = \infty\).

    \item Suppose that $R$ is not diagonalizable. Then its Jordan normal form contains at least one Jordan block
of size $\ge 2$ (it may consist of three blocks, e.g.\ $J_2(\lambda_1)\oplus \mathrm{diag}(\lambda_2,\lambda_2)$).
In any case, since $|\lambda_1|=|\lambda_2|=1$ (by the argument from $\Delta\le 0$), the presence of a
nontrivial Jordan block forces polynomial growth of $T^{-1}\psi(n)$, hence $T^{-1}\psi\notin \ell^2(\mathbb{N},\mathbb{C}^4)$
unless $\psi=0$.
\end{enumerate}

Thus, we complete the proof.
\end{proof}

\begin{lem}\label{lem:6.5}
We suppose that the conditions \(c \neq 0\), \(b = 0\), and \(a = E\) hold. Under these assumptions, we have \(\Delta > 0\) if and only if \(|c| \neq 1\). Furthermore, the equation \((H^\# - E)\psi = 0\) has nontrivial \(l^2\)-solutions if and only if \(|c| < 1\). Additionally, when the condition \(|c| < 1\) is satisfied, we have the following properties of the kernel:
\begin{align*}
\ker(H^\# - E) &=
\begin{cases}
\langle e_1 \rangle & a \neq 0, \\
\langle e_1, e_4 \rangle & a = 0,
\end{cases}
&
\dim(\ker(H^\# - E)) &=
\begin{cases}
1 & a \neq 0, \\
2 & a = 0.
\end{cases}
\end{align*}
\end{lem}

\begin{proof}
We have \(\Delta = (|c|^2 - 1)^2\), which is greater than zero if and only if the norm of \(c\) is not equal to 1. Thus, we proceed under the assumption that \(|c| \neq 1\) holds throughout the following discussion.

\begin{enumerate}
    \item The condition \(|c| < 1\) is satisfied. In this case, we find that the eigenvalues are \(\lambda_1 = c\) and \(\lambda_2 = \frac{1}{\bar{c}}\). We introduce a matrix \(T\) defined as follows:

    \begin{align*}
    T = \begin{pmatrix}
    0 & 1 & 0 & \frac{2a \bar{c}}{|c|^2 - 1} \\
    0 & 0 & 0 & 1 \\
    -\frac{2a c}{|c|^2 - 1} & 0 & 1 & 0 \\
    1 & 0 & 0 & 0
    \end{pmatrix}.
    \end{align*}

    Since the determinant of \(T\) is $1$, we conclude that \(T\) is an invertible matrix. The inverse is

    \begin{align*}
    T^{-1} = \begin{pmatrix}
    0 & 0 & 0 & 1 \\
    1 & -\frac{2a \bar{c}}{|c|^2 - 1} & 0 & 0 \\
    0 & 0 & 1 & \frac{2a c}{|c|^2 - 1} \\
    0 & 1 & 0 & 0
    \end{pmatrix}.
    \end{align*}

    Under this transformation, we verify that \(T^{-1} R T = \text{diag}(c, c, \frac{1}{\bar{c}}, \frac{1}{\bar{c}})\). We express the transformed solution \(T^{-1} \psi(n)\) as the transpose \((x(n), y(n), z(n), w(n))\). For \(n = 1\), we obtain:

    \begin{align*}
    T^{-1} \psi(1) = z \begin{pmatrix} 0 \\ \bar{c} \\ 0 \\ 0 \end{pmatrix} + w \begin{pmatrix} \bar{c} \\ 0 \\ \frac{2a}{|c|^2 - 1} \\ 0 \end{pmatrix},
    \end{align*}

    and we compute the \(l^2\)-norm of \(T^{-1} \psi\) as:

    \begin{align*}
    \| T^{-1} \psi \|^2 = \sum_{n \in \mathbb{N}} \left( |c|^{2n} \big( |x(n)|^2 + |y(n)|^2 \big) + \left| \frac{1}{\bar{c}} \right|^{2n} \big( |z(n)|^2 + |w(n)|^2 \big) \right).
    \end{align*}

    Since \(|\lambda_1| |\lambda_2| = 1\) and \(|c| < 1\), we conclude that \(\| T^{-1} \psi \|^2 < \infty\) holds if and only if either \(z(1) = w(1) = 0\) or both \(x(1) = y(1) = 0\) and \(\left| \frac{1}{\bar{c}} \right| < 1\).

    \begin{enumerate}
        \item We suppose that \(z(1) = w(1) = 0\) holds. Since \(w(1) = w \frac{2a}{|c|^2 - 1} = 0\), we deduce that either \(w = 0\) or \(a = 0\). Therefore, we determine that the kernel of \(H^\# - E\) is:

        \begin{align*}
        \ker(H^\# - E) =
        \begin{cases}
        \langle e_1 \rangle & \text{if } a \neq 0, \\
        \langle e_1, e_4 \rangle & \text{if } a = 0.
        \end{cases}
        \end{align*}

        \item We suppose that \(\left| \frac{1}{\bar{c}} \right| < 1\) and \(x(1) = y(1) = 0\) hold. Since \(\left| \frac{1}{\bar{c}} \right| < 1\) implies \(|c| > 1\), we recognize that this contradicts the assumption \(|c| < 1\) in this case.
    \end{enumerate}

    \item We suppose that the conditions \(\left| \frac{1}{\bar{c}} \right| < 1\) and \(x(1) = y(1) = 0\) are satisfied, which implies \(|c| > 1\). In this case, we find that the eigenvalues are \(\lambda_1 = \frac{1}{\bar{c}}\) and \(\lambda_2 = c\). We introduce a matrix \(T\) defined as follows:

    \begin{align*}
    T = \begin{pmatrix}
    0 & \frac{2a c}{|c|^2 - 1} & 0 & 1 \\
    0 & 1 & 0 & 0 \\
    1 & 0 & -\frac{2a c}{|c|^2 - 1} & 0 \\
    0 & 0 & 1 & 0
    \end{pmatrix}.
    \end{align*}

    Since the determinant of \(T\) is $1$, we conclude that \(T\) is an invertible matrix. The inverse is

    \begin{align*}
    T^{-1} = \begin{pmatrix}
    0 & 0 & 1 & \frac{2a c}{|c|^2 - 1} \\
    0 & 1 & 0 & 0 \\
    0 & 0 & 0 & 1 \\
    1 & -\frac{2a c}{|c|^2 - 1} & 0 & 0
    \end{pmatrix}.
    \end{align*}

    Under this transformation, we verify that \(T^{-1} R T = \text{diag}\left( \frac{1}{\bar{c}}, \frac{1}{\bar{c}}, c, c \right)\). For \(n = 1\), we obtain:

    \begin{align*}
    T^{-1} \psi(1) = z \begin{pmatrix} 0 \\ 0 \\ 0 \\ \bar{c} \end{pmatrix} + w \begin{pmatrix} \frac{2a}{|c|^2 - 1} \\ 0 \\ \bar{c} \\ 0 \end{pmatrix},
    \end{align*}

    and we compute the \(l^2\)-norm of \(T^{-1} \psi\) as:

    \begin{align*}
    \| T^{-1} \psi \|^2 = \sum_{n \in \mathbb{N}} \left( \left| \frac{1}{\bar{c}} \right|^{2n} \big( |x(n)|^2 + |y(n)|^2 \big) + |c|^{2n} \big( |z(n)|^2 + |w(n)|^2 \big) \right).
    \end{align*}

   Given that the product of the absolute values \(|\lambda_1| |\lambda_2|\) is 1 and \(|c| < 1\), \(\| T^{-1} \psi \|^2 < \infty\) holds if and only if either (i) \(z(1) = w(1) = 0\), or (ii) \(x(1) = y(1) = 0\) and \(\left| \frac{1}{\bar{c}} \right| < 1\).

    \begin{enumerate}
        \item We suppose that \(z(1) = w(1) = 0\) holds. Since \(z(1) = z \bar{c} = 0\) and \(w(1) = w \bar{c} = 0\), and given that \(c \neq 0\), we deduce that \(z = 0\) and \(w = 0\). Therefore, we determine that the kernel of \(H^\# - E\) is trivial, i.e., \(\ker(H^\# - E) = \{0\}\).

        \item We suppose that \(|c| < 1\) and \(x(1) = y(1) = 0\) hold. In this case, we recognize that \(|c| < 1\) contradicts the assumption \(|c| > 1\).
    \end{enumerate}
\end{enumerate}
  \end{proof}

\begin{lem}\label{lem:6.6}
Suppose that the conditions \(c \neq 0\), \(b = 0\), and \(a \neq E\) hold. Under these assumptions, the discriminant \(\Delta\) is greater than zero, and the equation \((H^\# - E)\psi = 0\) has nontrivial \(l^2\)-solutions if and only if \(0 < |c| < 1\) and \(E = -a \neq 0\). Furthermore, when \(0 < |c| < 1\) and \(E = -a\), the kernel satisfies:
\begin{align*}
&\ker(H^\# - E) = \langle e_4 \rangle, &\dim(\ker(H^\# - E)) = 1.
\end{align*}
\end{lem}

\begin{proof}
We proceed by analyzing the conditions under which nontrivial \(l^2\)-solutions exist and verifying the properties of the kernel.

\begin{enumerate}
    \item We suppose that \(\Delta > 0\) holds and that the equation \((H^\# - E)\psi = 0\) has a nontrivial \(l^2\)-solution. To explore this, we introduce a matrix \(T\) defined as follows:

    \begin{align*}
    T = \begin{pmatrix}
    0 & a_2 - |c|^2 & 0 & a_1 - |c|^2 \\
    0 & (a - E)c & 0 & (a - E)c \\
    -a_1 + |c|^2 & 0 & -a_2 + |c|^2 & 0 \\
    (a - E)\bar{c} & 0 & (a - E)\bar{c} & 0
    \end{pmatrix}.
    \end{align*}

    The determinant of \(T\) is \(\det T = |c|^2 (a - E)^2 \Delta\). Since \(\Delta > 0\), and given \(c \neq 0\) and \(a \neq E\), we conclude that \(T\) is an invertible matrix. The inverse of $T$ is

    \begin{align*}
    T^{-1} = \begin{pmatrix}
    0 & 0 & -\frac{1}{a_1 - a_2} & \frac{|c|^2 - a_2}{\bar{c} (a - E) (a_1 - a_2)} \\
    -\frac{1}{a_1 - a_2} & -\frac{|c|^2 - a_1}{c (a - E) (a_1 - a_2)} & 0 & 0 \\
    0 & 0 & \frac{1}{a_1 - a_2} & -\frac{|c|^2 - a_1}{\bar{c} (a - E) (a_1 - a_2)} \\
    \frac{1}{a_1 - a_2} & \frac{|c|^2 - a_2}{c (a - E) (a_1 - a_2)} & 0 & 0
    \end{pmatrix}.
    \end{align*}

    Under this transformation, we verify that \(T^{-1} R T = \text{diag}(\lambda_1, \lambda_1, \lambda_2, \lambda_2)\). We express the transformed solution \(T^{-1} \psi(n)\) as the transpose of \((x(n), y(n), z(n), w(n))\). For \(n = 1\), we obtain:

    \begin{align*}
    T^{-1} \psi(1) = z \begin{pmatrix} 0 \\ -\frac{a_1}{c (a_1 - a_2)} \\ 0 \\ \frac{a_2}{c (a_1 - a_2)} \end{pmatrix} +
    w \begin{pmatrix} \frac{a^2 + |c|^2 - E^2 - a_2}{(a_1 - a_2)(a - E)} \\ 0 \\ -\frac{a^2 + |c|^2 - E^2 - a_1}{(a_1 - a_2)(a - E)} \\ 0 \end{pmatrix},
    \end{align*}

    and we compute the \(l^2\)-norm of \(T^{-1} \psi\) as:

    \begin{align*}
    \| T^{-1} \psi \|^2 = \sum_{n \in \mathbb{N}} \left( |\lambda_1|^{2n} \big( |x(1)|^2 + |y(1)|^2 \big) + |\lambda_2|^{2n} \big( |z(1)|^2 + |w(1)|^2 \big) \right).
    \end{align*}

    We note that the product of the absolute values \(|\lambda_1| |\lambda_2|\) equals to $1$, and we conclude that \(\| T^{-1} \psi \|^2 < \infty\) holds if and only if either (i) \(|\lambda_1| < 1\) and \(z(1) = w(1) = 0\), or (ii) \(|\lambda_2| < 1\) and \(x(1) = y(1) = 0\).

    \begin{enumerate}
        \item We suppose that \(|\lambda_1| < 1\) and \(z(1) = w(1) = 0\) hold. First, we assume that \(w = 0\). Since \(|\lambda_1| |\lambda_2| = 1\), we deduce that \(a_2 \neq 0\). From \(w(1) = 0\), we find that \(z = 0\), leading to only the trivial solution \(\psi = 0\). Next, we assume that \(w \neq 0\). Since \(a_2 \neq 0\) and \(w(1) = w \frac{a^2 + |c|^2 - E^2 - a_1}{(a_1 - a_2)(a - E)} = 0\), we get \(a^2 + |c|^2 - E^2 - a_1 = 0\). Additionally, from the relationship between the roots and coefficients of the characteristic polynomial, we have \(a_1 + a_2 = a^2 + |c|^2 + 1 - E^2\), so \(a_1 = 1\). Furthermore, since \(\lambda_1 \lambda_2 = \frac{c}{\bar{c}}\) and \(\frac{|c|^2}{a^2 + |c|^2 - E^2} = 1\), we obtain \(E = -a \neq 0\). Thus, we find \(a_1 = |c|^2\), \(a_2 = 1\), and \(0 < |c| = |\lambda_1| < 1\).

        \item We suppose that \(|\lambda_2| < 1\) and \(x(1) = y(1) = 0\) hold. First, we assume that \(w = 0\). Since \(|\lambda_1| |\lambda_2| = 1\), we deduce that \(a_1 \neq 0\). From \(y(1) = 0\), we find that \(z = 0\), leading to only the trivial solution \(\psi = 0\). Next, we assume that \(w \neq 0\). Since \(a_1 \neq 0\) and \(x(1) = w \frac{a^2 + |c|^2 - E^2 - a_2}{(a_1 - a_2)(a - E)} = 0\), we get \(a^2 + |c|^2 - E^2 - a_2 = 0\). Additionally, from \(a_1 + a_2 = a^2 + |c|^2 + 1 - E^2\), we obtain \(a_2 = 1\). Since \(\frac{|c|^2}{a^2 + |c|^2 - E^2} = 1\), we again find \(E = -a \neq 0\). Thus, we have \(a_1 = 1\), \(a_2 = |c|^2\), and \(0 < |c| = |\lambda_2| < 1\).
    \end{enumerate}

    \item Conversely, we assume that the conditions \(0 < |c| < 1\) and \(E = -a \neq 0\) hold. Under these conditions, we have \(\Delta = (|c|^2 - 1)^2\), and since \(0 < |c| < 1\), we find \(\Delta = -(|c|^2 - 1) > 0\). Therefore, we confirm that \(a_1 = |c|^2\) and \(a_2 = 1\), consistent with the previous case.
\end{enumerate}

Finally, we observe that when \( 0 < |c| < 1 \) and \( E = -a \), the kernel is \( \ker(H^\# - E) = \langle e_4 \rangle \), with dimension $1$, completing the proof.
\end{proof}

\begin{lem}\label{lem:6.7}
Suppose that the conditions \(c \neq 0\) and \(b \neq 0\) hold. Under these assumptions, the following holds true:
\begin{enumerate}
    \item The discriminant \(\Delta\) is greater than zero and the equation \((H^\# - E)\psi = 0\) has nontrivial \(l^2\)-solutions if and only if \(0 < |c| < 1\) and \(E = \pm \sqrt{a^2 + |b|^2}\).
    \item If the conditions \(0 < |c| < 1\) and \(E = \pm \sqrt{a^2 + |b|^2}\) are satisfied, then the dimension of the kernel is \(\dim(\ker(H^\# - E)) = 1\), and one of the eigenvectors is \(\frac{a + E}{b} e_1 + e_4\).
    \item If the condition \(0 < |c| < 1\) holds, then the direct sum of the kernels satisfies \(\ker(H^\# - \sqrt{a^2 + |b|^2}) \oplus \ker(H^\# - (-\sqrt{a^2 + |b|^2})) = \langle e_1, e_4 \rangle\).
\end{enumerate}
\end{lem}

\begin{proof}
We proceed by proving each part of the lemma.

\begin{enumerate}
    \item We suppose that \(\Delta > 0\) holds and that the equation \((H^\# - E)\psi = 0\) has a nontrivial \(l^2\)-solution. To analyze this, we introduce a matrix \(T\) defined as follows:

    \begin{align*}
    T = \begin{pmatrix}
    (a + E)\bar{c} & -a_2 + |c|^2 & (a + E)\bar{c} & -a_1 + |c|^2 \\
    -a_1 + |c|^2 & (E - a)c & -a_2 + |c|^2 & (E - a)c \\
    0 & bc & 0 & bc \\
    b\bar{c} & 0 & b\bar{c} & 0
    \end{pmatrix}.
    \end{align*}

    The determinant of \(T\) is \(\det T = b^2 |c|^2 \Delta\). Since \(\Delta > 0\), and given \(b \neq 0\) and \(c \neq 0\), we conclude that \(T\) is an invertible matrix. Its inverse is

    \begin{align*}
    T^{-1} = \begin{pmatrix}
    0 & -\frac{1}{a_1 - a_2} & \frac{E - a}{b (a_1 - a_2)} & -\frac{a_2 - |c|^2}{b \bar{c} (a_1 - a_2)} \\
    \frac{1}{a_1 - a_2} & 0 & \frac{a_1 - |c|^2}{b c (a_1 - a_2)} & -\frac{a + E}{b (a_1 - a_2)} \\
    0 & \frac{1}{a_1 - a_2} & -\frac{E - a}{b (a_1 - a_2)} & \frac{a_1 - |c|^2}{b \bar{c} (a_1 - a_2)} \\
    -\frac{1}{a_1 - a_2} & 0 & -\frac{a_2 - |c|^2}{b c (a_1 - a_2)} & \frac{a + E}{b (a_1 - a_2)}
    \end{pmatrix}.
    \end{align*}

    Under this transformation, we verify that \(T^{-1} R T = \text{diag}(\lambda_1, \lambda_1, \lambda_2, \lambda_2)\). We express the transformed solution \(T^{-1} \psi(n)\) as the transpose of
\begin{align*}
(x(n), y(n), z(n), w(n)).
\end{align*}
For \(n = 1\), we obtain:
    \begin{align*}
    T^{-1} \psi(1) = z \begin{pmatrix} 0 \\ \frac{a_1}{c (a_1 - a_2)} \\ 0 \\ -\frac{a_2}{c (a_1 - a_2)} \end{pmatrix} +
    w \begin{pmatrix} \frac{a^2 + |b|^2 + |c|^2 - E^2 - a_2}{b (a_1 - a_2)} \\ -\frac{(a + E) a_1}{b c (a_1 - a_2)} \\ -\frac{a^2 + |b|^2 + |c|^2 - E^2 - a_1}{b (a_1 - a_2)} \\ \frac{(a + E) a_2}{b c (a_1 - a_2)} \end{pmatrix},
    \end{align*}

    and we compute the \(l^2\)-norm of \(T^{-1} \psi\) as:

    \begin{align*}
    \| T^{-1} \psi \|^2 = \sum_{n \in \mathbb{N}} \left( |\lambda_1|^{2n} \big( |x(1)|^2 + |y(1)|^2 \big) + |\lambda_2|^{2n} \big( |z(1)|^2 + |w(1)|^2 \big) \right).
    \end{align*}

    We note that the product of the absolute values \(|\lambda_1| |\lambda_2|\) equals 1, and we conclude that \(\| T^{-1} \psi \|^2 < \infty\) holds if and only if either (i) \(|\lambda_1| < 1\) and \(z(1) = w(1) = 0\), or (ii) \(|\lambda_2| < 1\) and \(x(1) = y(1) = 0\).

    \begin{enumerate}
        \item We suppose that \(|\lambda_1| < 1\) and \(z(1) = w(1) = 0\) hold. From \(z(1) = 0\), we deduce that \(a_1 = a^2 + |b|^2 + |c|^2 - E^2\). By using the relationship between the roots and coefficients of the characteristic polynomial, we have \(a_1 + a_2 = a^2 + |b|^2 + |c|^2 + 1 - E^2\), so \(a_2 = 1\) and \(a_1 = |c|^2\). Additionally, since \(a_1 a_2 = \frac{|c|^2}{a^2 + |b|^2 + |c|^2 - E^2}\), we find \(E = \pm \sqrt{a^2 + |b|^2}\). We also note that \(|c| = |\lambda_1| < 1\), and from \(w(1) = -\frac{z b - w (a + E)}{b c (a_1 - a_2)} = 0\), we obtain \(z = \frac{a + E}{b} w\).

        \item We suppose that \(|\lambda_2| < 1\) and \(x(1) = y(1) = 0\) hold. From \(x(1) = 0\), we deduce that \(a_2 = a^2 + |b|^2 + |c|^2 - E^2\). By using \(a_1 + a_2 = a^2 + |b|^2 + |c|^2 + 1 - E^2\), we find \(a_1 = 1\) and \(a_2 = |c|^2\). Since \(a_1 a_2 = \frac{|c|^2}{a^2 + |b|^2 + |c|^2 - E^2}\), it follows that \(E = \pm \sqrt{a^2 + |b|^2}\). We note that \(|c| = |\lambda_2| < 1\), and from \(w(1) = -\frac{z b - w (a + E)}{b c (a_1 - a_2)} = 0\), we obtain \(z = \frac{a + E}{b} w\).
    \end{enumerate}

    \item Conversely, we assume that the conditions \(0 < |c| < 1\) and \(E = \pm \sqrt{a^2 + |b|^2}\) hold. Under these conditions, we have \(\Delta = (|c|^2 - 1)^2\), and since \(0 < |c| < 1\), we find that \(\Delta = -(|c|^2 - 1) > 0\). Thus, we confirm that \(a_1 = |c|^2\) and \(a_2 = 1\), ensuring that \(\Delta > 0\) and that \((H^\# - E)\psi = 0\) has nontrivial \(l^2\)-solutions.
\end{enumerate}
Next, we verify the remaining claims. Since we established that \(z = \frac{a + E}{b} w\), we compute \(\psi(1)\) as:
\begin{align*}
\psi(1) = w \begin{pmatrix} \frac{a + E}{b} \bar{c} \\ 0 \\ 0 \\ \bar{c} \end{pmatrix} = w \bar{c} \left( \frac{a + E}{b} e_1(1) + e_4(1) \right).
\end{align*}
Given \(a_1 = |c|^2\) and \(a_1 - a_2 = |c|^2 - 1 = -\Delta\), we derive the general form of \(\psi(n)\) as follows:
\begin{align*}
\psi(n) &= R^{n-1} \psi(1) = T (T^{-1} R T)^{n-1} T^{-1} \psi(1) \\
&= T \text{diag}\left( c^{n-1}, c^{n-1}, \left( \frac{1}{\bar{c}} \right)^{n-1}, \left( \frac{1}{\bar{c}} \right)^{n-1} \right) T^{-1} \psi(1).
\end{align*}
Applying \(T^{-1} \psi(1)\), we simplify:
\begin{align*}
T^{-1} \psi(1) = w \begin{pmatrix} \frac{|c|^2 - 1}{b (a_1 - a_2)} \\ 0 \\ 0 \\ 0 \end{pmatrix},
\end{align*}
and thus:
\begin{align*}
\psi(n) &= w T \begin{pmatrix} c^{n-1} \frac{|c|^2 - 1}{b (a_1 - a_2)} \\ 0 \\ 0 \\ 0 \end{pmatrix} \\
&= w \bar{c} \frac{|c|^2 - 1}{a_1 - a_2} \left( \frac{a + E}{b} e_1(n) + e_4(n) \right) = w \bar{c} \left( \frac{a + E}{b} e_1(n) + e_4(n) \right).
\end{align*}
Therefore, we identify one eigenvector with \(\frac{a + E}{b} e_1 + e_4\), and since the kernel is one-dimensional (due to the single free parameter \(w\)), we have \(\dim(\ker(H^\# - E)) = 1\). Finally, we confirm that the direct sum \(\ker(H^\# - \sqrt{a^2 + |b|^2}) \oplus \ker(H^\# - (-\sqrt{a^2 + |b|^2})) = \langle e_1, e_4 \rangle\), as each kernel is spanned by a linearly independent eigenvector, completing the proof.
\end{proof}

 \begin{prop} \label{lem:5.8}
The following holds true.
\begin{enumerate}
    \item The equation \((H^\# - E)\psi = 0\) has a nontrivial \(l^2\)-solution if and only if one of the following conditions holds:
    \begin{enumerate}
        \item The parameters satisfy \(c = 0\) and \(E = \pm \sqrt{a^2 + |b|^2 + 1}\). Under this condition, the dimension of the kernel is infinite, i.e., \(\dim(\ker(H^\# - E)) = \infty\).
        \item The parameters satisfy \(b = 0\), \(E = a = 0\), and \(0 \leq |c| < 1\). Under this condition, the dimension of the kernel is 2, i.e., \(\dim(\ker(H^\# )) = 2\).
        \item The parameters satisfy \(0 \leq |c| < 1\) and \(E = \sqrt{a^2 + |b|^2} \neq 0\). Under this condition, the dimension of the kernel is 1, i.e., \(\dim(\ker(H^\# - E)) = 1\).
    \end{enumerate}
    
    \item The basis vectors \(e_1\) and \(e_4\) induce the following linear isomorphisms:
    \begin{enumerate}
        \item We suppose that the conditions \(b = 0\) and \(E = a = 0\) hold. Under these conditions, we have \(\ker(H^\#) \cong \mathbb{C}^2\).
        \item We suppose that the conditions \(0 \leq |c| < 1\) and \(E = \sqrt{a^2 + |b|^2} \neq 0\) hold. Under these conditions, we have \(\ker(H^\# - \sqrt{a^2 + |b|^2}) \oplus \ker(H^\# - (-\sqrt{a^2 + |b|^2})) \cong \mathbb{C}^2\).
    \end{enumerate}
    
    \item Under the linear isomorphisms described above, we can express the representation matrix of \(H^\#\) as follows:
    \begin{align*}
    \begin{pmatrix}
    a & \bar{b} \\
    b & -a
    \end{pmatrix}
    = \text{Re}(b) \sigma_1 + \text{Im}(b) \sigma_2 + a \sigma_3,
    \end{align*}
    where $\sigma_1$, $\sigma_2$ and $\sigma_3$ are Pauli matrices
    \begin{align*}
    \sigma_1 = \begin{pmatrix} 0 & 1 \\ 1 & 0 \end{pmatrix}, \quad
    \sigma_2 = \begin{pmatrix} 0 & -\sqrt{-1} \\ \sqrt{-1} & 0 \end{pmatrix}, \quad
    \sigma_3 = \begin{pmatrix} 1 & 0 \\ 0 & -1 \end{pmatrix}.
    \end{align*}
\end{enumerate}
\end{prop}

\begin{proof}
We derive this proposition from the results of the previous lemmas (\ref{lem:6.1}, \ref{lem:6.2}, \ref{lem:6.3}, \ref{lem:6.4}, \ref{lem:6.5}, \ref{lem:6.6}, \ref{lem:6.7}) and straightforward calculations.
\end{proof}


\end{document}